\documentclass[reqno,10pt]{amsart}
\usepackage{amssymb,amsmath}
\usepackage{mathrsfs}
\usepackage{latexsym}
\usepackage{amscd}
\usepackage{color}  
\usepackage{tikz}
\usepackage{caption}
\usepackage{subcaption}
\usetikzlibrary{arrows,calc,positioning,through,intersections,backgrounds,matrix,fit,shapes,decorations}

\tikzset{
  vertex/.style={circle,color=black,draw=black,fill=black,inner sep=0cm,minimum size=.2cm},
  edge/.style={thick, black},
}

\let\cal=\mathcal

\newtheorem{theorem}{Theorem}[section]
\newtheorem{definition}[theorem]{Definition}

\newtheorem{prop}[theorem]{Proposition}
\newtheorem{fact}[theorem]{Fact}

\newtheorem{cor}[theorem]{Corollary}
\newtheorem{lemma}[theorem]{Lemma}
\newtheorem{claim}[theorem]{Claim}
\newtheorem{remark}[theorem]{Remark}

\newtheorem{obs}[theorem]{Observation}

\def\epsilon{\varepsilon}

\def\eps{\varepsilon}

\def\high{{\rm high}}

\def\good{{\rm good}} 
\def\bad{{\rm bad}} 
\def\typ{{\rm typ}} 
\def\spec{{\rm spec}}

\begin{document}

\title[Even Rainbow Cycles; Nontriangular Directed Cycles]{On Even Rainbow or Nontriangular 
Directed Cycles}  

\date{\today}

\author[A.~Czygrinow]{Andrzej Czygrinow}\thanks{The first author was partially supported
by Simons Foundation Grant $\#$521777.}    
\address{School of Mathematical and Statistical Sciences, 
Arizona State University, Tempe, AZ 85281, USA}
\email{aczygri@asu.edu}  

\author[T.~Molla]{Theodore Molla}\thanks{The second author was partially supported by 
NSF Grants DMS~1500121
and DMS~1800761.}  
\address{Department of Mathematics and Statistics, University of 
South Florida, Tampa, FL 33620, USA}
\email{molla@usf.edu}  

\author[B.~Nagle]{Brendan Nagle}\thanks{The third author was partially supported
by NSF Grant DMS~1700280.}  
\address{Department of Mathematics and Statistics, University of 
South Florida, Tampa, FL 33620, USA}  
\email{bnagle@usf.edu}  

\author[R.~Oursler]{Roy Oursler}
\address{School of Mathematical and Statistical Sciences, 
Arizona State University, Tempe, AZ 85281, USA}  
\email{Roy.Oursler@asu.edu}


\begin{abstract}
Let $G = (V, E)$ be an $n$-vertex edge-colored graph.  In~2013, H.~Li proved 
that 
if every vertex $v \in V$ is incident to at least $(n+1)/2$ distinctly colored edges, 
then $G$ admits a rainbow triangle.  
We establish a corresponding 
result 
for fixed even rainbow $\ell$-cycles $C_{\ell}$:  
if every vertex $v \in V$ 
is incident to at least 
$(n+5)/3$ distinctly colored edges, where $n \geq n_0(\ell)$ is sufficiently large, 
then $G$ admits an even rainbow $\ell$-cycle $C_{\ell}$.   
This result is best possible whenever $\ell \not\equiv 0$ (mod 3).   
Correspondingly, 
we also show that 
for a fixed (even or odd) integer $\ell \geq 4$, 
every large $n$-vertex 
oriented graph $\vec{G} = (V, \vec{E})$ with minimum outdegree at least $(n+1)/3$
admits a (consistently) directed
$\ell$-cycle $\vec{C}_{\ell}$.  
Our latter 
result
relates to one of Kelly, K\"uhn, and Osthus, who proved a similar statement 
for oriented graphs with large semi-degree.  
Our proofs are based on the stability method.  
\end{abstract}

\maketitle


\section{Introduction}

An {\it edge-colored graph} is a pair $(G, c)$, where $G = (V, E)$ is a graph and 
$c : E \to P$ is a function mapping edges to some palette of colors $P$.  
A subgraph $H \subseteq G$ is a {\it rainbow} subgraph if the edges of $H$ are distinctly
colored by $c$.   
We consider degree conditions ensuring the existence
of rainbow cycles $C_{\ell}$ in $(G, c)$ of fixed even length $\ell \geq 4$.   
To that end, 
a vertex $v \in V$  
in an edge-colored graph $(G, c)$ 
has {\it $c$-degree} $\deg_G^c(v)$ 
given by the number of distinct colors assigned by $c$ to the edges $\{v, w\} \in E$, where we   
set $\delta^c(G) = \min_{v \in V} \deg^c_G(v)$.
The following result of H.~Li~\cite{li2013rainbow}
motivates the main results of our paper. 
\begin{theorem}[H.~Li, 2013]
\label{thm:Li}  
Let $(G, c)$ be an $n$-vertex edge-colored graph.  If $\delta^c(G) \geq (n+1)/2$, then 
$(G,c)$ admits a rainbow 3-cycle $C_3$.  
\end{theorem}  
\noindent  A rainbow 
$K_{\lfloor n/2 \rfloor, \lceil n/2 \rceil}$ shows that  
Theorem~\ref{thm:Li} is best possible.  

Our first result 
ensures 
rainbow cycles of fixed even length.

\begin{theorem}
\label{thm:even}
There exists an absolute constant $\alpha > 0$ so that, 
for every even integer $\ell \geq 4$,  
every 
edge-colored graph
$(G, c)$ 
on $n \geq n_0(\ell)$ many vertices 
satisfying 
\begin{equation}
\label{eqn:degcond}
\delta^c(G) \geq 
\left\{
\begin{array}{cc}  
\big(\tfrac{1}{3} - \alpha\big) & \text{{\rm if $\ell \equiv 0$ (mod 3)}}, \\
\tfrac{n+5}{3}          & \text{{\rm if $\ell \not\equiv 0$ (mod 3)}}, 
\end{array}
\right.  
\end{equation}  
admits a rainbow $\ell$-cycle $C_{\ell}$.  
\end{theorem}  

\noindent  Theorem~\ref{thm:even} is best
possible for $\ell \not\equiv 0$ (mod 3), which we verify at the end of the Introduction.    
We prove 
Theorem~\ref{thm:even}
in Section~\ref{sec:proofevenor} using 
the stability method.

\begin{remark}
\rm 
In a related paper \cite{oddcyc}, we establish an analogue of Theorem~\ref{thm:even} for fixed odd integers
$\ell \geq 3$.  In particular, we show that for large integers $n \geq n_0(\ell)$, 
H.~Li's condition $\delta^c(G) \geq (n+1)/2$ ensures rainbow $\ell$-cycles
$C_{\ell}$ in $(G, c)$,   
which is again best possible by a rainbow 
$K_{\lfloor n/2 \rfloor, \lceil n/2 \rceil}$.  
\hfill $\Box$  
\end{remark}  

We also consider an analogue of Theorem~\ref{thm:even} 
for {\it oriented graphs} $\vec{G} = (V, \vec{E})$,
i.e., those for which 
$\vec{E} \subset V \times V$ satisfies the rule that $(u, v) \in \vec{E}$ forbids $(v, u) \in \vec{E}$.  
Here, we seek a {\it directed} or consistently oriented $\ell$-cycle $\vec{C}_{\ell}$,
whose vertices $V(\vec{C}_{\ell})$
may be ordered $(v_0, \dots, v_{\ell-1})$ so that $(v_i, v_{i+1}) \in \vec{E}$ for all $i \in 
\mathbb{Z}_{\ell}$.  
In this context, we may take $\ell \geq 4$ to be even or odd.  

\begin{theorem}
\label{thm:or}  
For every fixed integer $\ell \geq 4$, whether even or odd, every 
oriented graph $\vec{G} = (V, \vec{E})$
on $n \geq n_0(\ell)$ many vertices 
with minimum out-degree $\delta^+(\vec{G}) \geq (n+1)/3$ admits a directed $\ell$-cycle  
$\vec{C}_{\ell}$.  
\end{theorem}  

\noindent  We prove Theorem~\ref{thm:or} in Section~\ref{sec:proofevenor} using ideas similar to that
of Theorem~\ref{thm:even}.      
Note that Theorem~\ref{thm:or} is best possible for every $\ell \not\equiv 0$
(mod 3), as seen by the {\it blow-up} $\vec{G} = (V, \vec{E})$ of a directed triangle:  
$$
\text{let }  
V = V_0 \cup V_1 \cup V_2 
\text{ be a partition, and let }  
\vec{E} = (V_0 \times V_1) \cup (V_1 \times V_2) \cup (V_2 \times V_0),   
$$
where $|V_2| \leq |V_1| \leq |V_0|
\leq |V_2| + 1$.
Here, $\delta^+(\vec{G}) = |V_2|
\geq ((n+1)/3) - 1$.  

Note that Theorem~\ref{thm:or} omits the case $\ell = 3$,  
which is the triangular case
of the Caccetta-H\"aggkvist conjecture (cf.~\cite{caccetta1978minimal, 
hladky2017counting}) 
and is beyond the reach of our methods.  
We also mention that 
Theorem~\ref{thm:or} relates to the following 
result of Kelly, K\"uhn, and Osthus~\cite{kelly2010cycles}.  

\begin{theorem}[Kelly, K\"uhn, Osthus, 2010]  
\label{thm:KKO}  
For every integer $\ell \geq 4$ and for every integer $n \geq 10^{10}\ell$, every $n$-vertex
oriented graph $\vec{G} = (V, E)$ with $\delta_0(\vec{G}) 
= \min\{ \delta^+(\vec{G}), \delta^-(\vec{G})\} 
\geq (n+1)/3$ contains a directed
$\ell$-cycle $\vec{C}_{\ell}$.  Moreover, every vertex $v \in V$ belongs to a directed $\ell$-cycle  
$\vec{C}_{\ell}$.  
\end{theorem}

The remainder of this paper is organized as follows.  
In Section~\ref{sec:proofevenor}, we prove both Theorems~\ref{thm:even} and~\ref{thm:or}.
For these proofs, we need upcoming 
Lemmas~\ref{lem:non-ext} and~\ref{lem:ext}, which (in a sense made precise later) 
distinguish whether or not 
a given context is {\it extremal}.
We prove Lemma~\ref{lem:non-ext} 
in Sections~\ref{sec:non-ext}--\ref{sec:prooffacts}  
where we also prove 
supplemental results needed along the way.    
We prove Lemma~\ref{lem:ext} 
in Sections~\ref{sec:ext}--\ref{sec:ex3}, where 
again we prove supplemental results needed along the way.  
We conclude this Introduction by 
verifying the sharpness
of Theorem~\ref{thm:even} when $\ell \not\equiv 0$ (mod 3).

\subsection{Theorem~\ref{thm:even} is sharp for $\boldsymbol{\ell \not\equiv 0}$ (mod 3)} \label{sharpness}
Fix an integer $\ell \not\equiv 0$ (mod 3), 
which 
in the constructions below can be even or odd.  
Let $V = V_0 \cup V_1 \cup V_2$ be a partition of an $n$-element set $V$, 
where for optimality we take $\lfloor n/3 \rfloor = m = |V_2| \leq |V_1| \leq |V_0| \leq |V_2| + 1$.  
Let $G = (V, E)$ be given by the complete 3-partite graph $K[V_0, V_1, V_2]$.  
We now distinguish the cases $n, \ell$ (mod 3).  \\

\noindent  {\bf Case 1 ($n \not\equiv 2$ {\rm (mod 3)}).}  
Define $c_+: E \to V$
by setting, for each 
$i \in \mathbb{Z}_3$ and 
$(v_i, v_{i+1}) \in V_i \times V_{i+1}$, 
\begin{equation}
\label{eqn:4.6.2019.12:04p}  
c_+(\{v_i, v_{i+1}\}) = v_{i+1}.  
\end{equation}  
We say this same 
edge $e = \{v_i, v_{i+1}\} \in E$ is of {\it type $i$}, 
and we write $t(e) = i$ for its type.  
We write a 
fixed $\ell$-cycle $C_{\ell}$ in $G$ by a cyclic ordering $(e_0, e_1, \dots, e_{\ell-1})$
of its consecutive edges.  
A consecutive such pair $(e_k, e_{k+1})$ is a {\it 
reversal} 
when $e_k$ and $e_{k+1}$ are of the same type $t(e_k) = t(e_{k+1}) = i \in \mathbb{Z}_3$, where   
$(e_k, e_{k+1})$ is a {\it backward reversal} when $e_k \cap e_{k+1} \in V_{i+1}$, 
and $(e_k, e_{k+1})$ is a {\it forward reversal} when $e_k \cap e_{k+1} \in V_i$.  
Since $C_{\ell}$ is a cycle, the number of backward reversals is the number of forward reversals, 
and   
$C_{\ell}$ admits backward reversals lest $\ell \equiv 0$ (mod 3).  
Fix 
an arbitrary backward reversal $(e_k, e_{k+1})$ of $C_{\ell}$, where $k \in \mathbb{Z}_{\ell}$, where   
$t(e_k) = t(e_{k+1}) = i \in \mathbb{Z}_3$, and 
where
$e_k \cap e_{k+1} = \{v_{i+1}\} \subset V_{i+1}$.  
Then 
\begin{equation}
\label{eqn:5.22.2019.12:04p}  
c_+(e_k) 
\stackrel{(\ref{eqn:4.6.2019.12:04p})}{=}  
v_{i+1} 
\stackrel{(\ref{eqn:4.6.2019.12:04p})}{=}  
c_+(e_{k+1}),   
\end{equation}  
whence $C_{\ell}$ isn't rainbow.  
We observe from~(\ref{eqn:4.6.2019.12:04p})   
that 
$\deg_G^{c_+}(v_i) = 1 + |V_{i+1}|$ holds for each fixed $i \in \mathbb{Z}_3$ and for each fixed $v_i \in V_i$.  Indeed,  
an incident edge $e = \{v_i, v_j\} \in E$ is assigned the fixed color $c_+(e) = v_i$ when $v_j \in V_{i-1}$, 
and is assigned 
the variable color 
$c_+(e) = v_j$ among all $|V_{i+1}|$ many possible $v_j \in V_{i+1}$.
As such, 
$\delta^{c_+}(G) = \deg_G^{c_+}(v_1) = m + 1$ is achieved 
by any vertex $v_1 \in V_1$, while   
$\lceil (n+5)/3 \rceil = m+2$ 
is ensured by $n \not\equiv 2$ (mod 3).  \hfill $\Box$  \\

\noindent  {\bf Case 2 ($n \equiv 2$, $\ell \equiv 1$ {\rm (mod 3)}).}  
Here, $n\equiv 2$ (mod 3) ensures that 
$|V_0| = |V_1| = m+1$.  
Fix a perfect matching 
$M = \{ \{x_1, y_1\}, \dots, \{x_{m+1}, y_{m+1}\} \}$ of $G[V_0, V_1] = K[V_0, V_1]$, where 
$V_0 = \{x_1, \dots, x_{m+1}\}$ and $V_1 = \{y_1, \dots, y_{m+1}\}$, and fix a symbol
$\star \not\in V$.    
Define $c_M: E \to \{\star\} \cup V$ by 
\begin{equation}
\label{eqn:matchcolor}    
c_M(e) = 
\left\{ 
\begin{array}{cc}
c_+(e) & \text{if $e \in E \setminus E_G[V_0, V_1]$,}  \\
\star & \text{if $e \in M$,}  \\
x_b & \text{if $e = \{x_a, y_b\} \in E_G[V_0, V_1] \setminus M$.}  
\end{array}  
\right.  
\end{equation}  
We observe 
from~(\ref{eqn:matchcolor}) that     
$(G, c_M)$ is $(m+2)$-color-regular, while 
$\lceil (n+5)/3 \rceil = m+3$ is ensured by $n \equiv 2$ (mod 3).  
Indeed, 
as before in Case~1, a vertex $v_2 \in V_2$ has color-degree $\deg_G^{c_M}(v_2) = \deg_G^{c_+}(v_2) = 1 + |V_0| = m + 2$.  
Less easily, 
fix $x_a \in V_0$ and fix an incident edge $x_a \in e \in E$.  
If 
$e \cap V_2 \neq \emptyset$, then 
$e$ is assigned the fixed color 
$c_M(e) = c_+(e) = x_a$, and if $e = \{x_a, y_a\} \in M$, then $e$ is assigned the fixed color $c_M(e) = \star$.  
Otherwise, 
$e = \{x_a, y_b\} \in E[V_0, V_1] \setminus M$ for some $y_a \neq y_b \in V_1$, whence 
$e$ is assigned the variable color 
$c_M(e) = x_b$
among all 
$|V_1| - 1 = m$ many possible $y_b \in V_1 \setminus \{y_a\}$.  
Similarly, fix $y_b \in V_1$, and fix an incident edge $y_b \in e \in E$.  
If $e = \{x_b, y_b\} \in M$, then $e$ is assigned the fixed color $c_M(e) = \star$, 
and if $e = \{x_a, y_b\} \in E[V_0, V_1] \setminus M$ for some 
$x_b \neq x_a \in V_0$, then $e$ is assigned the fixed color $c_M(e) = x_b$.  
Otherwise, $e = \{y_b, v_2\} \in E[V_1, V_2]$
for some $v_2 \in V_2$, whence $e$ is assigned the variable color $c_M(e) = c_+(e)  = v_2$
among all $|V_2| = m$ many possible $v_2 \in V_2$.

We now observe  
that $(G, c_M)$ avoids rainbow $\ell$-cycles $C_{\ell}$.
For that, 
fix an $\ell$-cycle $C_{\ell} = (e_0, \dots, e_{\ell-1})$ of $G$ with backward reversal  
$(e_k, e_{k+1})$, where $k \in \mathbb{Z}_{\ell}$.  
For $C_{\ell}$ to be rainbow, 
we claim that $G$ must assume the color $\star$ within the backward reversal $(e_k, e_{k+1})$.  
Indeed, let 
$t(e_k) = t(e_{k+1}) = i \in \mathbb{Z}_3$, and let 
$e_k \cap e_{k+1} = \{v_{i+1}\} \subset V_{i+1}$.     
For $C_\ell$ to be rainbow, $i = 0$ 
is necessary 
lest~(\ref{eqn:5.22.2019.12:04p}) holds, so write   
$v_{i+1} = y_1 \in V_1$.  
Since $M$ is a matching, at most one of 
$e_k, e_{k+1} \in M$, but for $C_\ell$ to be rainbow, 
at least one such containment is necessary (as claimed) lest~(\ref{eqn:matchcolor})
gives 
$c_M(e_k) 
= 
x_1 
= 
c_M(e_{k+1})$.  
Now, for $C_{\ell}$ to be rainbow, 
the following are necessary:  
\begin{enumerate}
\item[$(a)$]  
$e_k \in M$ implies 
$(e_{k-1}, e_k)$ 
is a 
forward reversal, lest  
$$
c_M(e_{k-1})  
\stackrel{(\ref{eqn:matchcolor})}{=}      
c_+(e_{k-1}) 
\stackrel{(\ref{eqn:4.6.2019.12:04p})}{=}   
x_1
\stackrel{(\ref{eqn:matchcolor})}{=} 
c_M(e_{k+1});    
$$
\item[$(b)$]  
$e_{k+1} \in M$ implies 
$(e_{k+1}, e_{k+2})$ is a forward reversal, lest   
$$
c_M(e_{k+2})  
\stackrel{(\ref{eqn:matchcolor})}{=}      
c_+(e_{k+2}) 
\stackrel{(\ref{eqn:4.6.2019.12:04p})}{=}   
x_1
\stackrel{(\ref{eqn:matchcolor})}{=} 
c_M(e_k).     
$$
\end{enumerate}  
Either way, $C_{\ell}$ has further backward reversals 
(assuming $\star$ again) lest $\ell \equiv 2$ (mod 3).  \hfill $\Box$  \\

\noindent {\bf Case 3 ($n, \ell \equiv 2$ {\rm (mod 3)}).}  
We first slightly alter the graph $G = K[V_0, V_1, V_2]$ above, as follows.    
Fix $x \in V_0$ and $y \in V_1$ so that $U_0 = V_0 \setminus \{x\}$, $U_1 = V_1 \setminus \{y\}$, and $U_2 = V_2$ all have size $m$.    
Define $\hat{E}$ by the rule that, for each $\{u, v\} \in \binom{V}{2}$, we put $\{u, v\} \in \hat{E}$ if, and only if, 
$$
\{y\} \times U_0 
\not\ni 
(u, v) \not\in \bigcup_{i \in \mathbb{Z}_3} (U_i \times U_i).   
$$
In other words, $G = (V, E)$ and $\hat{G} = (V, \hat{E})$ differ only in the $3m = n - 2$ elements 
among 
$$
\hat{E} \setminus  E = 
\bigcup \big\{ \{x, u_0\}: \, u_0 \in U_0\big\}
\cup 
\bigcup \big\{ \{y, u_1\}: \, u_1 \in U_1\big\}
\quad \text{and} \quad 
E \setminus \hat{E} = 
\bigcup \big\{ \{y, v_0\}: \, v_0 \in U_0\big\}.  
$$
Define $\hat{c} : \hat{E} \to \{\star\} \cup V$ by setting, for each $e \in \hat{E}$,  
\begin{equation}
\label{eqn:hardcol}  
\hat{c}(e) = 
\left\{
\begin{array}{cc}
\star & \text{if $e \in \hat{E} \setminus E$,} \\
\star & \text{if $e = \{x, u_2\} \in \hat{E} \cap E$ for some $u_2 \in U_2$,}  \\
c_+(e) & \text{otherwise.}
\end{array}
\right.
\end{equation}  
We observe from~(\ref{eqn:hardcol}) that $(\hat{G}, \hat{c})$ is $(m+2)$-color-regular, while
$\lceil (n+5)/3) \rceil = m + 3$ is ensured by $n \equiv 2$ (mod 3).  Indeed, fix a vertex $u_0 \in U_0$, 
and fix an incident edge $u_0 \in e \in \hat{E}$.  
If $x \in e$, then $e$ is assigned the fixed color
$\hat{c}(e) = \star$, and if $e \cap U_2 \neq \emptyset$, then $e$ is assigned the fixed color $\hat{c}(e) = c_+(e) = u_0$.  
Otherwise, $e = \{u_0, u_1\} \in \hat{E}[U_0, U_1] = E[U_0, U_1]$ for some $u_1 \in U_1$, whence $e$ is assigned
the variable color $\hat{c}(e) = c_+(e) = u_1$ among all $|U_1| = m$ many possible $u_1 \in U_1$.  
Vertices $u_1 \in U_1$ and $u_2 \in U_2$ similarly have
$\hat{c}$-degree 
$m + 2$.  
For the fixed vertex $x \in V$, fix an incident edge $x \in e \in \hat{E}$.  If $e \cap (U_0 \cup U_2) \neq \emptyset$,
then $e$ is assigned the fixed color $\hat{c}(e) = \star$, and if $y \in e$, then $e$ is assigned the fixed color
$\hat{c}(e) = y$.  Otherwise, $e = \{x, u_1\}$ for some $u_1 \in U_1$, whence $e$ is assigned the variable color
$\hat{c}(e) = c_+(e) = u_1$ among all $|U_1| = m$ many possible $u_1 \in U_1$.  The fixed vertex
$y \in V$
similarly has $\hat{c}$-degree $m+2$.  
That 
$(\hat{G}, \hat{c})$ avoids rainbow $\ell$-cycles $C_{\ell}$ 
is 
sketched in the Appendix, when more needed concepts are developed.  
\hfill $\Box$

\section{Proofs of Theorems~\ref{thm:even} and~\ref{thm:or}}  
\label{sec:proofevenor}  
The proofs of Theorems~\ref{thm:even} and~\ref{thm:or} 
are based on the well-known stability method, together with a few elementary
results.  We present the tools we need in order of increasing technicality.    

\subsection{Elementary tools}  
\label{subsec:eletools}  
Edge-colored graphs $(G, c)$ on a vertex set $V$ 
correspond to directed graphs $\vec{G} = (V, \vec{E})$, as follows.  
For each $v \in V$, 
let $\{v, w_1\}, \dots, \{v, w_d\} \in E$ 
be a system of representatives
of the color classes of $c$ on edges at $v$, 
where $d = \deg_G^c(v)$.  
We  
put $(v, w_1), \dots, (v, w_d) \in \vec{E}$, and     
we 
say that a directed graph $\vec{G} = (V, \vec{E})$
obtained in this way (which need be neither oriented nor unique)    
is {\it associated with} $(G, c)$.   
Directed graphs $\vec{G} = (V, \vec{E})$ correspond to edge-colored graphs $(G, c)$, as 
follows.
For each $(v, w) \in \vec{E}$, we put $\{v, w\} \in E(G)$ and define $c(\{v, w\}) = w$.  
Then $(G, c)$ is uniquely determined by $\vec{G}$, although $G = (V, E)$ may be a multigraph.    
We pause for the following remark.  

\begin{remark}  
\rm 
\label{rem:2.25.2019.1:24p}  
In this paper, no directed graph $\vec{G} = (V, \vec{E})$ 
will allow $\vec{E}$ to be a multiset, nor will $\vec{E}$ consist of any directed loops.  When    
$(v, w) \in \vec{E}$ forbids $(w, v) \in \vec{E}$, 
then $\vec{G} = (V, \vec{E})$
is an {\it oriented graph}.  When so, the edge-colored graph $(G, c)$ determined by 
$\vec{G}$ 
is simple.  \hfill $\Box$  
\end{remark}  

In the contexts above, we make a couple of elementary observations.  
On the one hand, 
if $(G, c)$ is an edge-colored graph and
$\vec{G} = (V, \vec{E})$ is a directed
graph associated
with $(G, c)$, then   
every vertex $v \in V$ has out-degree
$\deg^+_{\vec{G}}(v) = \deg^c_G(v)$.  On the other hand, 
if $\vec{G} = (V, \vec{E})$ is an oriented graph and 
$(G, c)$ is the edge-colored graph 
determined 
by $\vec{G} = (V, \vec{E})$, then every vertex $v \in V$ satisfies 
$\deg_G^c(v) = \deg^+_{\vec{G}}(v) + 1$ when $v$ has positive in-degree in $\vec{G}$, 
and 
$\deg_G^c(v) = \deg^+_{\vec{G}}(v)$
otherwise.  
In these contexts, 
we next 
consider the extent to which rainbow cycles 
of $(G, c)$ 
relate to directed cycles of $\vec{G}$, and vice versa.  We begin 
with 
the following elementary but useful observation first noted by H.~Li in~\cite{li2013rainbow}.  

\begin{fact}
\label{fact:Li} 
Let $\vec{G} = (V, \vec{E})$ be an oriented graph, and let $(G, c)$
be the edge-colored graph determined by $\vec{G}$.  
Every directed $\ell$-cycle $\vec{C}_{\ell}$
in $\vec{G}$
corresponds to a rainbow $\ell$-cycle $C_{\ell}$ in $(G, c)$.  Moreover, 
every properly colored $\ell$-cycle 
$C_{\ell}$ in $(G, c)$ 
is, in fact, a rainbow $\ell$-cycle, and 
corresponds to a directed $\ell$-cycle $\vec{C}_{\ell}$ in $\vec{G}$.  
\end{fact}

In Fact~\ref{fact:Li}, the 
edge-colored graph $(G, c)$ is derived from a given 
oriented graph $\vec{G} = (V, \vec{E})$, and 
directed $\ell$-cycles $\vec{C}_{\ell}$ of $\vec{G}$ are in one-to-one correspondence
with rainbow $\ell$-cycles $C_{\ell}$ of $(G, c)$.  
However, 
when $(G, c)$ is given and $\vec{G} = (V, \vec{E})$ is associated with $(G, c)$, 
the same conclusion 
need not hold.  

\begin{fact}
\label{prop:easyprop}  
Let $(G, c)$ be an $n$-vertex edge-colored graph, and 
let $\vec{G} = (V, \vec{E})$ be a directed graph associated with $(G, c)$.  
Then $\vec{G}$ admits at most $n^{\ell - 1}$ many directed $\ell$-cycles
$\vec{C}_{\ell}$ that 
were not rainbow in $(G, c)$.  
\end{fact}  

\begin{proof}[Proof of Fact~\ref{prop:easyprop}]  
Let ${\cal A}_{\ell}$ 
(${\cal N}_{\ell}$) denote the family of all 
(non-rainbow) 
directed $\ell$-cycles
$(v_0, v_1, \dots, v_{\ell-1})$
in $\vec{G}$, written here as cyclic permutations.  Then 
\begin{equation}
\label{eqn:10.4.2019.2:26p}  
\ell 
|{\cal A}_{\ell}| =
\sum_{v_0 \in V}
\, 
\sum_{v_1 \in N_{\vec{G}}^+(v_0)} \, \dots 
\, 
\sum_{v_{\ell-2} \in N_{\vec{G}}^+(v_{\ell-3})} 
\, 
\sum_{v_{\ell-1} \in N_{\vec{G}}^+(v_{\ell-2}) \cap N_{\vec{G}}^-(v_0)} 
\, 
1.  
\end{equation}  
Each $(v_0, v_1, \dots, v_{\ell-1}) \in {\cal N}_{\ell}$ identifies 
a sum
in~\eqref{eqn:10.4.2019.2:26p}  
with at most $\ell$-terms, so $\ell |{\cal N}_{\ell}| \leq \ell n^{\ell-1}$.  
\end{proof}

The following concept is central throughout the remainder of the paper.  

\begin{definition}[$\lambda$-extremal]  
\label{def:lambdaext}  
\rm 
Fix $\lambda \geq 0$, an $n$-vertex directed graph 
$\vec{G} = (V, \vec{E})$, and an edge-colored graph $(G, c)$ with vertex set $V$ and edge set $E$.    
We say that 
\begin{enumerate}  
\item  
$\vec{G}$ is {\it $\lambda$-extremal} if there exists a partition $V = V_0 \cup V_1 \cup V_2$
where 
for all $i \in \mathbb{Z}_3$, 
\begin{equation}
\label{eqn:dirlambdaext}
e_{\vec{G}}(V_i, V_{i+1}) \geq \left(\tfrac{1}{9} - \lambda\right) n^2, 
\end{equation}  
where $e_{\vec{G}}(V_i, V_{i+1})$ denotes 
the number of edges $(v_i, v_{i+1}) \in \vec{E} \cap (V_i \times V_{i+1})$; 
\smallskip 
\item
$(G, c)$ is {\it $\lambda$-extremal} 
if there exists 
a partition $V = V_0 \cup V_1 \cup V_2$ 
on which some directed graph $\vec{G} = (V, \vec{E})$ associated with $(G, c)$
is $\lambda$-extremal.  
\end{enumerate}  
In these contexts, 
$V = V_0 \cup V_1 \cup V_2$ 
is said to be {\it $\lambda$-extremal} for $\vec{G}$ or~$(G, c)$.  
\end{definition}  

We conclude our elementary tools with the following fact.  

\begin{fact}
\label{fact:10}  
For all $0 \leq \lambda \leq 1/(28)$, and 
for every 
positive integer $\ell \equiv 0$ {\rm (mod 3)}, the following hold:    
\begin{enumerate}
\item 
Every $\lambda$-extremal $n$-vertex directed graph 
$\vec{G} = (V, \vec{E})$ has 
$\Omega(n^{\ell})$ 
many 
directed $\ell$-cycles $\vec{C}_{\ell}$.
\smallskip 
\item 
Every $\lambda$-extremal $n$-vertex edge-colored graph 
$(G, c)$ has 
$\Omega(n^{\ell})$  
many 
rainbow $\ell$-cycles $C_{\ell}$.  
\end{enumerate}  
\end{fact}  

\begin{proof}[Proof of Fact~\ref{fact:10}]  
Fix $0 \leq \lambda \leq 1/(28)$ and fix a positive integer $\ell \equiv 0$ (mod 3).  
To prove Statement~(1), set $k = \ell / 3$,   
and let $\vec{G} = (V, \vec{E})$ be an $n$-vertex directed graph with 
$\lambda$-extremal vertex partition $V = V_0 \cup V_1 \cup V_2$.  
Let $\vec{H}$ be a blow-up of the directed triangle on $V = V_0 \cup V_1 \cup V_2$, whose edges consist of 
$(V_0 \times V_1) \cup (V_1 \times V_2) \cup (V_2 \times V_0)$.  
Then, $\vec{H}$ admits precisely 
$(|V_0|)_k \times (|V_1|)_k \times (|V_2|)_k$ 
many directed $\ell$-cycles 
$\vec{C}_{\ell}$ meeting each of $V_0$, $V_1$, and $V_2$ exactly $k$ times.  
The number of these cycles having some edge $\vec{e} = (v_0, v_1)$ of $\vec{H} \setminus \vec{G}$, where $v_0 \in V_0$ and $v_1 \in V_1$, 
is at most 
$$
\left(
|V_0||V_1| - \left(\tfrac{1}{9} - \gamma\right)n^2 \right)|V_2| \times (|V_0| - 1)_{k-1} \times (|V_1|-1)_{k-1} (|V_2| - 1)_{k-1}.   
$$
More generally, the number of these cycles having some edge $\vec{e}$ of $\vec{H} \setminus \vec{G}$ is at most 
$$
\left(3 - \left(\tfrac{1}{9} - \gamma\right) \frac{n^3}{|V_0||V_1| |V_2|} \right) 
(|V_0|)_k \times (|V_1|)_k \times (|V_2|)_k.   
$$
Thus, $\vec{G}$ admits at least 
$$
\left(\left(\tfrac{1}{9} - \gamma\right) \frac{n^3}{|V_0||V_1| |V_2|}  - 2 \right) 
(|V_0|)_k \times (|V_1|)_k \times (|V_2|)_k 
$$
many directed $\ell$-cycles $\vec{C}_{\ell}$.    
Since $|V_0||V_1||V_2| \leq n^3/(27)$ holds by convexity, $\vec{G}$ admits at least 
$$
(1 - 27 \gamma) 
(|V_0|)_k \times (|V_1|)_k \times (|V_2|)_k = \Omega(n^{\ell}) 
$$
many directed $\ell$-cycles $\vec{C}_{\ell}$, where we used $\gamma \leq 1/(28)$.

For Statement~(2), let $(G, c)$ be an $n$-vertex $\lambda$-extremal 
edge colored graph, 
and let $\vec{G} = (V, \vec{E})$ be a directed graph 
associated
with $(G, c)$  
which has 
$\lambda$-extremal 
partition $V = V_0 \cup V_1 \cup V_2$.  
Let $\vec{F} \subseteq \vec{G}$ consist of all $(v_i, v_{i+1}) \in \vec{E}$
where $v_i \in V_i$ 
and $v_{i+1} \in V_{i+1}$ for $i \in \mathbb{Z}_3$.  
Then $\vec{F}$ is an 
oriented graph with $\lambda$-extremal partition $V = V_0 \cup V_1 \cup V_2$ which, by 
Statement~(1), admits $\Omega(n^{\ell})$ directed $\ell$-cycles $\vec{C}_{\ell}$.  
Fact~\ref{prop:easyprop} ensures that 
$\Omega(n^{\ell}) - n^{\ell - 1}$ of these directed cycles
correspond to rainbow cycles in $(G, c)$, 
because the edge-colored graph $F$ determined by $\vec{F}$ is, 
by construction, 
a subgraph of $G$.  
\end{proof}

\subsection{Stability results}  
In what follows, 
we distinguish
between whether or not a given structure is 
$\lambda$-extremal (cf.~Definition~\ref{def:lambdaext}).

\begin{lemma}
\label{lem:non-ext}
For all $\lambda > 0$, 
there exists $\alpha = \alpha (\lambda) > 0$ so that for all 
integers $\ell \geq 4$, 
there exists 
an integer $n_0 = n_0(\lambda, \alpha, \ell) \geq 1$
so that 
whenever 
$\vec{G}$ is an oriented graph
on $n \geq n_0$ many vertices 
satisfying 
\begin{equation}
\label{eqn:non-exthyp}
\delta^+(\vec{G}) \geq 
\left\{
\begin{array}{cc} 
\left(\frac{1}{3} - \alpha\right)n & \text{if $\ell \neq 5$,}  \\
\frac{n+1}{3}                      & \text{if $\ell = 5$,}  
\end{array}
\right.  
\end{equation}  
then $\vec{G}$ is $\lambda$-extremal or $\vec{G}$ admits a closed directed $\ell$-walk  
$\vec{W}_{\ell}$.  
\end{lemma}  

\noindent We prove Lemma~\ref{lem:non-ext} in 
Sections~\ref{sec:non-ext}--\ref{sec:nonext2}.    
We apply Lemma~\ref{lem:non-ext} 
in the following convenient form.  

\begin{cor}[the non-extremal case]  
\label{cor:non-ext}  
In the context of Lemma~\ref{lem:non-ext}, 
the following statements hold: 
\begin{enumerate}
\item  If $\ell = 5$ and $\vec{G}$ is not $\lambda$-extremal, then $\vec{G}$ contains
a directed 5-cycle $\vec{C}_5$;   
\item  If $\ell \neq 5$ and $\vec{G}$ is not $\lambda$-extremal, then $\vec{G}$ contains
$\Omega(n^{\ell})$ many directed $\ell$-cycles $\vec{C}_{\ell}$.   
\end{enumerate}  
Moreover, for even integers $\ell$,  Statement~(2) above holds when $\vec{G}$
is allowed to be a directed graph.  
\end{cor}  

\noindent  Note that Statement~(1) of Corollary~\ref{cor:non-ext} restates 
the conclusion of Lemma~\ref{lem:non-ext} when $\ell = 5$, since the only closed
directed 
$5$-walk $\vec{W}_5$ is the 5-cycle $\vec{C}_5$.  
It is standard 
to derive Statement~(2) of Corollary~\ref{cor:non-ext} from Lemma~\ref{lem:non-ext} 
by using a suitable regularity lemma.  
We sketch such a proof below.  

\begin{remark}
\rm 
In the context of Lemma~\ref{lem:non-ext}, let $\vec{G}$ be an oriented graph on $n \geq n_0(\lambda, 
\ell)$ many vertices  
which satisfies~(\ref{eqn:non-exthyp}), where $\ell \neq 5$.    
We may apply Lemma~3.2
from~\cite{kelly2008dirac}
to obtain 
a regular partition $V = V_0 \cup V_1 \cup \dots \cup V_t$ 
of $\vec{G}$ with cluster digraph $\vec{R}$, where $\vec{R}$ 
may not be oriented.   
Nonetheless, Lemma~3.2 guarantees that $\vec{R}$ 
admits an oriented spanning subgraph $\vec{Q} \subseteq \vec{R}$, where 
$\delta^+(\vec{Q}) / t$ can be taken arbitrarily close to $\delta^+(\vec{G}) / n$, 
and where 
$\delta^-(\vec{Q}) / t$ can be taken arbitrarily close to 
$\delta^-(\vec{G}) / n$.  
As such, if the oriented graph $\vec{G}$ is not $\lambda$-extremal, then the oriented
graph $\vec{Q}$ isn't $\lambda'$-extremal for some suitably small $0 < \lambda' \leq \lambda$.  
Lemma~\ref{lem:non-ext} then guarantees that $\vec{Q}$ admits a closed directed 
$\ell$-walk $\vec{W}_{\ell}$.
Applying 
a counting lemma to 
the system
of pairs 
$(V_i, V_j)$
corresponding to the edges of $\vec{W}_{\ell}$ 
guarantees 
$\Omega(n^{\ell})$ 
many directed $\ell$-cycles $\vec{C}_{\ell}$.

When $\ell$ is even, $\vec{G}$ need not be oriented.  Here, we may apply Lemma~3.1 of~\cite{alon2003testing}  
to obtain a regular partition $V = V_0 \cup V_1 \cup \dots \cup V_t$ of $\vec{G}$ with cluster digraph $\vec{R}$.  
Again, if $\vec{G}$ is not $\lambda$-extremal, then $\vec{R}$ is not $\lambda'$-extremal
for some suitably small $0 < \lambda' \leq \lambda$.  
If $\vec{R}$ is, in fact, an oriented graph, then we proceed identically to the above.  
Assume that 
$\vec{R}$ admits a 2-cycle, i.e., a closed 2-walk $\vec{W}_2$.   
Since $\ell$ is even, 
the pair $(V_i, V_j)$ corresponding 
to $\vec{W}_2$ admits $\Omega(n^{\ell})$ many directed $\ell$-cycles $\vec{C}_{\ell}$.  \hfill $\Box$  
\end{remark}  

We continue with an extremal counterpart to Corollary~\ref{cor:non-ext}.  

\begin{lemma}[the extremal case]  
\label{lem:ext}
There exists an absolute constant $\lambda_0 > 0$ so that, for all $0 < \lambda \leq \lambda_0$ and for all integers $\ell \geq 4$
not divisible by three, there exists an integer $n_0 = n_0(\lambda_0, \lambda, \ell) \geq 1$ 
so that whenever $(G, c)$ is a $\lambda$-extremal edge colored graph
on $n \geq n_0$ many vertices, the following hold:  
\begin{enumerate}
\item 
If $\ell \neq 5$ and 
$\delta^c(G) \geq (n+5)/3$  
(cf.~{\rm (\ref{eqn:degcond})}), then $(G, c)$ admits 
a rainbow $\ell$-cycle $C_{\ell}$;  
\item  
If 
$\delta^c(G) \geq (n+4)/3$, 
then 
$(G, c)$ admits a properly colored $\ell$-cycle $C_{\ell}$.  
\end{enumerate}  
\end{lemma} 

\noindent  We prove Lemma~\ref{lem:ext} in 
Sections~\ref{sec:ext}--\ref{sec:ex3}.    
We proceed to the proofs of Theorems~\ref{thm:even} and~\ref{thm:or}, which are formal consequences
of Corollary~\ref{cor:non-ext} and Lemma~\ref{lem:ext}.

\subsection{Proof of Theorem~\ref{thm:even}}
To define the absolute constant $\alpha > 0$ promised by 
Theorem~\ref{thm:even}, 
we consider auxiliary parameters.  
Let 
$\lambda_{\text{Lem.\ref{lem:ext}}}$
be the absolute constant 
$\lambda_0$ 
guaranteed by
Lemma~\ref{lem:ext}.  Set 
\begin{equation}
\label{eqn:lambda1.3}  
\lambda = \min \big\{ \tfrac{1}{28}, \,  
\lambda_{\text{Lem.\ref{lem:ext}}} \big\},   
\end{equation}
which is suitably small for an application of Fact~\ref{fact:10}.  
With $\lambda > 0$ given
in~(\ref{eqn:lambda1.3}),   
let 
\begin{equation}
\label{eqn:alpha1.3}
\alpha = \alpha_{\text{Lem.\ref{lem:non-ext}}}(\lambda) > 0 
\end{equation}  
be the constant guaranteed by Lemma~\ref{lem:non-ext}, which we take to be the constant
promised by Theorem~\ref{thm:even}.

Fix an even integer 
$\ell \geq 4$.  
Let $(G, c)$ be an $n$-vertex edge-colored graph 
satisfying~(\ref{eqn:degcond}), where in all that follows we assume that 
$n \geq n_0(\lambda, \alpha, \ell)$ is sufficiently
large.  
To prove Theorem~\ref{thm:even}, 
we distinguish between the cases of whether or not 
$(G, c)$ is $\lambda$-extremal, where $\lambda$ 
is given in~(\ref{eqn:lambda1.3}).    \\

\noindent {\bf Case~1 ({\rm $(G,c)$ is $\lambda$-extremal}).}  
In this case, 
we apply Fact~\ref{fact:10} or Lemma~\ref{lem:ext} to $(G, c)$.  
Assume first that $\ell \equiv 0$ (mod 3).       
By our choice of 
$\lambda \leq 1/(28)$ 
from~(\ref{eqn:lambda1.3}),     
Statement~(2) of Fact~\ref{fact:10} 
guarantees $\Omega(n^{\ell})$ many rainbow $\ell$-cycles $C_{\ell}$ in 
$(G, c)$.  
Assume now that 
$\ell \not\equiv 0$ (mod 3).   
By our choice of $\lambda \leq 
\lambda_{\text{Lem.\ref{lem:ext}}}$ 
from~(\ref{eqn:lambda1.3}),     
Statement~(1) of Lemma~\ref{lem:ext}  
guarantees a rainbow $\ell$-cycle $C_{\ell}$ in $(G, c)$.  
(Note: $\ell \neq 5$ by the parity of $\ell$.)  
\hfill $\Box$  \\

\noindent {\bf Case~2 ({\rm $(G,c)$ is not $\lambda$-extremal}).}  
In this case, we will indirectly apply 
Fact~\ref{prop:easyprop}  
and Corollary~\ref{cor:non-ext} to $(G, c)$.     
For that, 
let $\vec{G} = (V, \vec{E})$ be any directed graph associated with $(G, c)$, where    
necessarily 
$\vec{G}$ is not $\lambda$-extremal, and where 
$\delta^+\big(\vec{G}\big) = \delta^c(G) \geq \left(\frac{1}{3} - \alpha\right) n$  
is ensured by~(\ref{eqn:degcond}).  
By our choice of $\alpha = \alpha_{\text{Lem.\ref{lem:non-ext}}}(\lambda)$
in~(\ref{eqn:alpha1.3}) (and $\ell \neq 5$), 
Statement~(2) of 
Corollary~\ref{cor:non-ext}
guarantees $\Omega (n^{\ell})$ many directed $\ell$-cycles
$\vec{C}_{\ell}$ in $\vec{G}$.  
Fact~\ref{prop:easyprop}  
then 
guarantees that 
at least one of these 
corresponds to a rainbow $\ell$-cycle $C_{\ell}$
in $(G, c)$. \hfill $\Box$

\subsection{Proof of Theorem~\ref{thm:or}}  
We again use the auxiliary constants $\lambda > 0$ and $\alpha > 0$ determined 
in~(\ref{eqn:lambda1.3})    
and~(\ref{eqn:alpha1.3}).     
Fix an integer $\ell \geq 4$.  
Let $\vec{G} = (V, \vec{E})$ be an $n$-vertex oriented graph satisfying 
$\delta^+\big(\vec{G}\big) \geq (n+1)/3$, 
where in all that follows we assume
that $n \geq n_0(\lambda, \alpha, \ell)$ is sufficiently large.  
Let $\vec{H} \subseteq \vec{G}$ be maximally induced w.r.t.~satisfying
$\delta^-(\vec{H}) \geq 1$, and set $U = V(\vec{H})$.    
Note that every $u \in U$ satisfies
$\deg^+_{\vec{H}}(u) = \deg^+_{\vec{G}}(u)$.
Consequently, 
$|U|=\Omega (n)$ can be taken as large as needed since 
the number $e(\vec{H})$ of edges of $\vec{H}$ satisfies 
$$
\binom{|U|}{2} \geq e\big(\vec{H}\big) \geq |U| \delta^+\big(\vec{H}\big) \geq |U| \delta^+
\big(\vec{G}\big)   
\qquad \implies \qquad 
|U| \geq 2 \delta^+\big(\vec{G}\big) \geq 2n/3.  
$$
We now distinguish between the cases of whether or not $\vec{H}$ is $\lambda$-extremal, 
where $\lambda$ is determined 
in~(\ref{eqn:lambda1.3}).      \\

\noindent {\bf Case 1 ({\rm $\vec{H}$ is not $\lambda$-extremal}).}  
In this case, we apply Corollary~\ref{cor:non-ext} to 
$\vec{H}$, which is possible  
on account that $\delta^+(\vec{H}) \geq \delta^+(\vec{G})
\geq (n+1)/3 \geq ((1/3) - \alpha)n$, 
for $\alpha = \alpha_{\text{Lem.\ref{cor:non-ext}}}$ in (\ref{eqn:alpha1.3}).     
Whether or not $\ell = 5$, 
Corollary~\ref{cor:non-ext} 
guarantees a directed 
$\ell$-cycle $\vec{C}_{\ell}$ in $\vec{H}$, where $\vec{C}_{\ell}$ also appears in 
$\vec{G} \supseteq \vec{H}$.  \hfill $\Box$  \\

\noindent {\bf Case 2 ({\rm $\vec{H}$ is $\lambda$-extremal}).}  
In this case, we will apply Fact~\ref{fact:10} to $\vec{H}$ or we will 
indirectly apply 
Fact~\ref{fact:Li} 
and Lemma~\ref{lem:ext} to $\vec{H}$.   
Assume first that 
$\ell \equiv 0$ (mod 3).   
By our choice of $\lambda \leq 1/(28)$ 
in~(\ref{eqn:lambda1.3}), 
Statement~(1) of Fact~\ref{fact:10} 
guarantees $\Omega(n^{\ell})$ many directed $\ell$-cycles $\vec{C}_{\ell}$ in $\vec{H}$, 
each of which also appears
in $\vec{G} \supseteq \vec{H}$.  
Assume now that $\ell \not\equiv 0$ (mod 3).  
Let 
$(H, c)$ be the edge-colored graph determined by 
$\vec{H}$, where $H$ has vertex set $U = V(\vec{H})$.  Since every vertex $u \in U$
has positive in-degree in $\vec{H}$, we have that 
$$
\deg^c_H(u) = 1 + \deg^+_{\vec{H}}(u) \geq \tfrac{n+4}{3}.  
$$
By our choice of $\lambda \leq \lambda_{\text{Lem.\ref{lem:ext}}}$ 
in~(\ref{eqn:lambda1.3}), 
Statement~(2) of Lemma~\ref{lem:ext} guarantees a properly colored $\ell$-cycle 
$C_{\ell}$ in $(H, c)$.  
Since $(H, c)$ was determined by the oriented graph $\vec{H}$, 
Fact~\ref{fact:Li} guarantees that $C_{\ell}$
corresponds to a directed $\ell$-cycle $\vec{C}_{\ell}$
in $\vec{H}$, which also appears in $\vec{G} \supseteq \vec{H}$.  
\hfill $\Box$  

\section{Proof of Lemma~\ref{lem:non-ext}}  
\label{sec:non-ext}  
Lemma~\ref{lem:non-ext} is a formal consequence 
of
the following two propositions
(recall $\delta_0(\vec{G})$ from Theorem~\ref{thm:KKO}).        

\begin{prop}
\label{prop:Lem17}  
For all $\beta > 0$, there exists $\alpha = \alpha(\beta) > 0$ so that
for every integer $\ell \geq 4$, 
there exists
an integer $n_0 = n_0(\beta, \alpha, \ell) \geq 1$ so that the following holds.
Let $\vec{G} = (V, \vec{E})$ be an oriented graph 
satisfying~{\rm (\ref{eqn:non-exthyp})}    
on 
$n \geq n_0$ many vertices.  
If $\vec{G}$ admits no closed directed $\ell$-walk, then $\vec{G}$ admits 
an induced subgraph $\vec{H} = \vec{G}[U]$ on $|U| = m \geq (1 - \beta)n$
many 
vertices which satisfies 
\begin{equation}
\label{eqn:Lem17}  
\delta_0\big(\vec{H}\big) \geq \big(\tfrac{\delta^+(\vec{G})}{n} - \beta\big) m.  
\end{equation}  
\end{prop}

\begin{prop}
\label{prop:Lem19}  
For all $\lambda_0 > 0$, there exists 
$\beta = \beta(\lambda_0) > 0$ so that for every integer
$\ell \geq 4$, there exists an integer $m_0 = m_0(\lambda_0, \beta, \ell) \geq
1$ so that the following holds.  
Let $\vec{H}$ be an oriented graph on $m \geq m_0$ vertices
which admits no closed directed $\ell$-walk, but which satisfies 
$\delta_0(\vec{H}) \geq ((1/3) - \beta)m$.  Then $\vec{H}$ 
is $\lambda_0$-extremal.
\end{prop}

The proof of Proposition~\ref{prop:Lem17}  
is not too difficult, and will be given later in this section.  The proof 
of
Proposition~\ref{prop:Lem19} is more involved, and will be postponed to the following section.  

\subsection{Proof of Lemma~\ref{lem:non-ext}}
Let $\lambda > 0$ 
be given.  
To define the 
constant $\alpha = \alpha(\lambda) > 0$ promised by 
Lemma~\ref{lem:non-ext}, we consider several auxiliary constants.  
First, 
set $\lambda_0 = \lambda / 2$, and     
let 
\begin{equation}
\label{eqn:beta19}  
\beta_{\text{Prop.\ref{prop:Lem19}}} = 
\beta_{\text{Prop.\ref{prop:Lem19}}}(\lambda_0) > 0
\end{equation} 
be the constant
guaranteed by Proposition~\ref{prop:Lem19}.  
Second, set 
\begin{equation}
\label{eqn:beta17}  
\beta = \tfrac{1}{2} \min \{ \lambda_0, 
\beta_{\text{Prop.\ref{prop:Lem19}}} \}.   
\end{equation}
Third, let 
\begin{equation}
\label{eqn:alpha17}  
\alpha_{\text{Prop.\ref{prop:Lem17}}} = 
\alpha_{\text{Prop.\ref{prop:Lem17}}} (\beta) > 
0 
\end{equation}  
be the constant guaranteed by Proposition~\ref{prop:Lem17}.  We define 
\begin{equation}
\label{eqn:alphanon-ext}  
\alpha = \min\{\alpha_{\text{Prop.\ref{prop:Lem17}}}, \beta \} 
\end{equation}
to be the constant promised by Lemma~\ref{lem:non-ext}. 
Let an integer $\ell \geq 4$ be given.  
Let $\vec{G} = (V, \vec{E})$ be an $n$-vertex oriented graph 
satisfying~(\ref{eqn:non-exthyp})
with $\alpha$ in~(\ref{eqn:alphanon-ext}),   
where in all that follows we assume that $n \geq n_0(\lambda, \alpha, \ell)$
is sufficiently large.      
We 
assume that $\vec{G}$ admits
no closed directed $\ell$-walk, and establish 
that $\vec{G}$ is $\lambda$-extremal.

Since $\vec{G}$ admits no closed directed $\ell$-walk, and by our choice of 
$\alpha \leq \alpha_{\text{Prop.\ref{prop:Lem17}}}$ 
in~(\ref{eqn:alpha17})   
and~(\ref{eqn:alphanon-ext}),   
Proposition~\ref{prop:Lem17} guarantees that $\vec{G}$
admits  
an induced subgraph $\vec{H} = \vec{G}[U]$ on 
$|U| = m \geq (1 - \beta)n$
(cf.~(\ref{eqn:beta17}))      
many vertices for which 
$$
\delta_0\big(\vec{H}\big) 
\stackrel{(\ref{eqn:Lem17})}{\geq} 
\big(\tfrac{\delta^+(\vec{G})}{n} - \beta\big)m  
\stackrel{(\ref{eqn:non-exthyp})}{\geq}  
\big(\tfrac{1}{3} - \alpha - \beta\big)m 
\stackrel{(\ref{eqn:alphanon-ext})}{\geq}    
\big(\tfrac{1}{3} - 2 \beta\big)m 
\stackrel{(\ref{eqn:beta17})}{\geq}      
\big(\tfrac{1}{3} - 
\beta_{\text{Prop.\ref{prop:Lem19}}}
\big)m.   
$$
Since $\vec{H}$ admits no closed directed $\ell$-walk,  
and by our choice of 
$\beta_{\text{Prop.\ref{prop:Lem19}}}$ 
in~(\ref{eqn:beta19}),   
Proposition~\ref{prop:Lem19} 
guarantees that 
$\vec{H}$ is $\lambda_0$-extremal.  
Let $U = V(\vec{H}) = U_0 \cup U_1 \cup U_2$ be any 
$\lambda_0$-extremal partition of $\vec{H}$
(cf.~Definition~\ref{def:lambdaext}), and   
let 
$V = V(\vec{G}) = V_0 \cup V_1 \cup V_2$ be any partition satisfying
$U_i \subseteq V_i$ for each 
$0 \leq i \leq 2$.  
Then, for each 
$i \in \mathbb{Z}_3$, 
\begin{multline*}  
e_{\vec{G}}(V_i, V_{i+1}) \geq e_{\vec{G}}(U_i, U_{i+1})
= e_{\vec{H}}(U_i, U_{i+1})  \\
\stackrel{\text{{\tiny Prop.\ref{prop:Lem19}}}}{\geq}
\left(\tfrac{1}{9} - \lambda_0\right)m^2   
\stackrel{\text{{\tiny Prop.\ref{prop:Lem17}}}}{\geq}
\left(\tfrac{1}{9} - \lambda_0\right) (1 - \beta)^2 n^2  
\stackrel{(\ref{eqn:beta17})}{\geq}      
\left(\tfrac{1}{9} - \lambda\right)n^2,  
\end{multline*}  
where we also used $\lambda = 2\lambda_0$.  
Thus, 
$V = V_0 \cup V_1 \cup V_2$ is a $\lambda$-extremal partition of 
$\vec{G}$, as desired.  

\subsection{Proof of Proposition~\ref{prop:Lem17}}    
\label{sec:Lem17}  
Let $\beta > 0$ be given.  
Define 
\begin{equation}
\label{eqn:11.1.2018.4:31p}  
\alpha = \beta^6 / (96).   
\end{equation}  
Let integer $\ell \geq 4$ be given.  
Let $\vec{G} = (V, \vec{E})$ be an $n$-vertex
oriented graph 
satisfying~(\ref{eqn:non-exthyp}), where in all that follows, we take
$n \geq n_0(\beta, \alpha, \ell)$ to be sufficiently large.  
Assume that $\vec{G}$ 
admits no closed directed $\ell$-walks.  
The subgraph
$\vec{H} = \vec{G}[U]$ 
desired in~(\ref{eqn:Lem17}) is induced 
on the following vertices of large in-degree:  
\begin{equation}
\label{eqn:Vhigh}
U = 
V_{\rm high} = 
\big\{v \in V: \, \deg_{\vec{G}}^-(v) \geq \delta^+(\vec{G})
- n(\beta^2/2)\big\}.      
\end{equation}
To see that $\vec{H} = \vec{G}[V_{\rm high}]$ 
satisfies~(\ref{eqn:Lem17}), we use
the following
claim (whose proof we defer for a moment).    

\begin{claim}
\label{clm:11.1.2018}  
$\Delta^-(\vec{G})  \leq \delta^+(\vec{G}) + n (\beta^3 / 4)$,    
where 
$\Delta^-(\vec{G})$ denotes the maximum in-degree in 
$\vec{G}$.  
\end{claim}  

\noindent  Using 
Claim~\ref{clm:11.1.2018},   
we will verify that 
$|U| = |V_{\rm high}| = m \geq (1 - \beta) n$.  
Indeed, 
with $V_{\rm low} = V \setminus V_{\rm high}$,  
\begin{multline*}  
n \delta^+\big(\vec{G}\big) \leq 
\sum_{u \in V} \deg_{\vec{G}}^+(u) 
= \sum_{v \in V}
\deg^{-}_{\vec{G}}(v) = 
\sum_{w \in V_{\rm low}} \deg^-_{\vec{G}}(v) 
+ 
\sum_{x \in V_{\rm high}} \deg^-_{\vec{G}}(w)   \\
\stackrel{(\ref{eqn:Vhigh})}{<} 
\big|V_{\rm low}\big| \left(\delta^+\big(\vec{G}\big) - \tfrac{1}{2} \beta^2 n\right)
+ 
\big|V_{\rm high} \big| \Delta^-\big(\vec{G}\big)  
\stackrel{\text{\tiny Clm.\ref{clm:11.1.2018}}}{\leq}  
\big|V_{\rm low}\big| \left(\delta^+\big(\vec{G}\big) - \tfrac{1}{2} 
\beta^2 n\right)
+ 
\big|V_{\rm high} \big| \left(\delta^+\big(\vec{G}\big) + 
\tfrac{1}{4} \beta^3n\right), 
\end{multline*}  
from which 
$2|V_{\rm low}| \leq \beta |V_{\rm high}|
\leq \beta n$ and $|V_{\rm high}| \geq (1 - (\beta/2))n$ follow.  
By construction, 
both 
\begin{multline*}  
\delta^+\big(\vec{H}\big) \geq \delta^+\big(\vec{G}\big) - 
|V_{\rm low}| \geq 
\delta^+\big(\vec{G}\big) - 
\tfrac{1}{2} \beta n  
\geq 
\big(\tfrac{\delta^+(\vec{G})}{n} - \beta\big) n 
\geq 
\big(\tfrac{\delta^+(\vec{G})}{n} - \beta\big) m, \\  
\text{and} \qquad 
\delta^-\big(\vec{H}\big) \geq 
\min \big\{ \deg^-_{\vec{G}}(v) : v \in V_{\high} \big\} 
- |V_{\rm low}|  
\stackrel{(\ref{eqn:Vhigh})}{\geq}   
\delta^+\big(\vec{G}\big) - \tfrac{1}{2} \beta^2 n - \tfrac{1}{2} \beta n
\geq 
\big(\tfrac{\delta^+(\vec{G})}{n} - \beta\big) m 
\end{multline*}  
hold, as promised in~(\ref{eqn:Lem17}).  
Thus, it remains to prove 
Claim~\ref{clm:11.1.2018}, where we will use the following fact.

\begin{fact}
\label{fact:11.1.2018}  
Let $R, S \subset V$ 
be some disjoint pair 
with sizes 
$|R| \geq \Delta^-(\vec{G})$ and $|S| \geq \delta^+(\vec{G})$, 
where $(S, R)$   
admits no 
path $s \rightarrow v \rightarrow r$ in $\vec{G}$ with   
$s \in S$ and $r \in R$.    
Then 
$\Delta^-(\vec{G})  \leq \delta^+(\vec{G}) + n (\beta^3 / 4)$      
(cf.~Claim~\ref{clm:11.1.2018}).    
\end{fact}

\begin{proof}[Proof of Fact~\ref{fact:11.1.2018}]  
Let $R, S \subset V$
be given as above.  
Fix $S_0 \subseteq S$ with $|S_0| = \delta^+(\vec{G})$ 
and set
$S_1 = N^+_{\vec{G}}(S_0)$.  Then 
$N^+_{\vec{G}}(S_1) \cap R = \emptyset$.  
Set $S_2 = N^+_{\vec{G}}(S_1)
\setminus S_0$ so that $R$, $S_0$ and $S_2$ are pairwise disjoint.  
Thus 
\begin{equation}
\label{eqn:11.1.2018.7:18p}  
\Delta^-\big(\vec{G}\big) + \delta^+\big(\vec{G}\big)
+ |S_2| \leq  
|R| + |S_0| + |S_2| \leq n.   
\end{equation}  
We double-count the number $e_{\vec{G}}(S_1, S_2)$
of edges from $S_1$ to $S_2$.  On the one hand, 
\begin{equation}  
\label{eqn:11.1.2018.7:30p}  
e_{\vec{G}}(S_1, S_2) \leq |S_2| \Delta^-\big(\vec{G}\big)
\stackrel{(\ref{eqn:11.1.2018.7:18p})}{\leq}    
\Delta^-\big(\vec{G}\big) \left(
n - \Delta^-\big(\vec{G}\big) - \delta^+\big(\vec{G}\big)  
\right).    
\end{equation}  
On the other hand, 
\begin{multline}  
\label{eqn:11.1.2018.7:42p}  
e_{\vec{G}}(S_1, S_2) 
\geq
|S_1| 
\delta^+\big(\vec{G}\big) - e_{\vec{G}}(S_1, S_0)  
\geq 
|S_1| 
\delta^+\big(\vec{G}\big) - \left(
|S_0||S_1| - e_{\vec{G}}(S_0, S_1)\right)   \\
= 
e_{\vec{G}}(S_0, S_1) \geq |S_0| \delta^+\big(\vec{G}\big) 
= 
\left(\delta^+\big(\vec{G}\big)\right)^2, 
\end{multline}  
where we twice used that $|S_0| = \delta^+(\vec{G})$.  
Comparing~(\ref{eqn:11.1.2018.7:30p})  
and~(\ref{eqn:11.1.2018.7:42p}), we infer 
\begin{multline*}  
\left(\Delta^-\big(\vec{G}\big)\right)^2 - 
\left(n - \delta^+\big(\vec{G}\big)\right) \Delta^-\big(\vec{G}\big)
+ 
\left(\delta^+\big(\vec{G}\big)\right)^2  
\leq 0 \\
\implies \qquad 
\Delta^-\big(\vec{G}\big)
\leq 
\frac{1}{2} \left( n - \delta^+\big(\vec{G}\big) + 
\sqrt{\left(n - 3\delta^+(\vec{G})\right)
\left(n + \delta^+(\vec{G})\right)}  
\right)  
\stackrel{(\ref{eqn:non-exthyp})}{\leq}  
\tfrac{1}{2} \left( n - \delta^+\big(\vec{G}\big) + 
n \sqrt{6\alpha}  
\right)  \\
= 
\delta^+\big(\vec{G}\big) + \tfrac{1}{2} \left( n - 3\delta^+\big(\vec{G}\big)
+ n \sqrt{6\alpha} \right)  
\stackrel{(\ref{eqn:non-exthyp})}{\leq}  
\delta^+\big(\vec{G}\big) + \tfrac{1}{2} 
n\left(3\alpha  + \sqrt{6 \alpha} \right)
\stackrel{(\ref{eqn:11.1.2018.4:31p})}{\leq}    
\delta^+\big(\vec{G}\big) + n\sqrt{6\alpha},   
\end{multline*}  
and so our choice of $\alpha = \beta^6/(96)$ 
from~(\ref{eqn:11.1.2018.4:31p})
completes the proof  
of Fact~\ref{fact:11.1.2018}.    
\end{proof}

We now prove 
Claim~\ref{clm:11.1.2018}.   

\subsubsection*{Proof of 
Claim~\ref{clm:11.1.2018}}  
Assume, on the contrary, that 
\begin{equation}  
\label{eqn:11.1.2018.7:10p}  
\Delta^-\big(\vec{G}\big) > \delta^+\big(\vec{G}\big) + 
\tfrac{1}{4} \beta^3 n.  
\end{equation}  
Then Fact~\ref{fact:11.1.2018} ensures that 
\begin{multline}  
\label{eqn:Fact4.4}  
\text{{\it every disjoint pair $R, S \subset V$ 
with 
$|R| \geq \Delta^-\big(\vec{G}\big)$ and $|S| \geq 
\delta^+\big(\vec{G}\big)$}}  \\ 
\text{{\it admits a  
directed 
path $s \rightarrow v \rightarrow r$ in $\vec{G}$ with   
$s \in S$ and $r \in R$.}}    
\end{multline}  
Fix $x_{\max} \in V$ 
satisfying $\deg^-_{\vec{G}}(x_{\max}) = \Delta^-(\vec{G})$.  
We distinguish several cases of $\ell \geq 4$.  \\

\noindent {\bf Case~1 ($\ell = 4$).}  
Set $R = N^-_{\vec{G}}(x_{\max})$ and $S = N^+_{\vec{G}}(x_{\max})$, which 
are disjoint    
and satisfy $|R| = \Delta^-(\vec{G})$ and $|S| \geq \delta^+(\vec{G})$.  
Then~(\ref{eqn:Fact4.4}) guarantees a 
directed 4-cycle $(x_{\max}, s, v, r, x_{\max})$,  
which 
contradicts that $\vec{G}$ admits
no closed directed 4-walks.  
In other words, (\ref{eqn:11.1.2018.7:10p}) must be false when $\ell = 4$.  \hfill $\Box$  \\

\noindent {\bf Case~2 ($\ell = 5$).}  
We use 
the following peculiar observation, proven in a moment:  
\begin{equation}
\label{eqn:obs2.32}  
\text{{\it if $x_{\max} \rightarrow y \rightarrow z \rightarrow a$ 
is a directed path in $\vec{G}$, then $(x_{\max}, a) \not\in \vec{E}$, 
$(x_{\max}, z) \not\in \vec{E}$, and $(y, a) \not\in \vec{E}$.}} 
\end{equation}  
Using~(\ref{eqn:obs2.32}), 
$N^+_{\vec{G}}(x_{\max})$ is an independent
set whose every fixed element $y \in N^+_{\vec{G}}(x_{\max})$ has an 
independent
out-neighborhood $N^+_{\vec{G}}(y)$ which is disjoint from   
$N^+_{\vec{G}}(x_{\max})$.  
Thus, 
for $z \in 
N^+_{\vec{G}}(y)$ fixed, 
it must be that 
$N^+_{\vec{G}}(z) \cap N^+_{\vec{G}}(x_{\max}) \neq \emptyset$ 
since otherwise 
$N^+_{\vec{G}}(x_{\max}) \cup N^+_{\vec{G}}(y) \cup N^+_{\vec{G}}(z) \subseteq V$
is a disjoint union with 
$$
\deg^+_{\vec{G}}(x_{\max}) + \deg^+_{\vec{G}}(y) + \deg^+_{\vec{G}}(z) 
\geq 3 \delta^+\big(\vec{G}\big) 
\geq 
n+1 \qquad \text{(recall $\ell = 5$ in~(\ref{eqn:non-exthyp}))}.   
$$
On the other hand, 
$N^+_{\vec{G}}(z) \cap N^+_{\vec{G}}(x_{\max}) \neq \emptyset$
violates~(\ref{eqn:obs2.32}), and so   
(\ref{eqn:11.1.2018.7:10p}) is false when $\ell = 5$.

To see~(\ref{eqn:obs2.32}),   
we first 
observe that 
\begin{equation}
\label{eqn:4.8.2019.10:45a}  
N^+_{\vec{G}}(a) \cap N^-_{\vec{G}}(x_{\max}) = \emptyset,   
\end{equation}  
since 
$b \in N^+_{\vec{G}}(a) \cap N^-_{\vec{G}}(x_{\max})$
would give the directed 5-cycle 
$(x_{\max}, y, z, a, b, x_{\max})$, which would contradict that $\vec{G}$ admits no closed
directed 5-walks.     
Now, (\ref{eqn:4.8.2019.10:45a}) forbids $(x_{\max}, a) \in \vec{E}$, since otherwise   
we set 
$R = N^-_{\vec{G}}(x_{\max})$ and 
$S = N^+_{\vec{G}}(a)$ 
and use~(\ref{eqn:Fact4.4}) to guarantee a directed 5-cycle 
$(x_{\max}, a, s, v, r, x_{\max})$.  
We next observe that 
\begin{equation}
\label{eqn:4.8.2019.10:59a}  
N^+_{\vec{G}}(a) \cap N^+_{\vec{G}}(x_{\max}) \neq \emptyset, 
\end{equation}  
since 
otherwise~(\ref{eqn:4.8.2019.10:45a}) gives 
that 
$N^+_{\vec{G}}(a) \cup N^+_{\vec{G}}(x_{\max}) 
\cup N^-_{\vec{G}}(x_{\max}) \subseteq V$ is a disjoint union with 
$$
\deg^+_{\vec{G}}(a) + \deg^+_{\vec{G}}(x_{\max}) 
+ \deg^-_{\vec{G}}(x_{\max})
\stackrel{(\ref{eqn:11.1.2018.7:10p})}{>} 
3 \delta^+\big(\vec{G}\big)  
\stackrel{(\ref{eqn:non-exthyp})}{\geq}
n + 1.       
$$
Using~(\ref{eqn:4.8.2019.10:59a}),   
fix $b \in N^+_{\vec{G}}(a) \cap N^+_{\vec{G}}(x_{\max})$.   
Then 
\begin{equation}
\label{eqn:4.8.2019.11:12a}  
N^+_{\vec{G}}(b) \cap 
N^-_{\vec{G}}(x_{\max}) \neq \emptyset,  
\end{equation}  
as otherwise we set 
$R = N^-_{\vec{G}}(x_{\max})$ and 
$S = N^+_{\vec{G}}(b)$
and use~(\ref{eqn:Fact4.4})    
to guarantee 
a directed 5-cycle $(x_{\max}, b, s, v, r, x_{\max})$.  
Using~(\ref{eqn:4.8.2019.11:12a}), we fix   
$c \in N^+_{\vec{G}}(b) \cap 
N^-_{\vec{G}}(x_{\max})$, which forbids 
$(x_{\max}, z) \in \vec{E}$ lest
$(x_{\max}, z, a, b, c, x_{\max})$ is a directed 5-cycle.  
Similarly
$(y, a) \not\in \vec{E}$, 
which proves~(\ref{eqn:obs2.32}).  
\hfill $\Box$   \\

\noindent {\bf Case 3 ($\ell \geq 6$).}  
By the argument of Case~1, $x_{\max}$ belongs to a directed 4-cycle $\vec{C}_4$.  
We first observe that 
$x_{\max}$ does not belong 
to a directed 3-cycle $\vec{C}_3$.   
Indeed\footnote{This statement holds for all integers $\ell \geq 3$
outside of $\ell = 5$, and can be proven by inducting on $\ell = 
\lfloor \ell/2 \rfloor + \lceil \ell/2 \rceil$.},   
\begin{equation}
\label{eqn:numfact}
\text{{\it every integer $\ell \geq 6$ can be expressed as 
$\ell = 3i + 4j$ for some integers $i, j \geq 0$,}} 
\end{equation}  
and so the inclusion of $x_{\max}$ along both a directed 3-cycle $\vec{C}_3$ and a directed
4-cycle $\vec{C}_4$ 
would place $x_{\max}$ in a closed directed $\ell$-walk in $\vec{G}$, contradicting our hypothesis.  
We next observe that a longest directed path 
$\vec{P} = (y_1, \dots, y_k)$ 
in $N^+_{\vec{G}}(x_{\max})$
satisfies $k = \Omega(n)$.  
Indeed, 
$|N^+_{\vec{G}}(y_k) \cap N^+_{\vec{G}}(x_{\max})|
\leq k - 2$
holds by the optimal length of $\vec{P}$, and so 
\begin{equation}
\label{eqn:4.8.2019.12:42p}  
\big|N^+_{\vec{G}}(y_k) \cup N^+_{\vec{G}}(x_{\max}) \big| = 
\deg^+_{\vec{G}}(y_k) + \deg^+_{\vec{G}}(x_{\max}) - 
\big|N^+_{\vec{G}}(y_k) \cap N^+_{\vec{G}}(x_{\max}) \big| 
\geq 
2 \delta^+\big(\vec{G}\big) - k.  
\end{equation}  
Since $x_{\max}$ belongs to no directed 3-cycles $\vec{C}_3$, 
\begin{multline}  
\label{eqn:4.8.2019.12:44p} 
N^+_{\vec{G}}(y_k) \cap 
N^-_{\vec{G}}(x_{\max}) = \emptyset
= N^+_{\vec{G}}(x_{\max}) \cap 
N^-_{\vec{G}}(x_{\max})  \quad 
\implies \quad 
N^+_{\vec{G}}(y_k) \cup N^+_{\vec{G}}(x_{\max}) \subseteq V \setminus 
N^-_{\vec{G}}(x_{\max})  \\
\implies \quad 
\big|N^+_{\vec{G}}(y_k) \cup N^+_{\vec{G}}(x_{\max}) \big| \leq n - 
\deg^-_{\vec{G}}(x_{\max})  
= 
n - 
\Delta^-\big(\vec{G}\big).  
\end{multline}  
Then $k = \Omega(n)$ follows 
comparing~(\ref{eqn:4.8.2019.12:42p})   
and~(\ref{eqn:4.8.2019.12:44p}):  
$$
k \geq 
2 \delta^+\big(\vec{G}\big) + 
\Delta^-\big(\vec{G}\big)  
- n 
\stackrel{(\ref{eqn:11.1.2018.7:10p})}{>}    
3\delta^+\big(\vec{G}\big) + \tfrac{1}{4} \beta^3n - n 
\stackrel{(\ref{eqn:non-exthyp})}{\geq}  
n\left(
\tfrac{1}{4} \beta^3 - 3 \alpha\right)
\stackrel{(\ref{eqn:11.1.2018.4:31p})}{=}    
n\left(
\tfrac{1}{4} \beta^3 - \tfrac{1}{32} \beta^6\right)  
\geq \beta^3 n / 8.  
$$

To conclude Case~3, set 
$R = N^-_{\vec{G}}(x_{\max})$ and $S = N^+_{\vec{G}}(y_k)$, which we observed
above 
are disjoint.  
Then~(\ref{eqn:Fact4.4}) guarantees a path $(s, v, r)$ with $s \in S$ and $r \in R$, 
whence 
$$
\big(x_{\max}, y_{k - \ell + 5}, y_{k-\ell + 6}, \dots, y_k, s, v, r, x_{\max}\big) 
$$
is a closed directed $\ell$-walk, contradicting our hypothesis.
In other words, (\ref{eqn:11.1.2018.7:10p}) must be false when $\ell \geq 6$.  
which proves Claim~\ref{clm:11.1.2018}.   \hfill $\Box$

\section{Proof of Proposition~\ref{prop:Lem19}}    
\label{sec:nonext2}  
In this section, we prove Proposition~\ref{prop:Lem19}, where   
we will use several auxiliary facts.  
The first fact is taken from Corollary~1.5 in~\cite{ji2018short}.  

\begin{fact}
\label{fact:JiWuSong}  
Fix an integer $\ell \geq 4$.  
Let $\vec{G} = (V, \vec{E})$ be a large $n$-vertex
oriented graph
which contains no directed triangle, but 
which satisfies 
$\delta_0(\vec{G}) 
\geq (0.3025)n$.    
Then $\vec{G}$ admits a directed $\ell$-cycle $\vec{C}_{\ell}$.  
\end{fact}

\noindent  Our remaining facts 
are 
independent of the context of proving Proposition~\ref{prop:Lem19}, and are therefore 
verified in Section~\ref{sec:prooffacts}.  

\begin{fact}
\label{fact:Clm19.1}  
Fix an integer $\ell \geq 4$ and an 
$\eps \in (0, 1/(11)]$.  Let $\vec{G} = (V, \vec{E})$ be an oriented 
graph on $n \geq n_0(\ell, \eps)$ many vertices which admits 
no closed directed $\ell$-walk, but which satisfies $\delta_0(\vec{G}) \geq ((1/3) - \eps)n$.  
Let 
$(U_0, U_1)$ be a pair of subsets 
$U_0, U_1 \subseteq V$ satisfying the following conditions:
\begin{enumerate}
\item[(i)] $|U_0|, |U_1| \geq \delta_0(\vec{G})$; 
\item[(ii)] $|U_0 \cap U_1| \leq ((1/3) - 21\eps) n$; 
\item[(iii)] $\vec{G}$ admits no directed paths $u_0 \rightarrow v \rightarrow u_1$, where $u_0 \in U_0$ and $u_1 \in U_1$.  
\end{enumerate}  
Then, there exist independent sets $I_0 \subseteq U_0\setminus U_1$
and $I_1 \subseteq U_1 \setminus U_0$ with sizes 
\begin{equation}
\label{eqn:Clm19.1}  
|I_0| \geq |U_0 \setminus U_1| - 7\eps n \geq 20 \eps n \qquad \text{and} \qquad 
|I_1| \geq |U_1 \setminus U_0| - 7\eps n \geq 20 \eps n.  
\end{equation}  
\end{fact}  

\begin{remark}
\rm 
In many applications of 
Fact~\ref{fact:Clm19.1}, the pair $(U_0, U_1)$ 
will satisfy 
$U_0 \cap U_1 = \emptyset$.    \hfill $\Box$  
\end{remark}  

\begin{fact}
\label{fact:Clm19.2}  
Fix an integer $\ell \geq 4$ and an $\eps \in (0, 1/(54))$.  Let
$\vec{G} = (V, \vec{E})$ be an oriented graph on $n \geq n_0(\ell, \eps)$
many vertices which admits no closed directed $\ell$-walk, but which satisfies
$\delta_0(\vec{G}) \geq ((1/3) - \eps) n$.  
Let 
$(x, y, z, x)$ 
be a directed 
3-cycle $\vec{C}_3$ 
in $\vec{G}$, 
and assume that neither $x$ nor $y$ belongs to a directed 4-cycle
$\vec{C}_4$.  
Then, 
$|N^{-}_{\vec{G}}(x) \cap N^+_{\vec{G}}(y)| \geq 
((1/3) - 18 \eps)n$.  
\end{fact}

\noindent  We now prove Proposition~\ref{prop:Lem19}, and distinguish whether or not 
$\ell = 5$.

\subsection{Proof of Proposition~\ref{prop:Lem19} when $\boldsymbol{\ell \neq 5}$}        
Fix $\lambda_0 > 0$.  Define 
the promised constant 
\begin{equation}
\label{eqn:11.10.2018.6:07p}  
\beta = \beta(\lambda_0) = \min \left\{
\tfrac{1}{21} \lambda_0, \, 
\tfrac{1}{3} - 0.3025, \tfrac{1}{55} \right\}.  
\end{equation}  
Fix an integer $\ell \geq 4$, where $\ell \neq 5$.  
Let $\vec{H} = (V, \vec{E})$ be an $m$-vertex oriented graph, 
where $m \geq m_0(\lambda_0, \beta, \ell)$ is assumed to be sufficiently large
whenever needed.  Assume that $\vec{H}$   
admits no closed directed $\ell$-walk but satisfies 
$\delta_0(\vec{H}) \geq
((1/3) - \beta)m$.  
We prove that $\vec{H}$ is $\lambda_0$-extremal.

The central observation 
of the proof 
is 
that $\vec{H}$ admits directed triangles, since otherwise   
with 
$$
\delta_0\big(\vec{H}\big) 
\geq 
\left(\tfrac{1}{3} - \beta\right)m 
\stackrel{(\ref{eqn:11.10.2018.6:07p})}{\geq}    
(0.3025) m   
$$
Fact~\ref{fact:JiWuSong}
would guarantee a directed $\ell$-cycle $\vec{C}_{\ell}$ in $\vec{H}$, contradicting 
our hypothesis.  
Thus, fix a directed 3-cycle $(v_0, v_1, v_2, v_0)$ in $\vec{H}$.   
Our observation in~(\ref{eqn:numfact})   
guarantees that 
no vertex $v_i \in \{v_0, v_1, v_2\}$ can belong to a directed 4-cycle 
$\vec{C}_4$ lest $\vec{H}$ admits a closed directed $\ell$-walk.     
For fixed $i \in \mathbb{Z}_3$, 
we define 
\begin{equation}
\label{eqn:11.10.2018.3:08p}  
U_i = N_{\vec{H}}^-(v_i) \cap N_{\vec{H}}^+(v_{i+1}).  
\end{equation} 
Then $U_0$, $U_1$, and $U_2$ are pairwise disjoint because $\vec{H}$ is an oriented
graph.  
By our choice of $\beta < 1/(54)$ in~(\ref{eqn:11.10.2018.6:07p}), and by no 
$v_j \in \{v_0, v_1, v_2\}$ belonging to a directed 4-cycle $\vec{C}_4$,   
Fact~\ref{fact:Clm19.2}  
guarantees that 
\begin{equation}
\label{eqn:11.10.2018.3:11p}  
|U_i| = 
\big|N_{\vec{H}}^-(v_i) \cap N_{\vec{H}}^+(v_{i+1}) \big| 
\geq 
\left(\tfrac{1}{3} - 18 \beta\right) n.     
\end{equation}  
We claim that each $u_i \in U_i$ satisfies 
\begin{equation}
\label{eqn:11.10.2018.3:40p}  
\big|N^+_{\vec{H}}(u_i) \cap U_{i+1}\big| \geq 
\left(\tfrac{1}{3} - 45 \beta \right) n.    
\end{equation}  
If true, 
any partition $V = V_0 \cup V_1 \cup V_2$, where $U_j \subseteq V_j$ for each $j \in \mathbb{Z}_3$, 
is $\lambda_0$-extremal since 
\begin{multline*}  
e_{\vec{H}}(V_i, V_{i+1}) \geq 
e_{\vec{H}}(U_i, U_{i+1}) = \sum_{u_i \in U_i} \big|N^+_{\vec{H}}(u_i) \cap U_{i+1}\big|
\stackrel{(\ref{eqn:11.10.2018.3:40p})}{\geq}    
|U_i| 
\left(\tfrac{1}{3} - 45 \beta \right) n   \\ 
\stackrel{(\ref{eqn:11.10.2018.3:11p})}{\geq}    
\left(\tfrac{1}{3} - 45 \beta \right) \left(\tfrac{1}{3} - 18 \beta\right) n^2     
\geq 
\left(\tfrac{1}{9} - 21 \beta \right)n^2 
\stackrel{(\ref{eqn:11.10.2018.6:07p})}{\geq}    
\left(\tfrac{1}{9} - \lambda_0 \right)n^2.   
\end{multline*}

To prove~(\ref{eqn:11.10.2018.3:40p}), fix $i \in \mathbb{Z}_3$, and 
w.l.o.g.~assume 
$i = 0$.  
Fix $u_0 \in U_0 = N^-_{\vec{H}}(v_0) \cap N^+_{\vec{H}}(v_1)$.  
Then
$(v_1, u_0, v_0, v_1)$ is a directed 3-cycle $\vec{C}_3$, 
and so~(\ref{eqn:numfact})
gives that 
$u_0$ can belong to no 
directed 4-cycle $\vec{C}_4$.
As such, Fact~\ref{fact:Clm19.2} 
(applied to $(v_1, u_0, v_0, v_1)$)
guarantees that 
\begin{equation}  
\label{eqn:11.10.2018.5:27p}  
\big|N^-_{\vec{H}}(v_1) \cap N^+_{\vec{H}}(u_0)\big| \geq 
\left(
\tfrac{1}{3} 
- 18 \beta\right) n,    
\end{equation}  
which isn't 
yet~(\ref{eqn:11.10.2018.3:40p}), but it will be very close.    
With an error we can control, 
we shall `replace'  
$N^-_{\vec{H}}(v_1)$ 
in~(\ref{eqn:11.10.2018.5:27p})    
with 
$U_1 
\subseteq N^-_{\vec{H}}(v_1)$
from~(\ref{eqn:11.10.2018.3:08p}).  
We claim 
this error will be small 
if 
\begin{equation}
\label{eqn:11.11.2018.11:27a}  
\deg^-_{\vec{H}}(v_1) \leq \big(\tfrac{1}{3} + 9\beta \big)n.    
\end{equation}  
Indeed, 
if~(\ref{eqn:11.11.2018.11:27a}) holds, then we would have   
\begin{equation}  
\label{eqn:11.10.2018.4:42p}  
\big|N^-_{\vec{H}}(v_1) \setminus U_1\big|
\stackrel{(\ref{eqn:11.10.2018.3:08p})}{=}  
\deg^-_{\vec{H}}(v_1) - |U_1| 
\stackrel{(\ref{eqn:11.10.2018.3:11p})}{\leq}    
\deg^-_{\vec{H}}(v_1) - 
\left(\tfrac{1}{3} - 18 \beta\right)n  
\stackrel{(\ref{eqn:11.11.2018.11:27a})}{\leq}    
27 \beta n,   
\end{equation}  
and so comparing~(\ref{eqn:11.10.2018.5:27p})   
with~(\ref{eqn:11.10.2018.4:42p})    
yields 
$$
\big|N^+_{\vec{H}}(u_0) \cap U_1\big|
+ 27 \beta n 
\stackrel{(\ref{eqn:11.10.2018.4:42p})}{\geq}      
\big|
N^+_{\vec{H}}(u_0)
\cap 
N^-_{\vec{H}}(v_1) 
\big| 
\stackrel{(\ref{eqn:11.10.2018.5:27p})}{\geq}     
\left(\tfrac{1}{3} - 18 \beta\right)n,  
$$
which gives~(\ref{eqn:11.10.2018.3:40p}).  It thus remains to 
prove that~(\ref{eqn:11.11.2018.11:27a}) holds.   

To prove~(\ref{eqn:11.11.2018.11:27a}),  
we will apply 
Fact~\ref{fact:Clm19.1}  
to the pair 
$(N^-_{\vec{H}}(v_1), N^+_{\vec{H}}(v_1))$.  
Note that 
the hypotheses~(i)--(iii) 
of Fact~\ref{fact:Clm19.1}  
are met 
by $(N^-_{\vec{H}}(v_1), N^+_{\vec{H}}(v_1))$  
since 
$|N^-_{\vec{H}}(v_1)|$, $|N^+_{\vec{H}}(v_1)| \geq \delta_0(\vec{H})$, 
since 
$|N^-_{\vec{H}}(v_1) \cap N^+_{\vec{H}}(v_1)| = 0$, 
and since there are no paths $u^+ \rightarrow v \rightarrow u^-$ with 
$u^+ \in N^+_{\vec{H}}(v_1)$ and 
$u^- \in N^-_{\vec{H}}(v_1)$
lest 
$(v_1, u^+, v, u^-, v_1)$ 
is a directed 4-cycle containing $v_1$.   
Fact~\ref{fact:Clm19.1}  
guarantees 
an independent set 
$I_{v_1} \subseteq 
N^-_{\vec{H}}(v_1)$
of size 
\begin{equation}
\label{eqn:11.10.2018.4:33p}  
|I_{v_1}| \geq \deg^-_{\vec{H}}(v_1) - 7\beta n 
\geq \delta_0\big(\vec{H}\big) - 7 \beta n
\geq 
\left(\tfrac{1}{3} - \beta\right)n - 7\beta n  \geq 
\left(\tfrac{1}{3} - 8 \beta\right)n 
\stackrel{(\ref{eqn:11.10.2018.6:07p})}{>} 0,   
\end{equation}  
so fix $w_1 \in I_{v_1}$.  
Now, $N^+_{\vec{H}}(w_1) \cup N^-_{\vec{H}}(w_1) \cup I_{v_1} 
\subseteq V$ is a  
pairwise disjoint union, in which case
\begin{multline*}  
n \geq \deg^+_{\vec{H}}(w_1) + \deg^-_{\vec{H}}(w_1) + |I_{v_1}|  
\stackrel{(\ref{eqn:11.10.2018.4:33p})}{\geq}    
\deg^+_{\vec{H}}(w_1) + \deg^-_{\vec{H}}(w_1) + 
\deg^-_{\vec{H}}(v_1) - 7 \beta n  \\
\geq 
2 \delta_0\big(\vec{H}\big) 
+ \deg^-_{\vec{H}}(v_1) - 7 \beta n
\geq 
2 \left(\tfrac{1}{3} - \beta \right)n 
+ \deg^-_{\vec{H}}(v_1) 
- 7 \beta n 
= 
\deg^-_{\vec{H}}(v_1) 
+ \left(\tfrac{2}{3} - 9 \beta \right)n, 
\end{multline*}  
from which~(\ref{eqn:11.11.2018.11:27a}) now follows.   

\subsection{Proof of Proposition~\ref{prop:Lem19} when $\boldsymbol{\ell = 5}$}      
To prove Proposition~\ref{prop:Lem19} when 
$\ell = 5$, we use 
Facts~\ref{fact:JiWuSong}--\ref{fact:Clm19.2}  
together with the following two additional facts (which are also proven in Section~\ref{sec:prooffacts}).   

\begin{fact}
\label{fact:Cor18}  
Fix an integer $\ell \geq 4$ and an $\eps > 0$.  Let $\vec{G} = (V, \vec{E})$ be an oriented graph on $n \geq n_0(\ell, \eps)$ many vertices
which admits no closed directed $\ell$-walk.  
Then, $\delta^+(\vec{G}) \leq ((1/3) + \eps)n$    
and $\delta^-(\vec{G}) \leq ((1/3) + \eps)n$.  
\end{fact}  

\begin{fact}  
\label{prop:2.3}  
For all $\lambda > 0$, there exists $\eps = \eps (\lambda) > 0$ so that every 
oriented graph $\vec{G} = (V, \vec{E})$ on 
$n \geq n_0(\lambda, \eps)$ many vertices
with $\delta_0(\vec{G}) \geq ((1/3) - \eps)n$ 
will be $\lambda$-extremal, provided $\vec{G}$ 
has:  
\begin{enumerate}
\item  
a partition $V = V_0 \cup V_1 \cup V_2$ with  
$|V_1|,|V_2| \geq ((1/3) - \eps)n$  
and $e_{\vec{G}}(V_1), e_{\vec{G}}(V_2), e_{\vec{G}}(V_2, V_1) \leq 
\eps n^2$,     
\item or 
no transitive triangles.      
\end{enumerate}  
\end{fact}  

Now, 
let $\lambda_0 > 0$ be given.  Let 
\begin{equation}
\label{eqn:11.11.2018.12:03p} 
\eps_{\text{\tiny Fct.\ref{prop:2.3}}} 
= \eps_{\text{\tiny Fct.\ref{prop:2.3}}} (\lambda = \lambda_0) 
> 0 
\end{equation} 
be the constant guaranteed by 
Fact~\ref{prop:2.3}.  
We define the promised constant 
\begin{equation}
\label{eqn:11.11.2018.12:06p}  
\beta = 
\tfrac{1}{109}  
\eps_{\text{\tiny Fct.\ref{prop:2.3}}}.   
\end{equation}    
Let $\vec{H} = (V, \vec{E})$ be an $m$-vertex oriented graph, 
where in all that follows we assume $m \geq m_0(\lambda_0, 
\eps_{\text{\tiny Fct.\ref{prop:2.3}}}, \beta)$ is sufficiently large.   
Assume that $\vec{H}$ admits no closed directed 5-walks, i.e., 
directed 5-cycles $\vec{C}_5$, but 
which satisfies $\delta_0(\vec{H}) \geq ((1/3) - \beta)m$. 
We prove that $\vec{H}$ is $\lambda_0$-extremal.

For sake of argument, we assume that $\vec{H}$ admits some transitive
triangles, as otherwise by 
our choice of $\beta$ and 
$\eps_{\text{\tiny Fct.\ref{prop:2.3}}}$ 
in~(\ref{eqn:11.11.2018.12:03p})  
and~(\ref{eqn:11.11.2018.12:06p}),    
Conclusion~(2) of 
Fact~\ref{prop:2.3} would give that $\vec{H}$ is $\lambda_0$-extremal.  
For the remainder of the proof, 
we fix a transitive triangle $(x, y), (x, z), (y, z) 
\in \vec{E}$.  
Let $I = I_{x,z} = N^-_{\vec{H}}(x) \cap N^+_{\vec{H}}(z)$, which is an independent
set lest $(a, b) \in \vec{E} \cap (I \times I)$ gives the directed 5-cycle
$(x, y, z, a, b, x)$.  
Our first main observation is 
that $I$ is `large'.  

\begin{claim}  
\label{clm:4.10.2019.12:02p}  
\begin{equation}
\label{eqn:11.11.2018.12:25p}  
|I| \geq \left(\tfrac{1}{3} - 21 \beta\right) n.  
\end{equation}  
\end{claim}  

\begin{proof}[Proof of Claim~\ref{clm:4.10.2019.12:02p}]    
Assume for contradiction that~(\ref{eqn:11.11.2018.12:25p}) fails to hold.   
We will apply
Fact~\ref{fact:Clm19.1}
to the pair $(N^-_{\vec{H}}(x), N^+_{\vec{H}}(z))$.   
Note that 
the hypotheses~(i)--(iii) are met
by $(N^-_{\vec{H}}(x), N^+_{\vec{H}}(z))$   
since 
$|N^-_{\vec{H}}(x)|, |N^+_{\vec{H}}(z))| \geq \delta_0(\vec{H})$, 
since 
$|N^-_{\vec{H}}(x) \cap N^+_{\vec{H}}(z))|
\leq ((1/3) - 21\beta) n$ 
on account that~(\ref{eqn:11.11.2018.12:25p}) failed, and since
there are no paths $u^+ \rightarrow v \rightarrow u^-$ with $u^+
\in N^+_{\vec{H}}(z)$ and $u^-\in N^-_{\vec{H}}(x)$ lest
$(x, z, u^+, v, u^-, 
x)$
is a directed 5-cycle $\vec{C}_5$.  
Fact~\ref{fact:Clm19.1}
guarantees disjoint independent sets $I_x \subseteq N^-_{\vec{H}}(x) \setminus
N^+_{\vec{H}}(z)$ and 
$I_z \subseteq N^+_{\vec{H}}(z) \setminus N^-_{\vec{H}}(x)$  
(disjoint also from $I$) 
with sizes 
\begin{multline}  
\label{eqn:11.11.2018.12:38p}  
|I_x| 
\geq 
\big|N^-_{\vec{H}}(x) \setminus N^+_{\vec{H}}(z)\big| - 7 \beta n = 
\deg_{\vec{H}}^-(x) - |I| - 7\beta n  \geq 20 \beta n\\
\text{and} \quad 
|I_z| 
\geq 
\big|N^+_{\vec{H}}(z) \setminus N^-_{\vec{H}}(x)\big| - 7 \beta n = 
\deg_{\vec{H}}^+(z) - |I| - 7\beta n \geq 20 \beta n.     
\end{multline}  
Fix $a_x \in I_x$ and $b_z \in I_z$.   
One may check that 
$$
N^-_{\vec{H}}(a_x) \cap N^+_{\vec{H}}(b_z)
=
N^-_{\vec{H}}(a_x) \cap I
=
N^+_{\vec{H}}(b_z) \cap I 
=
N^-_{\vec{H}}(a_x) \cap I_z
=
N^+_{\vec{H}}(b_z) \cap I_x = \emptyset.
$$ 
Thus, together with the independence 
of $I$, $I_x$, and $I_z$, we have that 
$I \cup I_x \cup I_y \cup N^-(a_x) \cup N^+(b_z) \subseteq V$ is a pairwise disjoint union, and so 
\begin{multline*}  
n \geq |I| + |I_x| + |I_z| + \deg^-_{\vec{H}}(a_x) + \deg^+_{\vec{H}}(b_z)  
\stackrel{(\ref{eqn:11.11.2018.12:38p})}{\geq}    
\deg^-_{\vec{H}}(x)
+ 
\deg^+_{\vec{H}}(z)
- |I|  
+ 
\deg^-_{\vec{H}}(a_x)
+ 
\deg^+_{\vec{H}}(b_z)
- 14 \beta n \\
\geq 
4 \delta_0\big(\vec{H}\big) 
- |I|  
- 14 \beta n
\stackrel{\text{\tiny{hyp}}}{\geq}  
4 \left(\tfrac{1}{3} - \beta\right)n - |I| - 14 \beta n 
= 
n - |I| + \left(\tfrac{1}{3} - 18 \beta\right) n,  
\end{multline*}  
from which $|I| \geq ((1/3) - 18 \beta)n$ follows,  
and contradicts our assumption that~(\ref{eqn:11.11.2018.12:25p}) failed to hold.    
\end{proof}

Continuing the proof of Proposition~\ref{prop:Lem19}, we 
attempt to meet 
Condition~(1) 
of Fact~\ref{prop:2.3} to $\vec{H}$ with $V_1 = I$
and with $V_2$ which we now define.  
For the remainder of the proof, 
fix $v \in I$
and take $V_2 = I_{v}$ to be a maximal independent set 
in 
$N^+_{\vec{H}}(v)$    
or 
$N^-_{\vec{H}}(v)$.    
In the former case, 
$0 = e_{\vec{H}}(V_1) = e_{\vec{H}}(V_2) = e_{\vec{H}}(V_2, V_1)$ since each of
$V_1 = I$
and $V_2 = I_{v}$ is independent and since 
\begin{equation}
\label{eqn:11.11.2018.4:52p}  
\text{$a \in N^+_{\vec{H}}(v)$ forbids $b \in N^+_{\vec{H}}(a) \cap I$}, 
\end{equation}  
lest 
$(x, z, v, a, b, x)$
is a directed 5-cycle $\vec{C}_5$ in $\vec{H}$.      
In the latter case, 
$0 = e_{\vec{H}}(V_1) = e_{\vec{H}}(V_2) = e_{\vec{H}}(V_1, V_2)$, 
where the last equality holds by $a \in N^-_{\vec{G}}(v)$ forbidding
$b \in N^-_{\vec{H}}(a) \cap I$ lest 
$(x, z, b, a, v, x)$
is a directed 5-cycle $\vec{C}_5$.  
In either case, 
we make the following claim.  

\begin{claim}
\label{clm:4.10.2019.12:27p}  
\begin{equation}
\label{eqn:11.11.2018.1:46p}  
|V_2| = \big|I_{v}\big| \geq \left(\tfrac{1}{3} - 24 \beta\right) n.   
\end{equation}  
\end{claim}  
\noindent If 
Claim~\ref{clm:4.10.2019.12:27p} holds,  
then together
with~(\ref{eqn:11.11.2018.12:25p}) and the considerations above, 
the partition $V = V_0 \cup V_1 \cup V_2$, where $V_0 = V \setminus (V_1 \cup V_2)$,  
meets the hypotheses 
of Fact~\ref{prop:2.3}.   
By 
our choice of 
$\eps_{\text{\tiny Fct.\ref{prop:2.3}}}$  
and $\beta$ in~(\ref{eqn:11.11.2018.12:03p}) and~(\ref{eqn:11.11.2018.12:06p}), 
Fact~\ref{prop:2.3} 
guarantees that $V = V_0 \cup V_1 \cup V_2$ is a $\lambda_0$-extremal partition of $\vec{H}$.
Thus, the proof of Proposition~\ref{prop:Lem19} when $\ell = 5$ will be complete upon 
proving Claim~\ref{clm:4.10.2019.12:27p}.

\begin{proof}[Proof of 
Claim~\ref{clm:4.10.2019.12:27p}]      
Assume for contradiction that the $\vec{H}$-subgraphs 
$\vec{H}[N^+_{\vec{H}}(v)]$ 
and 
$\vec{H}[N^-_{\vec{H}}(v)]$ 
induced 
respectively on 
$N^+_{\vec{H}}(v)$ 
and 
$N^-_{\vec{H}}(v)$ 
satisfy 
\begin{equation}
\label{eqn:11.11.2018.2:03p}  
\alpha \big(
\vec{H}\big[N^+_{\vec{H}}(v)\big]\big) 
< 
\left(\tfrac{1}{3} - 24 \beta \right)n      
\qquad \text{and} \qquad 
\alpha \big(
\vec{H}\big[N^-_{\vec{H}}(v)\big]\big) 
< 
\left(\tfrac{1}{3} - 24 \beta \right)n,    
\end{equation}  
where $\alpha (\cdot)$ denotes the independence number.  
Since 
$$
\big|N^+_{\vec{H}}(v)\big| 
= \deg^+_{\vec{H}}(v)  \geq \delta_0\big(\vec{H}\big) \geq  
\left(\tfrac{1}{3} - \beta\right)n      
> 
\left(\tfrac{1}{3} - 24 \beta \right)n      
> 
\alpha \big(
\vec{H}\big[N^+_{\vec{H}}(v)\big]\big), 
$$
$\vec{H}[N^+_{\vec{H}}(v)]$ admits 
edges $(a, b) \in \vec{E}$.   We fix one such and will observe
that  
\begin{equation}
\label{eqn:4.10.2019.1:10p}  
N^+_{\vec{H}}(b) \cap N^+_{\vec{H}}(v) \neq \emptyset.  
\end{equation}  
Indeed, assuming otherwise 
the set 
$N^+_{\vec{H}}(b) \cap N^-_{\vec{H}}(v)$ satisfies 
\begin{equation}  
\label{eqn:4.10.2019.1:12p}  
\big|N^+_{\vec{H}}(b) \cap N^-_{\vec{H}}(v)\big|
= 
\big|N^+_{\vec{H}}(b)\big| +  \big|N^-_{\vec{H}}(v)\big|
- 
\big|N^+_{\vec{H}}(b) \cup N^-_{\vec{H}}(v)\big|.     
\end{equation}  
From 
$\emptyset = N^+_{\vec{H}}(b) \cap N^+_{\vec{H}}(v) = 
N^+_{\vec{H}}(b) \cap I$  
(cf.~(\ref{eqn:11.11.2018.4:52p}) 
and~(\ref{eqn:4.10.2019.1:10p}))   
we infer
$N^+_{\vec{H}}(b) \subseteq 
V \setminus (N^+_{\vec{H}}(v) \cup I)$, and from 
$\emptyset = 
N^-_{\vec{H}}(v) \cap 
N^+_{\vec{H}}(v)  = 
N^-_{\vec{H}}(v) \cap 
I$ 
(recall that $I$ is independent)  
we infer 
$N^+_{\vec{H}}(b) \cup N^-_{\vec{H}}(v) \subseteq 
V \setminus (N^+_{\vec{H}}(v) \cup I)$, where   
$N^+_{\vec{H}}(v) \cup I$ is a disjoint union by the independence of $I$.   
Thus, 
\begin{multline*}  
\big|N^+_{\vec{H}}(b) \cap N^-_{\vec{H}}(v)\big|
\stackrel{(\ref{eqn:4.10.2019.1:12p})}{\geq}  
\deg^+_{\vec{H}}(b) +  \deg^-_{\vec{H}}(v)
- 
\big(n - \deg^+_{\vec{H}}(v) - |I|\big) \\
\geq
3 \delta_0\big(\vec{H}\big) - n + |I|  
\geq 
3 \left(\tfrac{1}{3} - \beta\right)n - n + |I| 
\stackrel{(\ref{eqn:11.11.2018.12:25p})}{\geq}
\left(\tfrac{1}{3} - 24 \beta\right)n  
\stackrel{(\ref{eqn:11.11.2018.2:03p})}{>}    
\alpha \big(
\vec{H}\big[N^-_{\vec{H}}(v)\big]\big).   
\end{multline*}  
Consequently, there exists an edge $(c, d) \in \vec{E}$ with $c, d \in 
N^+_{\vec{H}}(b) \cap N^-_{\vec{H}}(v)$, in which case 
$(v, a, b, c, d, v)$
is a directed 5-cycle $\vec{C}_5$ in $\vec{H}$.  
This proves~(\ref{eqn:4.10.2019.1:10p}).

Now, 
define 
\begin{equation}
\label{eqn:11.11.2018.4:59p}  
C = \left\{c \in N^+_{\vec{H}}(v): \text{ 
  $\exists$ a directed path 
on $3$ vertices contained in $\vec{H}[N^+_{\vec{H}}(v)]$ that ends in $c$}
\right\}.   
\end{equation}  
Note that \eqref{eqn:4.10.2019.1:10p} implies that $C$ is non-empty.
By this definition, 
every element $c \in C$ satisfies 
\begin{equation}
\label{eqn:11.11.2018.4:36p}  
N^+_{\vec{H}}(c) \cap N^+_{\vec{H}}(v)
= N^+_{\vec{H}}(c) \cap C.  
\end{equation}  
Since the 
$\vec{H}$-subgraph 
$\vec{H}[C]$ induced on $C$ admits no directed 5-cycles
$\vec{C}_5$,    
Fact~\ref{fact:Cor18}  
guarantees the existence of a vertex $c_0 \in C \subseteq N^+_{\vec{H}}(v)$ 
(cf.~(\ref{eqn:11.11.2018.4:59p}))
so that 
\begin{equation}
\label{eqn:11.11.2018.4:28p}  
\big|N^+_{\vec{H}}(c_0) \cap N^+_{\vec{H}}(v)\big|  
\stackrel{(\ref{eqn:11.11.2018.4:36p})}{=}    
\big|N^+_{\vec{H}}(c_0) \cap C \big|  
\leq 
\left\{
\begin{array}{cc}
\left(\tfrac{1}{3} + \beta\right) |C| & \text{ if $|C| = \Omega(1)$}  \\
|C|  & \text{ else}  
\end{array}
\right\}    
\leq  
\left(\tfrac{1}{3} + \beta\right) \big|N^+_{\vec{H}}(v)\big|.   
\end{equation}
Consider 
now 
$N^+_{\vec{H}}(c_0) \cap N^-_{\vec{H}}(v)
= 
(N^+_{\vec{H}}(c_0) \setminus C) \cap N^-_{\vec{H}}(v)$, where   
$C \subseteq N^+_{\vec{H}}(v)$
from~(\ref{eqn:11.11.2018.4:59p}) but where $N^+_{\vec{H}}(v) \cap N^-_{\vec{H}}(v_i)
= \emptyset$.       
Then 
\begin{multline*}  
\big|N^+_{\vec{H}}(c_0) \cap N^-_{\vec{H}}(v)\big|
= 
\big|\big(N^+_{\vec{H}}(c_0) \setminus C\big) \cap N^-_{\vec{H}}(v)\big|
= 
\big|\big(N^+_{\vec{H}}(c_0) \setminus C\big)\big|
+ 
\big|N^-_{\vec{H}}(v)\big|
- 
\big|\big(N^+_{\vec{H}}(c_0) \setminus C\big) \cup N^-_{\vec{H}}(v)\big| \\
\stackrel{(\ref{eqn:11.11.2018.4:28p})}{\geq}    
\deg^+_{\vec{H}}(c_0) - 
\left(\tfrac{1}{3} + \beta\right) \deg^+_{\vec{H}}(v) 
+ \deg_{\vec{H}}^-(v) 
- 
\big|\big(N^+_{\vec{H}}(c_0) \setminus C\big) \cup N^-_{\vec{H}}(v)\big| \\  
\stackrel{(\ref{eqn:11.11.2018.4:36p})}{=}
\deg^+_{\vec{H}}(c_0) - 
\left(\tfrac{1}{3} + \beta\right) \deg^+_{\vec{H}}(v)  
+ \deg_{\vec{H}}^-(v) 
- 
\big|\big(N^+_{\vec{H}}(c_0) \setminus N^+_{\vec{H}}(v) \big) 
\cup N^-_{\vec{H}}(v)\big|.  
\end{multline*}  
Using~(\ref{eqn:11.11.2018.4:52p}) and the independence of $I$,    
the last union resides in $V \setminus
(N^+_{\vec{H}}(v) \cup I)$, and so   
\begin{multline}  
\label{eqn:11.11.2018.5:35p}  
\big|N^+_{\vec{H}}(c_0) \cap N^-_{\vec{H}}(v)\big|
\geq 
\deg^+_{\vec{H}}(c_0) - 
\left(\tfrac{1}{3} + \beta\right) \deg^+_{\vec{H}}(v) 
+ \deg_{\vec{H}}^-(v) 
- 
\big(n - \deg^+_{\vec{H}}(v) - |I|\big)  \\
= 
\deg_{\vec{H}}^+(c_0) + \deg^-_{\vec{H}}(v) 
+ 
\left(\tfrac{2}{3} - \beta\right) \deg^+_{\vec{H}}(v)  
+ |I| - n 
\geq 
\left(\tfrac{8}{3} - \beta\right) \delta_0\big(\vec{H}\big) + |I| - n   
\qquad 
\qquad 
\\
\geq 
\left(\tfrac{8}{3} - \beta\right) 
\left(\tfrac{1}{3} - \beta\right)n   
+ |I| - n 
\stackrel{(\ref{eqn:11.11.2018.12:25p})}{\geq}
\left(\tfrac{8}{3} - \beta\right) 
\left(\tfrac{1}{3} - \beta\right)n   
+ 
\left(\tfrac{1}{3} - 21\beta\right)n 
- n 
\geq 
\left(\tfrac{2}{9} 
- 24 \beta
\right)n,   
\end{multline}  
which is positive 
by~(\ref{eqn:11.11.2018.12:06p}).    
Now,~(\ref{eqn:11.11.2018.4:59p})   
and~(\ref{eqn:11.11.2018.5:35p}) render a directed path 
$a \rightarrow b \rightarrow c_0 \rightarrow d$ where 
$a, b, c_0 \in N^+_{\vec{H}}(v)$ and 
$d \in N^+_{\vec{H}}(c_0) \cap N^-_{\vec{H}}(v)$, in which case 
$(v, a, b, c_0, d, v)$ 
is a directed 5-cycle in $\vec{H}$.  
Thus, our assumption 
in~(\ref{eqn:11.11.2018.2:03p}) is incorrect,  
which completes the proof of Claim~\ref{clm:4.10.2019.12:27p}.        
\end{proof}

\section{Proofs of Facts~\ref{fact:Clm19.1}--\ref{prop:2.3}}    
\label{sec:prooffacts}  
The easiest proof here is that of Fact~\ref{fact:Cor18}, which we give 
immediately.  
Fix an integer $\ell \geq 4$ and fix $\eps > 0$.  Let $\vec{G} = (V, \vec{E})$
be an oriented graph on $n \geq n_0(\ell, \eps)$ many vertices which admits
no closed directed $\ell$-walk.  
The latter conclusion of Fact~\ref{fact:Cor18} 
follows from the former by reversing the orientations on $\vec{E}$.  
If the former fails, 
then Proposition~\ref{prop:Lem17} ensures a large $m$-vertex 
subgraph $\vec{H} \subseteq \vec{G}$ satisfying 
$$
\delta_0\big(\vec{H}\big)
\geq 
\big(\tfrac{\delta^+(\vec{G})}{n} - \tfrac{\eps}{2}\big)m \geq 
\big(\tfrac{1}{3} + \tfrac{\eps}{2} \big)m \geq \tfrac{m+1}{3},    
$$
and so Theorem~\ref{thm:KKO} guarantees a directed $\ell$-cycle $\vec{C}_{\ell}$
in $\vec{H}$, and hence in $\vec{G}$.

\subsection*{Proof of Fact~\ref{fact:Clm19.1}}  
Fix an integer $\ell \geq 4$ and an $\eps \in (0, 1/(11)]$.  Let $\vec{G} = (V, \vec{E})$
be an oriented graph on $n \geq n_0(\ell, \eps)$ many vertices which admits 
no closed directed $\ell$-walk, but which satisfies $\delta_0(\vec{G}) \geq ((1/3) - \eps)n$.
Let $(U_0, U_1)$ be a pair of subsets satisfying~(i)--(iii) in 
the hypotheses of Fact~\ref{fact:Clm19.1}.  We prove that there exist
independent sets $I_0 \subseteq U_0 \setminus U_1$ and $I_1 \subseteq U_1 \setminus U_0$
satisfying~(\ref{eqn:Clm19.1}).    
To that end, define 
$$
T_0 = \big\{u_0 \in U_0 \setminus U_1 : \, 
N^-_{\vec{G}}(u_0) \cap (U_0 \setminus U_1) \neq \emptyset \big\} 
\quad \text{and} \quad    
T_1 = \big\{u_1 \in U_1 \setminus U_0 : \, 
N^+_{\vec{G}}(u_1) \cap (U_1 \setminus U_0) \neq \emptyset \big\}.   
$$
For fixed $j \in \mathbb{Z}_2$, the set $I_j = U_j \setminus (U_{j+1} \cup T_j)$ 
is independent and of size 
$|I_j| = |U_j\setminus U_{j+1}| - |T_j|$, so 
to prove~(\ref{eqn:Clm19.1}) we will prove 
$|T_j| \leq 7 \eps n$.  
In particular, our argument will show that $|T_0| \geq \eps n$ and $|T_1| \geq \eps n$ can't both hold, and that $|T_j| < \eps n$ implies 
$|T_{j+1}| \leq 7 \eps n$.   
It remains to verify these details.    

Write $U = U_0 \cup U_1$, and define 
\begin{equation}
\label{eqn:Clm19.1.9}
S_0 = N^+(T_0) \setminus U \qquad \text{and} \qquad S_1 = N^-(T_1) \setminus U, 
\qquad \text{where} \qquad S_0 \cap S_1 \stackrel{\text{{\tiny (iii)}}}{=} \emptyset.    
\end{equation}  
For $j \in \mathbb{Z}_2$, 
we will verify the implications  
\begin{equation}
\label{eqn:11.3.2018.1:43p}
|T_j| \geq \eps n \qquad \implies \qquad |S_j| \geq \delta_0\big(\vec{G}\big) - \left(\tfrac{1}{3} + \eps\right) |T_j|  
\qquad \implies \qquad |T_{j+1}| < \eps n.  
\end{equation}  
Indeed, $\vec{G}[T_j]$ is a large oriented graph with no closed directed $\ell$-walks, so 
Fact~\ref{fact:Cor18} guarantees $t_j \in T_j$:  
\begin{equation}
\label{eqn:Clm19.1.10}
|N^+_{\vec{G}}(t_0) \cap T_0| \leq \left(\tfrac{1}{3} + \eps\right)|T_0| \qquad \text{and} \qquad 
|N^-_{\vec{G}}(t_1) \cap T_1| \leq \left(\tfrac{1}{3} + \eps\right)|T_1|.
\end{equation}  
By definition, there exist $u_0 \in U_0 \setminus U_1$ and 
$u_1 \in U_1 \setminus U_0$ with $(u_0, t_0), (t_1, u_1) \in \vec{E}$, where 
$N^+(t_0) \cap T_0 = N^+(t_0) \cap (U_0 \setminus U_1)$
and $N^-(t_1) \cap T_1 = N^-(t_1) \cap (U_1 \setminus U_0)$ hold.  
Moreover, 
$u \in N^+_{\vec{G}}(t_0) \cap U_1$ is impossible lest $(u_0, t_0, u)$ violates~(iii), 
and $N^-_{\vec{G}}(t_1) \cap U_0 \neq \emptyset$ is similarly 
impossible.
Altogether, 
\begin{multline*}  
\deg^+_{\vec{G}}(t_0) = 
\big|N^+_{\vec{G}}(t_0) \cap U\big| + \big|N^+_{\vec{G}}(t_0) \setminus U\big|
= 
\big|N^+_{\vec{G}}(t_0) \cap (U_0\setminus U_1) \big| + 
\big|N^+_{\vec{G}}(t_0) \setminus U\big|  \\
= \big|N^+_{\vec{G}}(t_0) \cap T_0 \big| + 
\big|N^+_{\vec{G}}(t_0) \setminus U\big|  
\stackrel{(\ref{eqn:Clm19.1.10})}{\leq}  
\left(\tfrac{1}{3} + \eps\right)|T_0| 
+ 
\big|N^+_{\vec{G}}(t_0) \setminus U\big|
\stackrel{(\ref{eqn:Clm19.1.9})}{\leq} 
\left(\tfrac{1}{3} + \eps\right)|T_0| 
+ |S_0|, 
\end{multline*}  
and so the former implication of~(\ref{eqn:11.3.2018.1:43p}) holds with $j = 0$.  
Similarly, 
$$
\deg^-_{\vec{G}}(t_1) 
= 
\big|N^-_{\vec{G}}(t_1) \cap U\big| + \big|N^-_{\vec{G}}(t_1) \setminus U\big|
\leq 
\big|N^-_{\vec{G}}(t_1) \cap T_1\big| + |S_1|
\leq 
\left(\tfrac{1}{3} + \eps\right)|T_1| + |S_1|, 
$$
and so the former implication of~(\ref{eqn:11.3.2018.1:43p}) holds with $j = 1$.  
Finally, if both $|T_0|, |T_1| \geq \eps n$, then 
\begin{multline*}  
n 
\stackrel{(\ref{eqn:Clm19.1.9})}{\geq}  
|U| + |S_0| + |S_1|  
\stackrel{(\ref{eqn:11.3.2018.1:43p})}{\geq}
|U| + 2 \delta_0\big(\vec{G}\big) - \left(\tfrac{1}{3} + \eps\right) (|T_0| + |T_1|)  \\
\geq 
|U|  + 2 \delta_0\big(\vec{G}\big) - \left(\tfrac{1}{3} + \eps\right) |U_0 \triangle  U_1|   
= 
\left(\tfrac{2}{3} - \eps\right) (|U_0| + |U_1|)
+ 2 \delta_0\big(\vec{G}\big)  
- 
\left(\tfrac{1}{3} - 2\eps\right) |U_0 \cap U_1|   \\
\stackrel{\text{{\tiny (i)}}}{\geq} 
2 \delta_0\big(\vec{G}\big) \left(\tfrac{5}{3} - \eps\right)  - \left(\tfrac{1}{3} - 2\eps\right) |U_0 \cap U_1|
\stackrel{\text{{\tiny (ii)}}}{\geq} 
2\delta_0\big(\vec{G}\big) \left(\tfrac{5}{3} - \eps\right) - 
\left(\tfrac{1}{3} - 2\eps\right)
\left(\tfrac{1}{3} - 21\eps\right)n  \\
\geq 
2 \delta_0\big(\vec{G}\big) \left(\tfrac{5}{3} - \eps\right)  - 
\left(\tfrac{1}{9}
- \tfrac{23}{3} \eps + 42 \eps^2\right)n  
\geq 
2 \left(\tfrac{1}{3} - \eps\right) \left(\tfrac{5}{3}-\eps\right)n - 
\left(\tfrac{1}{9}
- \tfrac{23}{3} \eps + 42\eps^2 \right)n, 
\end{multline*}  
from which $\eps \geq 11 / (120)$ follows and contradicts
the hypothesis $\eps \leq 1/(11)$.  
This proves~(\ref{eqn:11.3.2018.1:43p}).  

By~(\ref{eqn:11.3.2018.1:43p}), it suffices to assume for fixed $j \in \mathbb{Z}_2$
that $|T_j| \geq \eps n$,     
and then to prove that $|T_j| \leq 7\eps n$.  
To that end, 
we find a vertex $z_{j+1} \in U_{j+1} \setminus T_{j+1}$ where, 
\begin{enumerate}
\item[$(a)$]  
when $j = 0$, the vertex $z_1 \in U_1 \setminus T_1$ has no 
in-neighbors from $U_1 \setminus T_1$; 
\item[$(b)$]  
when $j = 1$, the vertex $z_0 \in U_0 \setminus T_0$ has no 
out-neighbors in $U_0 \setminus T_0$.
\end{enumerate}   
We start 
by fixing $v_{j+1} \in I_{j+1}
= U_{j+1} \setminus (U_j \cup T_{j+1})$, 
which is possible by 
\begin{multline*}  
|I_{j+1}| = \big|U_{j+1} \setminus (U_j \cup T_{j+1})\big| = |U_{j+1}| - |U_0 \cap U_1| - |T_{j+1}| 
\stackrel{\text{{\tiny (i)}}}{\geq} 
\delta_0\big(\vec{G}\big) - |U_0 \cap U_1| - |T_{j+1}|   \\
\stackrel{\text{{\tiny (ii)}}}{\geq} 
\delta_0\big(\vec{G}\big) - \left(\tfrac{1}{3} - 21\eps\right) n - |T_{j+1}|  
\stackrel{(\ref{eqn:11.3.2018.1:43p})}{\geq}
\delta_0\big(\vec{G}\big) - \left(\tfrac{1}{3} - 21\eps\right) n - \eps n 
\geq 
19 \eps n.        
\end{multline*}  
Consider~$(a)$ above ($j = 0$).  
If $v_1$ has an in-neighbor $w_1 \in U_1 \setminus T_1$, then $w_1 \in U_0 \cap U_1$
because $I_1$ is independent.  If $w_1$ has an in-neighbor $x_1 \in U_1 \setminus T_1$, 
then $x_1 \in I_1$ lest we violate~(iii).  
If $x_1$ has an in-neighbor $y_1 \in U_1 \setminus T_1$, then $y_1 \in U_0 \cap U_1$ because 
$I_1$ is independent, but now we have violated~(iii).  Thus, some $z_1 \in \{v_1, w_1, x_1, y_1\}$
has no in-neighbor within $U_1 \setminus T_1$.  
Purely symmetric arguments establish~$(b)$.  

We use~$(a)$ and~$(b)$ above to conclude the proof  
of Fact~\ref{fact:Clm19.1}, where we first consider
$j = 0$.  The sets $U_1\setminus T_1$, $S_0$, and $T_0$
are pairwise disjoint by construction, and the set $Z_1 = N^-_{\vec{G}}(z_1)$
is disjoint from $U_1 \setminus T_1$ by~$(a)$ and is disjoint from each of $S_0$ and $T_0$
by~(iii).  
When $j = 1$, 
the sets $Z_0 = N^+(z_0)$, $U_0 \setminus T_0$, $S_1$, and $T_1$
are similarly pairwise disjoint.  
Thus, 
for whichever  
$j \in \mathbb{Z}_2$ satisfies $|T_j| \geq \eps n$, we have 
\begin{multline*}  
n \geq |Z_{j+1}| + |U_{j+1} \setminus T_{j+1}| + |S_j| + |T_j| 
\geq 
\delta_0\big(\vec{G}\big) + |U_{j+1}| - |T_{j+1}| + |S_j| + |T_j| \\
\stackrel{(\ref{eqn:11.3.2018.1:43p})}{>}  
2 \delta_0\big(\vec{G}\big) + |U_{j+1}| - \eps n  
+ \left(\tfrac{2}{3} - \eps\right) |T_j| 
\stackrel{\text{{\tiny (i)}}}{\geq}  
3 \delta_0\big(\vec{G}\big) - \eps n  
+ \left(\tfrac{2}{3} - \eps\right) |T_j| 
\geq 
n - 4\eps n 
+ \left(\tfrac{2}{3} - \eps\right) |T_j|,  
\end{multline*}  
from which $|T_j| \leq (132/(19)) \eps n < 7 \eps n$ follows 
from $\eps \in (0, 1/(11)]$.

\subsection*{Proof of Fact~\ref{fact:Clm19.2}}  
Fix an integer $\ell \geq 4$ and an $\eps \in (0, 1/(54))$.  Let $\vec{G} = (V, \vec{E})$
be an oriented graph on $n \geq n_0(\ell, \eps)$ many vertices which 
admits no closed directed 
$\ell$-walk, but which satisfies $\delta_0(\vec{G}) \geq ((1/3) - \eps) n$.
Let $(x, y, z, x)$ be a directed 3-cycle $\vec{C}_3$ in $\vec{G}$, 
and assume that neither $x$ nor $y$ belongs to a directed
4-cycle $\vec{C}_4$.  Assume, on the contrary, that 
\begin{equation}
\label{eqn:11.6.2018.5:45p}  
\big|N^{-}_{\vec{G}}(x) \cap N^+_{\vec{G}}(y) \big| < 
\left(\tfrac{1}{3} - 18 \eps\right) n.  
\end{equation}  
We will show that 
our assumption in~(\ref{eqn:11.6.2018.5:45p}) implies 
\begin{equation}
\label{eqn:11.9.2018.7:50p}  
N_{\vec{G}}^+(x) \cap N^-_{\vec{G}}(y) \neq \emptyset, 
\end{equation}  
in which case 
an element 
$v \in N_{\vec{G}}^+(x) \cap N^-_{\vec{G}}(y)$ 
would 
result in the directed
4-cycle 
$(x, v, y, z, x)$
containing both 
$x$ and $y$.    
We now establish the details for~(\ref{eqn:11.9.2018.7:50p}).    

First, 
we apply Fact~\ref{fact:Clm19.1} to each of the pairs $(X_0, X_1)$ and $(Y_0, Y_1)$, where 
$$
X_0 = N^+_{\vec{G}}(x), \qquad 
X_1 = N^-_{\vec{G}}(x), \qquad 
Y_0 = N^+_{\vec{G}}(y), \qquad 
Y_1 = N^-_{\vec{G}}(y).      
$$
Note that the hypotheses of 
Fact~\ref{fact:Clm19.1} 
are met 
since $\vec{G}$ 
admits no closed directed $\ell$-walks
but satisfies $\delta_0(\vec{G}) \geq ((1/3) - \eps)n$ for $0< \eps < 1/(54) < 1/(11)$, 
and where    
e.g.~$(X_0, X_1)$ satisfies the hypotheses~(i)--(iii) of Fact~\ref{fact:Clm19.1}  
since $|X_0|, |X_1| \geq \delta_0(\vec{G})$,  
since $|X_0 \cap X_1| = 0$, 
and since a directed path 
$x_0 \rightarrow v \rightarrow x_1$
with $x_0 \in X_0$ and $x_1 \in X_1$ would give a directed 4-cycle
$\vec{C}_4$ containing $x$.  
Fact~\ref{fact:Clm19.1} guarantees
independent sets $I_x \subseteq X_1 \setminus X_0 = X_1 = N^-_{\vec{G}}(x)$ 
and $I_y \subseteq Y_0 \setminus Y_1 = Y_0 = N^+_{\vec{G}}(y)$ of respective
sizes 
$|I_x| \geq |N^-_{\vec{G}}(x)| - 7\eps n$
and $|I_y| \geq |N^+_{\vec{G}}(y)| - 7 \eps n$.  
Note that $|I_x \cap I_y|$ 
is bounded by~(\ref{eqn:11.6.2018.5:45p}), we but claim that $I_x \cap I_y = 
\emptyset$.
Indeed, 
a vertex $v \in I_x \cap I_y$ must have its neighborhood $N_{\vec{G}}(v) = N^-_{\vec{G}}(v)
\cup N^+_{\vec{G}}(v)$ outside of $I_x \cup I_y$, 
and so 
\begin{multline*}  
\big|N^{-}_{\vec{G}}(x) \cap N^+_{\vec{G}}(y) \big|  
\geq |I_x \cap I_y|  
= |I_x| + |I_y| - |I_x \cup I_y| \geq 
|I_x| + |I_y| + 
\big|N_{\vec{G}}^+(v)\big| 
+ \big|N_{\vec{G}}^-(v)\big| 
- n   \\
\stackrel{\text{{\tiny Fct.\ref{fact:Clm19.1}}}}{\geq} 
\big|N_{\vec{G}}^-(x)\big| 
- 7 \eps n 
+ 
\big|N_{\vec{G}}^+(y)\big| 
- 7 \eps n 
+ 
\big|N_{\vec{G}}^+(v)\big| 
+ \big|N_{\vec{G}}^-(v)\big| 
- n  \\ 
\geq 4 \delta_0\big(\vec{G}\big) - 14 \eps n - n  
\geq 
4 \left(\tfrac{1}{3} - \eps\right)n - 14\eps - n  
= 
\left(\tfrac{1}{3} - 18 \eps\right) n 
\end{multline*}  
contradicts~(\ref{eqn:11.6.2018.5:45p}).

Second,
we claim that every $a_x \in I_x$ and $b_y \in I_y$ satisfy 
\begin{equation}
\label{eqn:11.9.2018.7:07p}  
N_{\vec{G}}^-(a_x) \cap N^+_{\vec{G}}(b_y)  \neq \emptyset.   
\end{equation}  
Indeed, $(b_y, a_x) \not\in \vec{E}$ lest 
$(x, y, b_y, 
a_x, x)$ is a directed 4-cycle containing both $x$ and $y$.  
Thus, 
$V \setminus (I_x \cup I_y)$ contains each of 
$N^-_{\vec{G}}(a_x)$ and $N^+_{\vec{G}}(b_y)$, and hence their union.
As such, 
\begin{multline}  
\label{eqn:11.9.8:10p}  
\big|N^-_{\vec{G}}(a_x) \cap N^+_{\vec{G}}(b_y)\big|
= 
\big|N^-_{\vec{G}}(a_x)\big| + \big|N^+_{\vec{G}}(b_y)\big|
- 
\big|N^-_{\vec{G}}(a_x) \cup N^+_{\vec{G}}(b_y)\big|
\geq 
\big|N^-_{\vec{G}}(a_x)\big| + \big|N^+_{\vec{G}}(b_y)\big|
+ |I_x \cup I_y| - n   \\
= 
\big|N^-_{\vec{G}}(a_x)\big| + \big|N^+_{\vec{G}}(b_y)\big|
+ |I_x| + |I_y| - n   
\stackrel{\text{\tiny 
Fct.\ref{fact:Clm19.1}}}{\geq}  
\big|N^-_{\vec{G}}(a_x)\big| + \big|N^+_{\vec{G}}(b_y)\big|
+ 
\big|N^-_{\vec{G}}(x)\big| + \big|N^+_{\vec{G}}(y)\big| - 14 \eps n - n   \\
\geq 
4 \delta_0\big(\vec{G}\big) - 14\eps n - n  
\geq 
4 \left(\tfrac{1}{3} - \eps \right)n - 14 \eps n - n 
= 
\left(\tfrac{1}{3} - 18 \eps \right)n 
> 
0.     
\end{multline}

Third and finally, 
we observe that $N^+_{\vec{G}}(x) \cap I_y = \emptyset = N^-_{\vec{G}}(y) \cap I_x$.  Indeed, and for example, 
any 
$b_y \in N^+_{\vec{G}}(x) \cap I_y$ 
and $a_x \in I_x$ beget $c_{xy} \in 
N_{\vec{G}}^-(a_x) \cap N^+_{\vec{G}}(b_y)$   
by~(\ref{eqn:11.9.2018.7:07p}), in which case 
$(x, b_y, c_{xy}, a_x, x)$
would be a directed 4-cycle containing $x$. 
Now, $V \setminus (I_x \cup I_y)$ contains 
each of $N^+_{\vec{G}}(x)$ and $N^-_{\vec{G}}(y)$, 
and so 
$$
\big|N^+_{\vec{G}}(x) \cap N^-_{\vec{G}}(y)\big| \geq 
\big|N^+_{\vec{G}}(x)\big| +  \big|N^-_{\vec{G}}(y)\big| + |I_x \cup I_y| - n,  
$$
where calculations identical to~(\ref{eqn:11.9.8:10p})   
establish~(\ref{eqn:11.9.2018.7:50p}).    

\subsection*{Proof of Fact~\ref{prop:2.3}}  
Let $\lambda > 0$ be given.  
The promised constant 
$\eps = \eps(\lambda) > 0$ 
will be defined in context.  
Let $\vec{G} = (V, \vec{E})$ be an oriented graph on 
$n \geq n_0(\lambda, \eps)$ many vertices which satisfies 
$\delta_0(\vec{G}) \geq ((1/3) - \eps)n$.  
We show that $\vec{G}$ is $\lambda$-extremal when Conditions~(1) or~(2) hold, which 
we handle separately.  

For Condition~(1), 
it suffices to take $\eps \in (0, \lambda/7]$.
Let $V = V_0 \cup V_1 \cup V_2$
be a partition satisfying $|V_1|, |V_2| \geq ((1/3) - \eps)n$ and 
$e_{\vec{G}}(V_1), e_{\vec{G}}(V_2), e_{\vec{G}}(V_2, V_1) \leq \eps n^2$.  
We bound each of $e_{\vec{G}}(V_2, V_0)$, $e_{\vec{G}}(V_0, V_1)$, 
and $e_{\vec{G}}(V_1, V_2)$  
suitably 
from below.  
First, 
our hypotheses give 
\begin{equation}
\label{eqn:10.26.4:35p}  
e_{\vec{G}}(V_2, V_0) \geq \Big(\sum_{v_2 \in V_2} \deg^+_{\vec{G}}(v_2) \Big)
- e_{\vec{G}}(V_2) - e_{\vec{G}}(V_2, V_1) 
\geq 
\left(\tfrac{1}{3} - \eps\right)^2n^2 - 2 \eps n^2 
\geq \left(\tfrac{1}{9} - \lambda\right)n^2,  
\end{equation}  
where we used $3\eps \leq \lambda$.  
Second, and similarly, 
$$
e_{\vec{G}}(V_0, V_1) \geq \Big(\sum_{v_1 \in V_1} \deg^-_{\vec{G}}(v_1) \Big) - 
e_{\vec{G}}(V_1) - e_{\vec{G}}(V_2, V_1)  
\geq 
\left(\tfrac{1}{3} - \eps\right)^2n^2 - 2 \eps n^2 
\geq \left(\tfrac{1}{9} - \lambda \right)n^2, 
$$
where we again used $3\eps \leq \lambda$.  
Note that, 
since $\vec{G}$ is oriented, our hypotheses 
and~(\ref{eqn:10.26.4:35p}) give   
\begin{equation}
\label{eqn:10.26.4:40p}  
e_{\vec{G}}(V_0, V_2) \leq |V_0||V_2| - 
e_{\vec{G}}(V_2, V_0)  
\leq 
\big(\tfrac{n - |V_1|}{2}\big)^2 - 
e_{\vec{G}}(V_2, V_0)  
\leq 5\eps n^2.  
\end{equation}    
Third, our hypotheses 
and~(\ref{eqn:10.26.4:40p}) give
$$
e_{\vec{G}}(V_1, V_2) \geq \Big(\sum_{v_2 \in V} \deg^-_{\vec{G}}(v_2)\Big)
- e_{\vec{G}}(V_2) - e_{\vec{G}}(V_0, V_2)
\geq 
\left(\tfrac{1}{3} - \eps\right)^2 n^2 - 6\eps n^2 
\geq 
\left(\tfrac{1}{9} - \lambda\right) n^2, 
$$  
where we used $7\eps \leq \lambda$.

For Condition~(2), 
we 
consider a suitably small $\gamma \in (0, \lambda/3]$ in context,  
and we take $\eps = \eps(\gamma) > 0$
according to an application of the Erd\H{o}s-Stone theorem~\cite{ErdosStone}, 
discussed below.  
Assume that $\vec{G} = (V, \vec{E})$ has 
no transitive triangles.  Then the underlying graph $G = (V, E)$
(obtained by removing orientations on arcs) is $K_4$-free.  
By our hypothesis, 
$$
\big|\vec{E}\big|
= \sum_{v\in V} \deg_{\vec{G}}^+(v) \geq n \delta_0\big(\vec{G}\big) \geq
\left(\tfrac{1}{3} - \eps\right)n^2,   
$$
and so altogether the underlying graph 
$G = (V, E)$ is $K_4$-free with 
$|E| \geq ((1/3) - \eps)n^2$ 
many edges.  
As such, 
the Erd\H{o}s-Stone theorem~\cite{ErdosStone} 
guarantees 
a partition $V = V_0 \cup V_1 \cup V_2$, where 
$|V_0| \leq |V_1| \leq |V_2| \leq |V_0| + 1$, and where 
each $0 \leq i < j \leq 2$ satisfies 
\begin{equation}
\label{eqn:7:31p}  
|E(G[V_i, V_j])| \geq \left(\tfrac{1}{9} - \gamma\right) n^2.   
\end{equation}  
Then\footnote{See, for example, the proof of 
Fact~\ref{fact:10}.}, the 3-partite graph 
\begin{equation}
\label{eqn:6.9.2019.12:30p}  
\text{$G[V_0, V_1] \cup G[V_1, V_2] \cup G[V_2, V_0]$ 
admits at least $((1/(27)) - \mu)n^3$  
many triangles $K_3$,}
\end{equation}   
where $\mu = \mu(\gamma) \to 0$ as $\gamma \to 0$.  
Since $\vec{G}$ has no transitive triangles, 
every triangle of $G$ corresponds to a 
directed 3-cycle $\vec{C}_3$ in $\vec{G}$.  Among other conclusions, we 
will show that almost all of the triangles 
of~(\ref{eqn:6.9.2019.12:30p})   
are commonly oriented, in one of the following two senses.  
We say that 
a directed 3-cycle $\vec{C}_3$ of $\vec{G}$ 
is 
{\it positively oriented}
when all of its arcs are among 
$(V_0 \times V_1) \cup (V_1 \times V_2) \cup (V_2 \times V_0)$, and we say that it 
is 
{\it negatively oriented}
when all of its arcs are among $(V_0 \times V_2) \cup (V_2 \times V_1) \cup (V_1 \times V_0)$.  

We average~(\ref{eqn:6.9.2019.12:30p}) over,   
say $V_0$, 
to obtain 
a vertex $\bar{v}_0 \in V_0$ belonging to at least $((1/9) - 2\mu) n^2$ many directed 3-cycles   
$\vec{C}_3$ of $\vec{G}$.   
At least half of these directed 3-cycles 
are commonly oriented, so w.l.o.g.~assume that at least half are positively
oriented.
Then 
\begin{multline}  
\label{eqn:6.9.2019.12:37p}
\text{$\bar{v}_0 \in V_0$ belongs to 
at least 
$((1/9) - 2\mu)n^2$ many directed 3-cycles $\vec{C}_3$ of $\vec{G}$, and in particular}  \\
\text{$\bar{v}_0 \in V_0$ belongs to 
at least 
$((1/(18)) - \mu)n^2$ many positively oriented 3-cycles $\vec{C}_3$ of $\vec{G}$.  }  
\end{multline}  
From~(\ref{eqn:6.9.2019.12:37p}), we 
will prove that 
\begin{equation}
\label{eqn:6:05p}
e_{\vec{G}}(V_1, V_2) \geq \left(\tfrac{1}{9} - \lambda\right)n^2  
\end{equation}  
follows.  
Indeed, 
$\vec{G}$ has no transitive triangles, 
so 
each of 
$$
K\big[
N^+_{\vec{G}}(\bar{v}_0) \cap V_1, 
N^+_{\vec{G}}(\bar{v}_0) \cap V_2\big]  
\qquad  
\text{and}
\qquad 
K\big[N^-_{\vec{G}}(\bar{v}_0) \cap V_1,  
N^-_{\vec{G}}(\bar{v}_0) \cap V_2\big]   
$$
is edge-disjoint from $E$.  Consequently, 
(\ref{eqn:7:31p})   
gives that 
each has size at most 
\begin{multline*}  
|V_1||V_2| - 
\left(\tfrac{1}{9} - \gamma\right)n^2  
\leq
\big(\tfrac{n-|V_0|}{2}\big)^2 - 
\left(\tfrac{1}{9} - \gamma\right)n^2  
\leq 2 \gamma n^2  \\
\implies \qquad 
\big|N^+_{\vec{G}}(\bar{v}_0) \cap V_1\big| \leq \sqrt{2\gamma} n 
\quad \text{or} \quad 
\big|N^+_{\vec{G}}(\bar{v}_0) \cap V_2\big| \leq \sqrt{2\gamma} n,  \\   
\text{and} \qquad 
\big|N^-_{\vec{G}}(\bar{v}_0) \cap V_1\big| \leq \sqrt{2\gamma} n  
\quad \text{or} \quad 
\big|N^-_{\vec{G}}(\bar{v}_0) \cap V_2\big| \leq \sqrt{2\gamma} n.    
\end{multline*}  
By~(\ref{eqn:6.9.2019.12:37p}), 
it must be that 
both 
$\big|N^+_{\vec{G}}(\bar{v}_0) \cap V_2\big| \leq \sqrt{2\gamma} n$ and 
$|N^-_{\vec{G}}(\bar{v}_0) \cap V_1| \leq \sqrt{2\gamma}n$ hold.  As such, 
$\bar{v}_0 \in V_0$ belongs to at most $2\gamma n^2$ many negatively oriented triangles, and so~(\ref{eqn:6.9.2019.12:37p}) may be updated to say that 
$\bar{v}_0 \in V_0$ belongs to at least $((1/9) - 2\mu - 2\gamma)n^2$ many positively oriented
triangles.  As such,  
$$
e_{\vec{G}}(V_1, V_2) \geq \big(\tfrac{1}{9} - 2\mu - 2\gamma\big)n^2  
\geq \big(\tfrac{1}{9} - \lambda\big) n^2 
$$ 
holds by taking 
$2\mu + 2 \gamma \leq \lambda$, and renders~(\ref{eqn:6:05p}).

The argument above shows that, for each $i \in \mathbb{Z}_3$, 
\begin{enumerate}
\item[$(a_i)$]  
either 
$e_{\vec{G}}(V_i, V_{i+1}) \geq ((1/9) - \lambda) n^2$, 
\item[$(b_i)$]  
or 
$e_{\vec{G}}(V_{i+1}, V_i) \geq ((1/9) - \lambda) n^2$.  
\end{enumerate}  
These outcomes must be consistent across $i \in \mathbb{Z}_3$, which is to say that either $(a_0)$, $(a_1)$, and $(a_2)$ all hold, or 
$(b_0)$, $(b_1)$, and $(b_2)$ all hold.  Indeed, assuming otherwise $\vec{G}$ would have $\Omega(n^3)$ many transitive triangles, contradicting our hypothesis.  
This proves that $\vec{G}$ is $\lambda$-extremal, as desired.

\section{Proof of Lemma~\ref{lem:ext} - Part 1: Strategy and Coarse Structure}  
\label{sec:ext}  
It suffices to take the promised constant $\lambda_0 > 0$ as 
\begin{equation}
\label{eqn:5.8.2019.1:39p}  
\lambda_0 = \big(\tfrac{1}{32,000}\big)^4.  
\end{equation}  
Now, fix $0 < \lambda \leq \lambda_0$ and fix an 
integer $\ell \geq 4$ which is not divisible by three.  In all that follows, 
we take the integer $n_0 = n_0(\lambda_0, \lambda, \ell)$ to be sufficiently 
large whenever needed.  
The proof of Lemma~\ref{lem:ext} is fairly technical, so we begin by outlining some of its strategy.

\subsection{Initial strategy}  
\label{sec:initstrat}  
Let $(G, c)$ be a $\lambda$-extremal edge-colored graph on $n \geq n_0$ many vertices.  
Recall that the hypotheses in Statements~(1) and~(2) 
of Lemma~\ref{lem:ext} 
assume 
\begin{equation}
\label{eqn:overarching2}  
\delta^c(G) \geq 
\left\{
\begin{array}{cc}
(n+5)/3 & \qquad \text{in Statement~(1)},  \\ 
(n+4)/3 & \qquad \text{in Statement~(2).}
\end{array}
\right.
\end{equation}  
If $(G, c)$ admits a rainbow $\ell$-cycle $C_{\ell}$,  
then the conclusions of Lemma~\ref{lem:ext} hold, so 
\begin{equation}
\label{eqn:overarching1}  
\text{{\it we assume throughout this proof that $(G, c)$ 
admits no rainbow $\ell$-cycles $C_{\ell}$.}}  
\end{equation}  
Moreover, 
\begin{equation}
\label{eqn:overarching3}  
\text{{\it we assume throughout this proof that $(G, c)$ 
is edge-minimal w.r.t.~satisfying 
both~{\rm (\ref{eqn:overarching2})}    
and~{\rm (\ref{eqn:overarching1})}}}.   
\end{equation}  
Observe, for 
example, that~(\ref{eqn:overarching3}) 
implies that $(G, c)$ admits no monochromatic paths $P$ 
or cycles $C$ on three or more edges, lest removing an internal edge $\{x, y\} \in E$
from $P$ or $C$ 
lowers neither $\deg^c_G(x)$ nor $\deg^c_G(y)$.    
Finally, 
for both cases of~(\ref{eqn:overarching2}), we set $m = \lfloor n/3 \rfloor$, 
where 
\begin{equation}
\label{eqn:theo23}  
\delta^c(G) \geq \tfrac{n+4}{3} \quad \implies \quad \delta^c(G) \geq 
\big\lfloor \tfrac{n}{3} \big\rfloor + 2 = 
m + 2.  
\end{equation}

Since we assume in Lemma~\ref{lem:ext} that $(G, c)$ is $\lambda$-extremal,   
fix a $\lambda$-extremal partition $V = V(G) = V_0 \cup V_1 \cup V_2$ 
of $(G, c)$ 
(recall Definition~\ref{def:lambdaext}).    
Our first main goal in proving Lemma~\ref{lem:ext} is to infer 
from~(\ref{eqn:overarching3}) that $(G,c)$ enjoys {\it nearly cannonical} structure 
on $V_0 \cup V_1 \cup V_2$, in the following sense.     
Let 
$H = K[V_0, V_1, V_2]$
be the complete 3-partite graph with vertex partition $V_0 \cup V_1 \cup V_2$, 
and consider an edge-coloring $\kappa$ on $H$ where,   
for each $i \in \mathbb{Z}_3$ and for each $v_i \in V_i$, 
\begin{enumerate}
\item[$(a)$]
$\kappa$ assigns 
all distinct colors
to the edges $\{v_i, v_{i+1}\} \in E$, where $v_{i+1} \in V_{i+1}$; 
\item[$(b)$]  
$\kappa$ assigns a common color to all the edges $\{v_i, v_{i-1}\} \in E$, where $v_{i-1} \in V_{i-1}$. 
\end{enumerate}  
We say that any such $(H, \kappa)$
is {\it cannonical} w.r.t.~$V_0 \cup V_1 \cup V_2$.  
(For example, a cannonical edge-colored graph $(H, \kappa)$ 
was used in Case~1 of Section~1.1, where $\kappa = c_+$ 
was defined in~(\ref{eqn:4.6.2019.12:04p}).)    
In the immediate sequel, we 
use~(\ref{eqn:overarching3}) to 
prove that $(G, c)$ is {\it nearly cannonical} w.r.t.~$V_0 \cup V_1 \cup V_2$, 
in the sense that  
for each $i \in \mathbb{Z}_3$, 
{\sl almost all}   
$v_i \in V_i$ admit distinctly colored edges to {\sl almost all} $v_{i+1} \in V_{i+1}$, and 
{\sl almost all}   
$v_i \in V_i$ admit commonly colored edges to {\sl almost all} $v_{i-1} \in V_{i-1}$.  
We now make these details precise.

\subsection{$\boldsymbol{(G, c)}$ is nearly cannonical: getting started}  
Definition~\ref{def:lambdaext} ensures that each $i \in \mathbb{Z}_3$ satisfies 
\begin{equation}
\label{eqn:sizeVi}  
|V_i| = \big(\tfrac{1}{3} \pm 3 \sqrt{\lambda} \big) n.  
\end{equation}  
Indeed, $G$ has at least $(1/3 - 3 \lambda)n^2$ many edges, so 
$\binom{V_0}{2} \cup \binom{V_1}{2} \cup \binom{V_2}{2}$
consists of at most $((1/6) + 3\lambda)n^2$ many pairs.  
Set $|V_i| = ((1/3) + e_i)n$, $i \in \mathbb{Z}_3$, so that 
$e_0 + e_1 + e_2 = 0$.  
Then 
$\binom{V_0}{2} \cup \binom{V_1}{2} \cup \binom{V_2}{2}$ has size 
$$
(1 +o(1)) \big(\tfrac{n^2}{6} + \tfrac{n^2}{2} \big(e_0^2 + e_1^2 + e_2^2\big) \big),  
$$
and this 
is too large when $\max\{|e_0|, |e_1|, |e_2| \} > \sqrt{6\lambda}$.

Next, fix $i \in \mathbb{Z}_3$.  
We shall say that a vertex $v_i \in V_i$ is an {\it $i$-good} vertex if 
\begin{equation}
\label{eqn:gooddeg}  
\deg_G^c (v_i, V_{i+1}) \geq |V_{i+1}| - \lambda^{1/4} n 
\qquad \text{and} \qquad 
\deg_G (v_i, V_{i-1}) \geq |V_{i-1}| - \lambda^{1/4} n,  
\end{equation}  
where as usual $\deg_G(v_i, V_{i-1})$ denotes the number of neighbors of $v_i$ in $V_{i-1}$, and where here
$\deg_G^c(v_i, V_{i+1})$ denotes the number of colors seen on the edges of $v_i$ to $V_{i+1}$.  
Then~(\ref{eqn:gooddeg}) says an $i$-good vertex $v_i$ 
admits distinctly colored edges to all but $\lambda^{1/4} n$ many vertices
$v_{i+1} \in V_{i+1}$, and it admits edges of varying colors to all but $\lambda^{1/4} n$ 
many vertices 
$v_{i-1} \in V_{i-1}$. 
Let $V_i^{\good}$ denote the set of $i$-good vertices $v_i \in V_i$.     
Using~(\ref{eqn:sizeVi})
and Definition~\ref{def:lambdaext}, 
it is easy to show  
that 
\begin{equation}
\label{eqn:Clm13.2B}  
\big|V_i^{\good}\big| \geq |V_i| - 24 \lambda^{1/4} n. 
\end{equation}  
With $i \in \mathbb{Z}_3$ still fixed, 
we shall say that 
a vertex $v_i \in V_i \setminus V_i^{\good}$ is an {\it $i$-bad
vertex}.  We 
write 
$V_i^{\bad} = V_i \setminus V_i^{\good}$ for the set of $i$-bad vertices, 
and 
we write 
$V^{\bad} = V_0^{\bad} \cup V_1^{\bad} \cup V_2^{\bad}$ for the set of 
{\it bad vertices}.  
Then 
bad vertices total 
at most $72\lambda^{1/4} n$  
by~(\ref{eqn:Clm13.2B}).

We  now 
alter the partition $V = V_0 \cup V_1 \cup V_2$ to 
$V = U_0 \cup U_1 \cup U_2$,  
as follows.  
For each $i$-good vertex $v_i \in V_i^{\good}$, we put $v_i \in U_i$.   
For each bad vertex $v \in V^{\bad}$, let $j_v \in \mathbb{Z}_3$ achieve 
\begin{equation}
\label{eqn:badvertexsplit}  
\deg^c_G\big(v_i, V^{\good}_{j_v}\big) =  
\max \Big\{\deg_G^c\big(v, V_0^{\good}\big), \, \deg_G^c\big(v, V_1^{\good}\big), \, 
\deg_G^c\big(v, V_2^{\good}\big) \Big\}.     
\end{equation}  
We then put $v \in U_{j_v-1}$.  Then $U_i$ consists of 
$V_i^{\good}$ 
together with those 
bad
vertices $v \in V^{\bad}$ satisfying
\begin{equation}  
\label{eqn:badvertexsplitprime}  
\deg_G^c\big(v, V_{i+1}^{\good}\big) \geq 
\max\Big\{
\deg_G^c\big(v, V_{i}^{\good}\big), \,  
\deg_G^c\big(v, V_{i-1}^{\good}\big)\Big\}.   
\end{equation}  
We write $U_i^{\good} = U_i \cap V_i^{\good} = V_i^{\good}$, 
and we maintain that these vertices are {\it good}.   
We write
$U_i^{\bad} = U_i \cap V^{\bad}$, and we maintain 
that these vertices are {\it bad}.   
Then 
(\ref{eqn:sizeVi})--(\ref{eqn:Clm13.2B}) give:      
\begin{multline}  
\label{eqn:Clm13.2}  
|U_i| = \big(\tfrac{1}{3} \pm 75 \lambda^{1/4} \big)n, \quad 
\big|U_i^{\good}\big| = \big(\tfrac{1}{3} \pm 75 \lambda^{1/4} \big)n, \quad 
\big|V^{\bad}\big| = 
\big|U^{\bad}_0\big|  
+ 
\big|U^{\bad}_1\big|  
+
\big|U^{\bad}_2\big|  
\leq 72 \lambda^{1/4} n, \\
\forall \ i \in \mathbb{Z}_3, \ 
\forall \ u \in U_i^{\good}, \ 
\deg_G^c(u, U_{i+1}) \geq |V_{i+1}| - 73\lambda^{1/4} n \geq |U_{i+1}| - 145 \lambda^{1/4} n, \\
\forall \ i \in \mathbb{Z}_3, \ 
\forall \ u \in U_i^{\good}, \ 
\deg_G^c\big(u, U_{i+1}^{\good}\big) \geq 
\big(\tfrac{1}{3} - 76 \lambda^{1/4}\big)n, \\  
\forall \ i \in \mathbb{Z}_3, \ 
\forall \ u \in U_i^{\good}, \ 
\deg_G(u, U_{i-1}) \geq |V_{i-1}| - 73 \lambda^{1/4} n 
\geq |U_{i-1}| - 145 \lambda^{1/4} n, \\
\forall \ i \in \mathbb{Z}_3, \ 
\forall \ u \in U_i^{\good}, \ 
\deg_G\big(u, U_{i-1}^{\good}  \big) \geq 
\big(\tfrac{1}{3} - 76 \lambda^{1/4}\big)n, \\  
\forall \ i \in \mathbb{Z}_3, \, 
\forall \ u \in U_i^{\bad}, \, 
\deg_G^c(u, U_{i+1}) 
\stackrel{(\ref{eqn:badvertexsplit})}{\geq}    
\tfrac{1}{3} \delta^c(G) 
- 72\lambda^{1/4} n  
\stackrel{(\ref{eqn:theo23})}{\geq} \big(\tfrac{1}{9} - 72\lambda^{1/4}\big)n.  
\end{multline}  
Henceforth, the initial partition $V = V_0 \cup V_1 \cup V_2$
is largely usurped by $V = U_0 \cup U_1 \cup U_2$.

\subsection{$\boldsymbol{(G, c)}$ is nearly cannonical: a next step}  
The inequalities in~(\ref{eqn:Clm13.2}) show that 
$|U_0|, |U_1|, |U_2|$ are nearly balanced, and that   
$G[U_0, U_1, U_3]$ differs
from the complete
3-partite graph $K[U_0, U_1, U_2]$
on few edges.   
The inequalities in~(\ref{eqn:Clm13.2}) also show 
that 
$(G, c)$ 
deviates very little from 
property~$(a)$  
of Section~\ref{sec:initstrat}, in that good vertices $u_i \in U_i^{\good}$   
(which are pervasive) have distinctly colored edges to nearly all $u_{i+1} \in U_{i+1}$.  
We now show that $(G, c)$ deviates little from the corresponding property~$(b)$.    
For that, we first  
show that good vertices $u_i \in U_i^{\good}$ are incident
to few colors $c(\{u_i, u_{i-1}\})$, where 
$u_{i-1} \in U_{i-1}^{\good}$.

\begin{fact}
\label{fact:2.17.2019.5:52p}  
For each $i \in \mathbb{Z}_3$ and for each $u_i \in U_i^{\good}$, we have 
$\deg_G^c(u_i, U_{i-1}^{\good}) \leq 160 \lambda^{1/4}  n$.  
\end{fact}

\begin{proof}[Proof of Fact~\ref{fact:2.17.2019.5:52p}]
Assume for contradiction 
that Fact~\ref{fact:2.17.2019.5:52p} is false for some index $i \in \mathbb{Z}_3$
and vertex $u_i \in U_i^{\good}$, and w.l.o.g.~assume $i = 2$.
We will first determine a set $T_1^{\good} \subseteq U_1^{\good}$ so that the fixed good
vertex 
$u_2 \in U_2^{\good}$ satisfies 
\begin{equation}
\label{eqn:2.17.2019.6:03p}  
\deg_G^c\big(u_2, T_1^{\good}\big) > 0 \text{ and 
every path $(u_2, u_1, v)$ in $G$ where $u_1 \in T_1^{\good}$ is rainbow}.  
\end{equation}  
To prove~(\ref{eqn:2.17.2019.6:03p}),   
we distinguish $\deg_G^c(u_2)$.   \\

\noindent {\bf Case 1 ($\deg_G^c(u_2) \geq (n+10)/3$).}  
Set $T_1^{\good} = U_1^{\good}$, where our contrary assumption gives 
\begin{equation}
\label{eqn:2.17.2019.6:38p}  
\deg_G^c(u_2, T_1^{\good}) = 
\deg_G^c(u_2, U_1^{\good}) > 160 \lambda^{1/4}  n  > 0,   
\end{equation}  
as desired.  
Now, every path $(u_2, u_1, v)$ with 
$u_1 \in T_1^{\good} = U_1^{\good}$ is rainbow  
lest the hypothesis of Case~1 gives that 
removing the edge
$\{u_2, u_1\} \in E$ from $G$
contradicts~(\ref{eqn:overarching3}).      \hfill $\Box$  \\

\noindent {\bf Case 2 ($\deg_G^c(u_2) < (n+10)/3$).}  
Set 
$$
T_1^{\good} = \big\{ u_1 \in U_1^{\good}: \ \exists \ u_0 \in U_0
\text{ where $c(\{u_2, u_0\}) = c(\{u_2, u_1\})$} \big\}.  
$$
Observe that 
$$
\deg_G^c(u_2) \geq \deg_G^c(u_2, U_0) 
+ \deg_G^c\big(u_2, U_1^{\good}\big)
- \deg_G^c\big(u_2, T_1^{\good}\big), 
$$
and so 
\begin{multline*}  
\deg_G^c\big(u_2, T_1^{\good}\big) 
\geq 
\deg_G^c(u_2, U_0) 
+ \deg_G^c\big(u_2, U_1^{\good}\big)
- \deg_G^c(u_2)   \\
\stackrel{(\ref{eqn:Clm13.2})}{\geq}    
\big(\tfrac{1}{3} - 76 \lambda^{1/4}\big)n  
+ \deg_G^c\big(u_2, U_1^{\good}\big)
- \deg_G^c(u_2)  
\stackrel{(\ref{eqn:2.17.2019.6:38p})}{>}    
\big(\tfrac{1}{3} - 76 \lambda^{1/4}\big)n  
+ 160 \lambda^{1/4} n 
- \deg_G^c(u_2)  \\
\stackrel{\text{{\tiny Case~2}}}{>}  
\big(\tfrac{1}{3} - 76 \lambda^{1/4}\big)n  
+ 160\lambda^{1/4} n 
- \tfrac{n+10}{3}  
= 
84\lambda^{1/4}n - O(1) > 0.   
\end{multline*}  
If $(u_2, u_1, v)$ is a monochromatic path with $u_1 \in T_1^{\good}$, 
then there exists $u_0 \in U_0$ where 
$(u_0, u_2, u_1, v)$ is monochromatic.  
Whether or not $v = u_0$,   
removing the edge
$\{u_2, u_1\} \in E$ from $G$  
contradicts~(\ref{eqn:overarching3}).  \hfill $\Box$

We now use~(\ref{eqn:2.17.2019.6:03p})   
to complete the proof of Fact~\ref{fact:2.17.2019.5:52p}.  
Fix an arbitrary vertex $u_1 \in T_1^{\good}$
(cf.~(\ref{eqn:2.17.2019.6:03p})),    
and fix an arbitrary vertex $u_0 \in N_G(u_1, U_0^{\good})$ (cf.~(\ref{eqn:Clm13.2})).  
By~(\ref{eqn:2.17.2019.6:03p}), 
the path $(u_2, u_1, u_0)$ is rainbow.    
We now distinguish the cases $\ell \equiv 1, 2$ (mod 3).  \\

\noindent {\bf Case~A ($\ell \equiv 1$ {\rm (mod 3)}).}  
Using~(\ref{eqn:Clm13.2}) and $n \geq n_0(\ell)$ sufficiently large, we can easily  
extend the rainbow path $(u_2, u_1, u_0)$ to 
a rainbow path 
$R_{\ell - 4} = 
(u_2, u_1, u_0, v_1, v_2, \dots, v_{\ell - 4})$
on $\ell-1$ vertices, 
where for each $1 \leq j \leq \ell
- 4$, 
we may choose 
$v_j \in U_J^{\good}$ 
for 
$J \equiv j$ (mod 3).  
Our contrary assumption gives that 
$u_2 \in U_2^{\good}$ sees
$160 \lambda^{1/4} n$ colors into $U_1^{\good}$, and at most $\ell - 1$ of them
were used on $R_{\ell - 4}$.   
Similarly, 
(\ref{eqn:Clm13.2})
gives that the vertex $v_{\ell - 4} \in U_0^{\good}$ (recall $\ell \equiv 1$ (mod 3)) 
sees 
$((1/3) - 76 \lambda^{1/4})n$ colors 
into $U_1^{\good}$, and at most $\ell - 1$ of them were used on $R_{\ell - 4}$.  
Then some $w_1 \in N_G(u_2, U^{\good}_1) \cap N_G(v_{\ell - 4}, U^{\good}_1)$
extends $R_{\ell-4}$ to a rainbow $\ell$-cycle $C_{\ell}$ 
as 
$$
160 \gamma^{1/4} n - (\ell - 1) + \big(\tfrac{1}{3} - 76 \lambda^{1/4} \big) n - (\ell - 1)
\geq \big(\tfrac{1}{3} + 84\lambda^{1/4}\big)n  - O(1)  
\stackrel{(\ref{eqn:Clm13.2})}{>}
|U_1|.
$$
\hfill $\Box$ \\


\noindent {\bf Case~B ($\ell \equiv 2$ {\rm (mod 3)}).}  
The rainbow
path $(u_2, u_1)$ may be extended to a rainbow path $\hat{R}_{\ell - 2} = 
(u_2, u_1, v_2, \dots, 
v_{\ell - 2})$ on $\ell-1$ vertices, where for each $2 \leq j \leq \ell - 2$, 
we may choose $v_j \in U_J^{\good}$ for $J
\equiv j$ (mod 3).  
Identically to the above, we may extend
the rainbow path 
$\hat{R}_{\ell - 2}$ to a rainbow 
$\ell$-cycle $C_{\ell}$.   
\end{proof}

\subsection{$\boldsymbol{(G, c)}$ is nearly cannonical: finale}  
We now show that, for a fixed $u_i \in U_i^{\good}$, 
edges $\{u_i, u_{i-1}\} \in E$, where $u_{i-1} \in 
U_{i-1}^{\good}$, 
are dominated by a single color.

\begin{prop}
\label{prop:Clm13.3}  
For each $i \in \mathbb{Z}_3$ and for each $u_i \in U_i^{\good}$, there exists a color
$c_{u_i}$ from $c$ where 
all but $161 \lambda^{1/4} n$ many vertices $u_{i-1} \in U_{i-1}^{\good}$  
satisfy $\{u_i, u_{i-1}\} \in E$ and $c(\{u_i, u_{i-1}\}) = c_{u_i}$.  
Together 
with~{\rm (\ref{eqn:Clm13.2})}, 
all but $233 \lambda^{1/4} n$ vertices $u_{i-1} \in U_{i-1}$ 
satisfy $\{u_i, u_{i-1}\} \in E$ and $c(\{u_i, u_{i-1}\}) = c_{u_i}$.  
\end{prop}  

For the proof and use of Proposition~\ref{prop:Clm13.3}, we establish some notation.  
Fix $i \in \mathbb{Z}_3$ and fix $u_i \in U_i$.  On the edges 
$E_G(u_i, U_{i-1})$ between $u_i$ and $U_{i-1}$, 
let $c_{u_i}$ be a most frequent color, which we call
the {\it primary} color of $E_G(u_i, U_{i-1})$.    
Edges 
of $E_G(u_i, U_{i-1})$ 
colored by $c_{u_i}$ are called {\it typical} edges, and edges
of $E_G(u_i, U_{i-1})$ 
colored otherwise are called {\it special} edges.  
We write $N_G^{\typ}(u_i, U_{i-1})$ for the set of $u_{i-1} \in U_{i-1}$ 
where $\{u_i, u_{i-1}\} \in E$ is 
a 
typical edge, 
and we write $N_G^{\spec}(u_i, U_{i-1})$ for the set of $u_{i-1} \in U_{i-1}$ 
where $\{u_i, u_{i-1}\} \in E$ is 
a 
special edge.  
We write 
\begin{equation}
\label{eqn:3.3.2019.6:33p}  
\deg_G^{\typ}\big(u_i, U_{i-1}\big) = \big|N_G^{\typ}\big(u_i, U_{i-1}\big)\big| 
\qquad \text{and} \qquad 
\deg_G^{\spec}\big(u_i, U_{i-1}\big) = \big|N_G^{\spec}\big(u_i, U_{i-1}\big)\big|.   
\end{equation}  

\begin{proof}[Proof of Proposition~\ref{prop:Clm13.3}]  
Assume for contradiction that Proposition~\ref{prop:Clm13.3} is false for some index $i \in \mathbb{Z}_3$
and vertex $u_i \in U_i^{\good}$, and w.l.o.g.~assume $i = 2$.  
Then, the fixed vertex $u_2 \in U_2^{\good}$ satisfies
\begin{equation}
\label{eqn:2.19.2019.10:01a}  
\deg^{\spec}_G\big(u_2, U_1^{\good}\big) \geq 161\lambda^{1/4}  n \qquad \text{while} \qquad 
\deg_G^c\big(u_2, U_1^{\good}\big) 
\stackrel{\text{Fact\ref{fact:2.17.2019.5:52p}}}{\leq}
160 \lambda^{1/4} n.   
\end{equation}  
We will produce a contradiction similar to that for 
Fact~\ref{fact:2.17.2019.5:52p}, 
where we will use~(\ref{eqn:2.19.2019.10:01a}) to construct a rainbow $\ell$-cycle $C_{\ell}$
in $(G, c)$, which will   
contradict~(\ref{eqn:overarching1}).   
We again distinguish the cases 
$\ell \equiv 1, 2$ (mod 3).  \\

\noindent {\bf Case 1 ($\ell \equiv 1$ {\rm (mod 3)}).}  
The inequalities in~(\ref{eqn:2.19.2019.10:01a}) together imply that there exist neighbors $u_1 \neq v_1 \in N_G(u_2, U_1^{\good})$  
for which $c(\{u_2, u_1\}) = c(\{u_2, v_1\})$ differs from the primary color $c_{u_2}$.  
For simplicity, let $c_{u_2}$ be {\sl blue} and let 
$c(\{u_2, u_1\}) = c(\{u_2, v_1\})$ 
be {\sl red}.  
Using~(\ref{eqn:Clm13.2}), 
fix $u_0 \neq v_0 \in N_G(u_1, U_0^{\good}) \cap N_G(v_1, U_0^{\good})$.  
Since $(G, c)$ admits no monochromatic paths on four vertices, 
none of the edges of the four-cycle $(u_1, u_0, v_1, v_0)$
can be red, and not all of them can be blue.  W.l.o.g., assume $\{u_1, u_0\}$ is colored 
{\sl yellow} so that $(u_2, u_1, u_0)$ is a red-yellow path    
which avoids the primary color blue for $u_2$.  
Similarly to the proof of Fact~\ref{fact:2.17.2019.5:52p}, 
we will extend $(u_2, u_1, u_0)$ to a rainbow $\ell$-cycle $C_{\ell}$, which will 
contradict~(\ref{eqn:overarching1}).

Consider the following set which will be an eventual `target space':  
$$
T_1(u_2) = \big\{t_1 \in N_G\big(u_2, U_1^{\good}\big): \ c(\{u_2, t_1\}) 
\text{ is neither red nor yellow } \big\} \subseteq U_1^{\good}.   
$$
Since blue is the primary color for $u_2$, some edges $\{u_2, t_1\}$ with $t_1 \in T_1(u_2)$
are colored blue.  
Now, among the colors blue, red, and yellow, neither red nor yellow are primary, so at most a $2/3$ portion of 
neighbors $v_1 \in N_G(u_2, U_1^{\good})$ have red or yellow edges with $u_2$.  Thus, 
\begin{multline}  
\label{eqn:2.19.2019.10:44a}  
\qquad 
\qquad 
|T_1(u_2)| \geq \tfrac{1}{3} \deg_G\big(u_2, U_1^{\good}\big)
\stackrel{(\ref{eqn:Clm13.2})}{\geq}
\tfrac{1}{3}  
\big(\tfrac{1}{3} - 76 \lambda^{1/4} \big) 
n \stackrel{(\ref{eqn:5.8.2019.1:39p})}{\geq} \tfrac{n}{10}, \\    
\text{while $u_2$ sees at most $160 \lambda^{1/4} n$ colors into $T_1(u_2)$
(cf.~Fact~\ref{fact:2.17.2019.5:52p})}.
\qquad 
\qquad 
\end{multline}  
Let $C(u_2)$ be the set of colors used on edges between $u_2$ and $T_1(u_2)$.  
As we did for Fact~\ref{fact:2.17.2019.5:52p}, we extend the rainbow path
$(u_2, u_1, u_0)$ to a rainbow path $R_{\ell - 4}
= (u_2, u_1, u_0, w_1, w_2, \dots, w_{\ell - 4})$ on $\ell - 1$ vertices, 
where for each $1 \leq j \leq \ell - 4$, 
we may choose $w_j \in U_J^{\good}$ for $J \equiv j$ (mod 3), but where this time we avoid
the $|C(u_2)| \leq 160 \lambda^{1/4} n$ many colors of $C(u_2)$, which we may do on 
account of~(\ref{eqn:Clm13.2}).  
Since $w_{\ell-4} \in U_0^{\good}$ (recall $\ell \equiv 1$ (mod 3)), 
(\ref{eqn:Clm13.2}) gives that $\deg_G^c(w_{\ell-4}, U_1^{\good}) \geq |U_1^{\good}| - 
145 \lambda^{1/4} n$, 
so from $T_1(u_2) \subseteq U_1^{\good}$, 
\begin{equation}
\label{eqn:2.20.2019.9:26a}  
\deg_G^c(w_{\ell - 4}, T_1(u_2))
\geq |T_1(u_2)| - 145 \lambda^{1/4}n 
\stackrel{(\ref{eqn:2.19.2019.10:44a})}{\geq} \tfrac{n}{10} - 145 \lambda^{1/4} n
\stackrel{(\ref{eqn:5.8.2019.1:39p})}{>}  
160 \lambda^{1/4} n  + \ell - 1
\stackrel{(\ref{eqn:2.19.2019.10:44a})}{\geq} 
|C(u_2)| + \ell - 1.     
\end{equation}  
Thus, we may choose a neighbor 
$t_1 \in N_G(w_{\ell - 4}) \cap T_1(u_2)$ where $c(\{w_{\ell - 4}, t_1\}) \not\in C(u_2)$  
differs from any color used on $R_{\ell - 4}$.  Now, $(u_2, u_1, u_0, w_1, w_2, 
\dots, w_{\ell - 4}, t_1)$ is a rainbow $\ell$-cycle $C_{\ell}$ in $(G, c)$ 
(where $c(\{u_2, t_1\}) \in C(u_2)$ but where $C(u_2)$ was used 
nowhere else on $C_{\ell}$), 
which contradicts~(\ref{eqn:overarching1}).   \hfill $\Box$  \\

\noindent {\bf Case 2 ($\ell \equiv 2$ {\rm (mod 3)}).}  
The proof is analogous to that above, where we may simplify the preamble of 
Case~1.  Here, fix a single neighbor $u_1 \in N_G(u_2, U_1^{\good})$ where 
$c(\{u_2, u_1\})$ 
(which we assume is red)  
differs from the primary color blue for $u_2$. 
We will extend the rainbow path $(u_2, u_1)$ to a rainbow
$\ell$-cycle
$C_{\ell}$,  
which will contradict~(\ref{eqn:overarching1}).   To do so, this time we define
$$
T_1(u_2) = \big\{t_1 \in N_G\big(u_2, U_1^{\good}\big): \text{ $c(\{u_2, t_1\})$ is not red}
\big\},     
$$
and again we define $C(u_2)$ to be the set of colors on edges between $u_2$ and 
$T_1(u_2)$.  
Since red is not the primary color of $u_2$, at most half the neighbors
$v_1 \in N_G(u_2, U_1^{\good})$ have a red edge with $u_2$, and so 
the final conclusions of~(\ref{eqn:2.19.2019.10:44a}) hold.    
On account of~(\ref{eqn:Clm13.2}), we may extend the rainbow path $(u_2, u_1)$ to a rainbow
path $\hat{R}_{\ell - 2} = (u_2, u_1, v_2, \dots, v_{\ell - 2})$ on $\ell-1$
vertices, where for each 
$2 \leq j \leq \ell - 2$, we may choose $v_j \in 
U_J^{\good}$ for $J \equiv j$ (mod 3), and where again we may avoid the $|C(u_2)|
\leq 160 \lambda^{1/4} n$ many colors of $C(u_2)$.  
The inequality in~(\ref{eqn:2.20.2019.9:26a}) holds for the vertex 
$v_{\ell - 2} \in U_0^{\good}$
(recall $\ell \equiv 2$ (mod 3)),     
so we may choose 
$t_1 \in N_G(v_{\ell - 2}) \cap T_1(u_2)$ where $c(\{v_{\ell - 2}, t_1\}) \not\in C(u_2)$
differs from any color used on $\hat{R}_{\ell - 2}$.  Now, $(u_2, u_1, v_2, \dots, 
v_{\ell - 2}, t_1)$ is a rainbow $\ell$-cycle $C_{\ell}$ in $(G, c)$,   
which contradicts~(\ref{eqn:overarching1}).
\end{proof}

We conclude the nearly cannonical structure of $(G, c)$ by noting that, for each 
$i \in \mathbb{Z}_3$,  distinct
good vertices $u_i \neq v_i \in U_i^{\good}$ admit distinct
primary colors.  

\begin{cor}
\label{cor:Clm13.4}  
For each $i \in \mathbb{Z}_3$ and for each $u_i \neq v_i \in U_i^{\good}$, the primary colors
$c_{u_i}$ and $c_{v_i}$ differ.  
\end{cor}

\begin{proof}[Proof of Corollary~\ref{cor:Clm13.4}]
Fix $i \in \mathbb{Z}_3$ and fix $u_i \neq v_i \in U_i^{\good}$.    
Then 
$$
\big|N^{\typ}_G(u_i, U_{i-1}) \cap N^{\typ}_G(v_i, U_{i-1}) \big| 
\stackrel{\text{Prop.\ref{prop:Clm13.3}}}{\geq}  
|U_{i-1}| - 466 \lambda^{1/4} 
\stackrel{(\ref{eqn:Clm13.2})}{\geq}  
\big(\tfrac{1}{3} - 541 \lambda^{1/4} \big)n \stackrel{(\ref{eqn:5.8.2019.1:39p})}{\geq} 2.   
$$
If $c_{u_i} = c_{v_i}$, then 
any pair from the set above
renders a monochromatic 4-cycle, 
contradicting~(\ref{eqn:overarching3}).    
\end{proof}   

\section{Proof of Lemma~\ref{lem:ext} - Part 2: Strong Cycles and the Case 
$\ell \equiv 2$ (mod 3)}  
\label{sec:ext2}  
Continuing from the previous section, we now prepare to 
prove Lemma~\ref{lem:ext} when $\ell \equiv 2$ (mod 3).  
The central tools of this proof
are 
important observations on so-called {\it strong cycles} in the nearly cannonical
edge-colored graph $(G, c)$.    
Many of these observations will also be important later 
when we prove the case 
$\ell \equiv 1$ (mod 3) of Lemma~\ref{lem:ext}.

\subsection{Strong cycles}  
We say that a cycle 
$C_k = (u_1, \dots, u_k)$ 
(with prescribed vertex $u_1$)  
is a {\it strong} cycle if there exists $i \in \mathbb{Z}_3$ so that 
$u_1 \in U_i^{\good}$ and $u_k \in N^{\typ}_G(u_1, U_{i-1})$.  
We determine conditions under which rainbow or properly colored paths can be
extended to strong rainbow or strong properly colored cycles.

\begin{prop}
\label{prop:Clm13.5}  
Fix integers $1 \leq k < K \leq \ell$ and
fix 
$i, j \in \mathbb{Z}_3$ for which $K - k \equiv (i - 1) - j$ {\rm (mod 3)}.  
Let $P$ be a $(u_i, u_j)$-path on $k$ vertices linking $u_i \in U_i^{\good}$ and $u_j \in U_j$.  
The following statements hold:  
\begin{enumerate}
\item
If $P$ is rainbow and $c_{u_i}$-free, then $P$ may be extended
to a strong rainbow $K$-cycle $C_K$; 
\item  
If $P$ is properly colored
and 
its $u_i$-edge is not $c_{u_i}$-colored,   
then $P$ may be extended to a strong properly colored $K$-cycle $C_K$; 
\item  
When $K \equiv k$ {\rm (mod 3)}    
and $(G, c)$ admits 
a strong rainbow $k$-cycle $C_k$, then $(G, c)$
also 
admits a strong rainbow $K$-cycle $C_K$;   
\item  
When $K \equiv k$ {\rm (mod 3)}  
and $(G, c)$ admits a strong properly colored $k$-cycle $C_k$, then $(G, c)$
also 
admits a strong properly colored $K$-cycle $C_K$.   
\end{enumerate}  
\end{prop}

\begin{proof}[Proof of Proposition~\ref{prop:Clm13.5}]
Let integers $1 \leq k < K \leq \ell$ and elements $i, j \in \mathbb{Z}_3$ be given satisfying
$K - k \equiv (i-1) - j$ (mod 3), and let $P = (u_i, \dots, u_j)$
be a $(u_i, u_j)$-path on $k$ vertices linking $u_i \in U_i^{\good}$ and $u_j \in U_j$.  
To prove Statement~(1), assume that $P = R$ is rainbow and 
$c_{u_i}$-free.  
Similarly to the proofs of Fact~\ref{fact:2.17.2019.5:52p} 
and 
Proposition~\ref{prop:Clm13.3}, 
we will extend $R$ to a $c_{u_i}$-free rainbow path 
$\tilde{R}_{K-k-1} = (u_i, \dots, u_j, v_{j+1}, \dots, v_{j + K-k-1})$ on $K-1$ vertices,  
where for each $j+1 \leq h \leq j + K-k-1$, we may choose $v_{h} \in U^{\good}_{H}$
for $H \equiv h$ (mod 3).   
We begin with the first step, where it is not guaranteed
in our hypothesis
that $u_j \in U_j$ is a good vertex.  If $u_j \in U_j^{\bad}$, 
then 
\begin{equation}
\label{eqn:2.20.2019.5:56p}  
\deg_G^c(u_j, U_{j+1}) 
\stackrel{(\ref{eqn:Clm13.2})}{\geq}   
\big(\tfrac{1}{9} - 72 \lambda^{1/4}\big) n
\qquad 
\stackrel{(\ref{eqn:Clm13.2})}{\implies}  
\qquad 
\deg_G^c\big(u_j, U_{j+1}^{\good}\big) 
\stackrel{(\ref{eqn:Clm13.2})}{\geq}   
\big(\tfrac{1}{9} - 144 \lambda^{1/4}\big) n.  
\end{equation}  
Thus, we may select $v_{j+1} \in N_G(u_j, U_{j+1}^{\good})$ 
for $\tilde{R}_{K-k-1}$ 
while 
avoiding $c_{u_i}$ and the colors of $R$.
If $u_j \in U_j^{\good}$ is a good vertex, 
then the neighborhood 
$N_G(u_j, U_{j+1}^{\good})$ 
is larger still (cf.~(\ref{eqn:Clm13.2})), and
again 
we may select 
$v_{j+1}$ for $\tilde{R}_{K-k-1}$ as described above.  
We select all remaining vertices $v_h$ for $\tilde{R}_{K-k-1}$, where 
$j+2 \leq h \leq j + K - k - 1$, in a similar fashion.
By our hypothesis $K - k \equiv (i-1) - j$ (mod 3), the terminal vertex $v_{j+K-k-1}
\in U_{i-2}^{\good}$ while the initial vertex $u_i \in U_i^{\good}$.   
Comparing Proposition~\ref{prop:Clm13.3}    
and~(\ref{eqn:Clm13.2}), we see 
\begin{equation}
\label{eqn:2.20.2019.5:59p}  
\big|N_G^{\typ} (u_i, U_{i-1})
\cap 
N_G(v_{j + K - k - 1}, U_{i-1})  
\big|
\geq |U_{i-1}| - 306 \lambda^{1/4}n
\stackrel{(\ref{eqn:Clm13.2})}{\geq}  
\big(\tfrac{1}{3} - 381 \lambda^{1/4}\big)n 
\stackrel{(\ref{eqn:5.8.2019.1:39p})}{>} 0, 
\end{equation}  
and so we may select a vertex $u_{i-1}$ from the set above
whose adjacency 
with 
$v_{j + K - k - 1}$ avoids $c_{u_i}$ and the colors of 
$\tilde{R}_{K-k-1}$. Since $c(\{u_i, u_{i-1}\}) = c_{u_i}$ is the primary color
of $u_i$, which hasn't yet been used, $C_K = (u_i, \dots, u_j, v_{j+1}, \dots, v_{j+K-k-1}, 
u_{i-1})$ is a strong rainbow $K$-cycle.

The proof of Statement~(2) is absolutely the same as that of Statement~(1).  
In particular, for the properly colored $k$-vertex path 
$P = (u_i, \dots, u_j)$ linking
$u_i \in U_i^{\good}$ and $u_j \in U_j$ whose $u_i$-edge is not $c_{u_i}$-colored,   
the proof above allows 
the segment $(u_j, v_{j+1}, \dots, v_{j+K-k-1}, u_{i-1})$ 
of 
$\tilde{R}_{K-k-1}$
to be rainbow, $c_{u_i}$-free, and to be free of the colors from $P$.  Thus, 
$C_K = (u_i, \dots, u_j, v_{j+1}, \dots, v_{j+K-k-1}, u_{i-1})$ is a strong properly colored
$K$-cycle.

Statements~(3) and~(4) now follow immediately from Statements~(1) and~(2).  
Indeed, let $C_k = (u_1, \dots, u_k)$ be a strong rainbow or properly colored $k$-cycle
where $u_1 \in U_i^{\good}$ and $u_k \in N^{\typ}(u_1, U_{i-1})$ for some $i \in \mathbb{Z}_3$.  
Ignoring the edge $\{u_1, u_k\}$, the path $P_k = (u_1, \dots, u_k)$ is 
rainbow or properly 
colored, 
where $u_k \in U_{i-1}$ assumes $j = i - 1$.  Taking $K \equiv 
k + (i-1) - (i-1) \equiv k$ (mod 3) and $K \leq \ell$, Statements~(1) or~(2) 
extend $P_k$ to a strong rainbow or properly colored $K$-cycle $C_K$.  
\end{proof}

It will be convenient to have the following corollary of 
Proposition~\ref{prop:Clm13.5} in the case $\ell \equiv 2$ (mod 3).    

\begin{cor}
\label{cor:abc}
Let $\ell \equiv 2$ {\rm (mod 3)} and fix $i \in \mathbb{Z}_3$.   
The following statements hold:  
\begin{enumerate}  
\item 
Each $u_i \in U_i^{\good}$ satisfies $N^{\spec}_G(u_i, U_{i-1}) = \emptyset$; 
\item 
Let $R = (u_i, v, w_i)$ be a rainbow path with $u_i \in U_i^{\good}$ and $w_i \in U_i$.   
Then $c_{u_i}$ appears on $R$.  In particular, 
$c(\{u_i, v\}) = c_{u_i}$
or $(G, c)$ admits 
a properly colored $\ell$-cycle $C_{\ell}$; 
\item
Let $R = (u_i, v, u_{i-1})$ be a rainbow path with $u_i \in U_i^{\good}$ and 
$u_{i-1} \in U_{i-1}^{\good}$.  Then $c_{u_i} = c_{u_{i-1}}$ or 
$c_{u_i}$ or $c_{u_{i-1}}$ appears on $R$.  As well, if 
$c(\{u_i, v\}) \neq c_{u_i}$ and $c(\{u_{i-1}, v\}) \neq c_{u_{i-1}}$, then 
$(G, c)$ admits 
a properly colored $\ell$-cycle $C_{\ell}$; 
\item  
Let $\ell \neq 5$, and let $u_i, v_i \in U_i^{\good}$ and $w_i, x_i \in U_i$ span disjoint
edges $\{u_i, w_i\}, \{v_i, x_i\} \in E(G)$.   
Then $c_{u_i}$, 
$c_{v_i}$, $c(\{u_i, w_i\})$, and $c(\{v_i, x_i\})$ can't all be distinct.  
\item  Let $\ell \neq 5$, and let $T_i \subseteq U_i^{\good}$ be a set with the property
that for all $u_i \in T_i$, there exist 
$v_i \neq w_i \in N_G(u_i, U_i)$ so that $c(\{u_i, v_i\})$, $c(\{u_i, w_i\})$, and 
$c_{u_i}$ are all distinct.  Then $|T_i| \leq 5$.  
\end{enumerate}  
\end{cor}

\begin{proof}[Proof of Corollary~\ref{cor:abc}]  
Let $\ell \equiv 2$ (mod 3) and fix $i \in \mathbb{Z}_3$.  
For Statement~(1), fix
$u_i \in U_i^{\good}$.  If $u_{i-1} \in N_G^{\spec}(u_i, U_{i-1})$, then $\{u_i, u_{i-1}\}$
is a $c_{u_i}$-free rainbow path which   
Proposition~\ref{prop:Clm13.5} 
guarantees can be extended to a strong rainbow $\ell$-cycle $C_{\ell}$
(by setting $j = i - 1$ and $k = 2$, and with $\ell \equiv 2$ (mod 3)),  
which contradicts~(\ref{eqn:overarching1}).    

For Statement~(2), let $R = (u_i, v, w_i)$ be a rainbow path with $u_i \in U_i^{\good}$
and $w_i \in U_i$.  If $R$ is $c_{u_i}$-free, then 
Proposition~\ref{prop:Clm13.5} 
guarantees that $R$ can be extended to a strong rainbow $\ell$-cycle $C_{\ell}$
(by setting $j = i$ and $k = 3$, and with $\ell \equiv 2$ (mod 3)), 
which again contradicts~(\ref{eqn:overarching1}).    
In particular, if $c(\{u_i, v\}) \neq c_{u_i}$, then 
Proposition~\ref{prop:Clm13.5} 
guarantees that $R$ can be extended to a strong properly colored $\ell$-cycle $C_{\ell}$.  

For Statement~(3), let $R = (u_i, v, u_{i-1})$ be a rainbow path with $u_i \in U_i^{\good}$
and $u_{i-1} \in U_{i-1}^{\good}$.  
Assume for contradiction that 
$R$ avoids both 
$c_{u_i} \neq c_{u_{i-1}}$.  
Since $u_{i-1} \in U_{i-1}^{\good}$ is a good vertex, 
Proposition~\ref{prop:Clm13.3} guarantees a vertex $u_{i-2} \in 
N_G(u_{i-1}, 
U_{i-2}^{\good})$
distinct from $v$ for which 
$c(\{u_{i-1}, u_{i-2}\}) = c_{u_{i-1}}$.     
Then the path $S = (u_i, v, u_{i-1}, u_{i-2})$ is rainbow (because $R$ is rainbow
and avoids $c_{u_{i-1}}$), and the path $S$ avoids $c_{u_i}$ (because 
$R$ does and because 
$c_{u_i} \neq c_{u_{i-1}}$).  As such, 
Proposition~\ref{prop:Clm13.5}   
guarantees that $S$ can be extended to a strong rainbow $\ell$-cycle $C_{\ell}$
(by setting $j = i - 2$ and $k = 4$, and with $\ell \equiv 2$ (mod 3)), 
which contradicts~(\ref{eqn:overarching1}).    
In particular, assume $c(\{u_i, v\}) \neq c_{u_i}$ and $c(\{u_{i-1},v\})
\neq c_{u_{i-1}}$.  Then $S = (u_i, v, u_{i-1}, u_{i-2})$ is proper
(because $R$ is rainbow and $c(\{u_{i-1}, v\}) \neq c_{u_{i-1}}$).  
Since 
$c(\{u_i, v\}) \neq c_{u_i}$, 
Proposition~\ref{prop:Clm13.5}   
guarantees that $S$ can be extended to a strong properly colored
$\ell$-cycle $C_{\ell}$.

For Statement~(4), 
let $\ell \neq 5$, and let $u_i, v_i \in U_i^{\good}$ and $w_i, x_i \in U_i$
span disjoint edges $\{u_i, w_i\}, \{v_i, x_i\} \in E(G)$.  Assume, on the contrary, 
that $C = \{c_{u_i}, c_{v_i}, c(\{u_i, w_i\}), c(\{v_i, x_i\})\}$ 
is a set of four distinct colors.  
Fix any $u_{i+1} \in N_G(w_i, U_{i+1}^{\good})$
where the edge $\{w_i, u_{i+1}\} \in E(G)$ is $C$-free
(which is possible by the argument in~(\ref{eqn:2.20.2019.5:56p})).  
Now, fix any
$$
u_{i-1} \in N_G(u_{i+1}, U_{i-1}) \cap N^{\typ}_G(v_i, U_{i-1}) 
$$    
where the edge $\{u_{i-1}, u_{i+1}\} \in E(G)$ is $(C \cup c(\{w_i, u_{i+1}\}))$-free 
(which is possible by the argument in~(\ref{eqn:2.20.2019.5:59p})).    
Now, $(u_i, w_i, u_{i+1}, u_{i-1}, v_i, x_i)$ is a rainbow path avoiding $c_{u_i}$, 
which Proposition~\ref{prop:Clm13.5} guarantees can be extended to a strong rainbow
$\ell$-cycle $C_{\ell}$
(by setting $j = i$ and $k = 6$, and with $\ell \equiv 2$ (mod 3)),   
which again contradicts~(\ref{eqn:overarching1}).

For Statement~(5), 
let $T_i \subseteq U_i^{\good}$ be a set with the property so described, but assume
for contradiction 
that $|T_i| \geq 6$.  
Fix $u_i \in T_i$, where we take $c_{u_i}$ to be {\sl blue}, 
and let $v_i \neq w_i \in N_G(u_i, U_i)$ be guaranteed by the definition of $T_i$, 
where we take 
$c(\{u_i, v_i\})$
to be {\sl red} and $c(\{u_i, w_i\})$ to be {\sl yellow}.  
Since $|T_i| \geq 6$, there exists $x_i \in T_i \setminus \{u_i, v_i, w_i\}$ 
where $c_{x_i}$ is neither
red nor yellow.   
Since $x_i \neq u_i$, Corollary~\ref{cor:Clm13.4}  
guarantees that $c_{x_i} \neq c_{u_i}$ can't be blue, 
so we take $c_{x_i}$ to be {\sl green}.  
Let $y_i \neq z_i \in N_G(x_i, U_i)$ be guaranteed by the definition of $T_i$.  
We now distinguish the extent to which 
$\{u_i, v_i, w_i\}$ and 
$\{x_i, y_i, z_i\}$ overlap.  \\

\noindent {\bf Case 1 ($\{u_i, v_i, w_i\} \cap \{x_i, y_i, z_i\} = \emptyset$).}  
If $c(\{x_i, y_i\})$ is yellow, then $\{u_i, v_i\}$ and 
$\{x_i, y_i\}$ 
violate
Statement~(4) 
above.  
Similarly, if $c(\{x_i, y_i\})$ is red, then $\{u_i, w_i\}$ and 
$\{x_i, y_i\}$ 
violate
the same.   
Assume neither $\{x_i, y_i\}$ nor $\{x_i, z_i\}$ is red or yellow, 
where the definition of $T_i$ 
ensures neither is green.    
At most one of these pairs can be blue, so assume $\{x_i, y_i\}$ is neither
red, yellow, green, nor blue.  Now, $\{u_i, v_i\}$ and $\{x_i, y_i\}$ violate
Statement~(4) above.  \hfill $\Box$  \\

\noindent {\bf Case 2 ($u_i \in \{y_i, z_i\}$).}  
Assume w.l.o.g.~that $u_i = z_i$.  
If $c(\{u_i, x_i\})$ 
is yellow, then $(x_i, u_i, v_i)$ 
violates Statement~(2) above.
If $c(\{u_i, x_i\})$ is not yellow, then 
it is also not green 
by the definition of $T_i$, 
and so 
$(x_i, u_i, w_i)$ violates the same Statement~(2).   
\hfill $\Box$  \\

\noindent {\bf Remark.}  Since 
$x_i \in U_i^{\good} \setminus \{u_i, v_i, w_i\}$, we do not have 
the case $x_i \in \{v_i, w_i\}$.
\hfill $\Box$  \\

\noindent {\bf Case 3 ($u_i \not\in \{y_i, z_i\};$ 
$\{v_i, w_i\} \cap \{y_i, z_i\} \neq \emptyset$).}  
Assume w.l.o.g.~that $w_i = y_i$.  If $c(\{x_i, y_i\})$ is yellow, then 
$\{u_i, v_i\}$ and $\{x_i, y_i\}$
violate Statement~(4) above.  If $c(\{x_i, y_i\})$ is red, 
then $(u_i, w_i = y_i, x_i)$
violates Statement~(2) above.
If $c(\{x_i, y_i\})$ is blue, then $(x_i, y_i = w_i, u_i)$ 
violates Statement~(2) above.   
Otherwise, 
$c(\{x_i, y_i\})$ isn't green
by the definition of $T_i$, 
so 
$\{u_i, v_i\}$
and $\{x_i, y_i\}$
violate Statement~(4) above.  
\end{proof}

\subsection{Proof of Lemma~\ref{lem:ext}:  Statement~(1) when 
$\boldsymbol{\ell \equiv 2}$ (mod 3)}  
Let $\ell \equiv 2$ (mod 3), where $\ell \neq 5$.  
The hypothesis of Statement~(1) of Lemma~\ref{lem:ext} gives that 
$\delta^c(G) \geq (n+5)/3$.    
Assume w.l.o.g.~that 
\begin{equation}
\label{eqn:2.23.2019.12:33p}  
|U_2| \leq |U_1| \leq |U_0|, \qquad \text{in which case} \qquad 
|U_2| 
\leq 
\left\lfloor \tfrac{n}{3} \right\rfloor  
\leq 
\left\lceil \tfrac{n}{3} \right\rceil  
\leq |U_0|.        
\end{equation}  
In the immediate sequel, we motivate the main approach of the proof.  

\subsubsection{Main idea of proof}  
We shall make repeated use of Statement~(5) of Corollary~\ref{cor:abc}, 
for which 
we establish the following notation.  
Fix $i \in \mathbb{Z}_3$ and $u_i \in U_i^{\good}$, and define 
\begin{equation}
\label{eqn:4.27.2019.1:06p}  
c^{\spec}(u_i, U_i) = 
\big\{ c(\{u_i, v_i\}) \not= c_{u_i}: \, v_i \in N_G(u_i, U_i) \big\}  
\end{equation}  
for the set of special (non-primary) 
colors on edges $\{u_i, v_i\} \in E_G(u_i, U_i)$ incident
to $u_i$ in $U_i$.  
By Statement~(1) of Corollary~\ref{cor:abc}, all edges
$\{u_i, u_{i-1}\} \in E_G(u_i, U_{i-1})$ are colored by $c_{u_i}$, 
and so 
\begin{multline}
\label{eqn:4.20.2019.7:05p}  
\big|c^{\spec}(u_i, U_i)\big| \geq \deg_G^c(u_i) - \deg_G^c(u_i, U_{i-1}) - \deg_G^c(u_i, 
U_{i+1})  
= 
\deg_G^c(u_i) - 1 - \deg_G^c(u_i, U_{i+1}) \\
\geq \delta^c(G) - 1 - |U_{i+1}|  
\geq \tfrac{n+5}{3} - 1 - |U_{i+1}|  
= 
\tfrac{n+2}{3} - |U_{i+1}|.  
\qquad 
\qquad 
\end{multline}  
In particular, 
when $i = 1 \in \mathbb{Z}_3$, we infer 
that every $u_1 \in U_1^{\good}$ satisfies 
\begin{equation}
\label{eqn:4.20.2019.6:56p}  
\big|c^{\spec}(u_1, U_1)\big|  
\geq \tfrac{n+2}{3} - |U_2|  
\stackrel{(\ref{eqn:2.23.2019.12:33p})}{\geq}    
\tfrac{n+2}{3} - \left\lfloor \tfrac{n}{3} \right\rfloor.
\end{equation}  
As such, 
if $n \equiv 2$ (mod 3), 
then~(\ref{eqn:4.20.2019.6:56p}) gives 
$|c^{\spec}(u_1, U_1)| \geq 2$ for every $u_1 \in U_1^{\good}$, and so 
$T_1 = U_1^{\good}$ readily 
contradicts Statement~(5) of Corollary~\ref{cor:abc}   
(because $|U_1^{\good}|$ 
from~(\ref{eqn:Clm13.2}) is much too large).    
Similarly, if $|U_2| \leq \lfloor n/3 \rfloor - 1$, 
then~(\ref{eqn:4.20.2019.6:56p}) gives 
$|c^{\spec}(u_1, U_1)| \geq 2$ for every $u_1 \in U_1^{\good}$, giving the same contradiction.  
The main idea of the current proof exploits a similar theme to the instances
$n \equiv 2$ (mod 3) or $|U_2| \leq \lfloor n/3 \rfloor - 1$, which we announce as our goal:  
\begin{multline}  
\label{eqn:4.21.2019.4:18p}  
\qquad 
\qquad 
\text{{\it we seek to determine a large set $T_i \subseteq U_i^{\good}$, 
for some $i \in \mathbb{Z}_3$,}}  \\
\text{{\it 
where every $u_i \in T_i$
satisfies $|c^{\spec}(u_i, U_i)| \geq 2$.}}  
\qquad 
\qquad 
\end{multline}  
When so, 
we contradict Statement~(5) of Corollary~\ref{cor:abc}.

\subsubsection{Supporting details}  From the discussion above, 
it suffices to consider the case $n \not\equiv 2$ (mod 3)
and $|U_2| = \lfloor n/3 \rfloor$.  As such, 
$|U_2| = |U_1| = \lfloor n/3 \rfloor$ and $|U_0| = \lceil n/3 \rceil$.  
Now, for $u_1 \in U_1^{\good}$, we define 
$$
S(u_1) = \big\{v_1 \in N_G(u_1, U_1): \, c(\{u_1, v_1\}) \neq c_{u_1} \big\}.  
$$
We refine 
the partition $U_1 = U_1^{\good} \cup U_1^{\bad}$  
from~(\ref{eqn:Clm13.2}) by subdividing 
$U_1^{\good}$ into 
\begin{equation}
\label{eqn:4.20.2019.8:05p}  
A_1 = \big\{u_1 \in U_1^{\good}: \, S(u_1) \cap U_1^{\bad} \neq \emptyset
\big\}
\qquad \text{and} \qquad 
B_1 = U_1^{\good} \setminus A_1.     
\end{equation}  
We will observe the following fact.

\begin{fact}
\label{fact:u1B1A1}  
Every $u_1 \in B_1$ satisfies $S(u_1) \subseteq A_1$.  
\end{fact}

\begin{proof}[Proof of Fact~\ref{fact:u1B1A1}]    
Fix $u_1 \in B_1$, but assume for contradiction that 
$v_1 \in S(u_1) \cap B_1$.  
Since both $u_1 \neq v_1 \in U_1^{\good}$ are good vertices, 
Corollary~\ref{cor:Clm13.4}  
guarantees that $c_{u_1} \neq c_{v_1}$, where we will take 
$c_{u_1}$ to be {\sl red} and $c_{v_1}$ to be {\sl blue}.  
From $v_1 \in S(u_1)$, we infer that 
$c(\{u_1, v_1\})$ is not $c_{u_1} = \text{ red}$.  
We distinguish two cases.  \\

\noindent {\bf Case 1 ($c(\{u_1, v_1\}) \neq c_{v_1}$).}  
Here, 
we will take $c(\{u_1, v_1\})$ to be {\sl yellow}.  
Proposition~\ref{prop:Clm13.3} guarantees a vertex 
$u_0 \in N_G^{\typ}(v_1, U_0^{\good})$ so that 
$c(\{u_0, v_1\}) = c_{v_1}$ is blue but 
$c_{u_0}$ is neither red, blue, nor yellow.  
We take $c_{u_0}$ to be {\sl green}.  
Now, $R = (u_1, v_1, u_0)$ is a rainbow path where 
$u_1 \in U_1^{\good}$, where $u_0 \in U_0^{\good}$, but where 
neither 
$c_{u_1} \neq c_{u_0}$
(red nor green) 
appear
on $R$, which contradicts Statement~(3) of Corollary~\ref{cor:abc}.
\hfill $\Box$  \\

\noindent {\bf Case 2 ($c(\{u_1, v_1\}) = c_{v_1}$).}  
From~(\ref{eqn:4.20.2019.6:56p}), we infer that $|S(v_1)| \geq 1$, 
where $u_1 \not\in S(v_1)$ on account that $c(\{u_1, v_1\}) = c_{v_1}$
is blue.  
From $v_1 \in B_1$, we infer that $S(v_1) \cap U_1^{\bad} = \emptyset$, 
and so there exists $u_1 \neq w_1 \in S(v_1) \subseteq U_1^{\good}$.  
From $w_1 \in S(v_1)$, we infer that 
$c(\{v_1, w_1\}) \neq c_{v_1}$ is not blue.  
So the path $(u_1, v_1, w_1)$ is rainbow, and
Statement~(2) of Corollary~\ref{cor:abc} implies
that $c(\{v_1, w_1\})$ is $c_{u_1}$ (red).
Since $u_1, v_1, w_1 \in U_1^{\good}$ are good and distinct, 
Corollary~\ref{cor:Clm13.4}  
guarantees that 
the primary colors $c_{u_1}$ (red), 
$c_{v_1}$ (blue), and 
$c_{w_1}$ are distinct, where we take $c_{w_1}$ to be {\sl green}.   
Since $w_1 \in U_1^{\good}$ is good, 
Proposition~\ref{prop:Clm13.3} guarantees a vertex 
$u_0 \in N_G^{\typ}(w_1, U_0^{\good})$ so that $c_{u_0}$ is neither
$c_{v_1}$ (blue), $c(\{u_0, w_1\})$ (green), nor $c(\{v_1, w_1\})$ (red).  
Now, $(v_1, w_1, u_0)$ contradicts Statement~(3) of Corollary~\ref{cor:abc}.  
\end{proof}  

Fact~\ref{fact:u1B1A1} admits the following corollary.  

\begin{cor} 
\label{cor:4.20.2019.7:41p}  
There exist distinct $u_1, v_1, w_1 \in U_1^{\good}$
satisfying 
$S(u_1) \cap S(v_1) \cap S(w_1) \neq \emptyset$.  
\end{cor}

\begin{proof}[Proof of Corollary~\ref{cor:4.20.2019.7:41p}]
Define the auxiliary directed graph 
$\vec{\Gamma} = (U_1, \vec{E})$ by the rule
that for each $(u_1, v_1) \in U_1 \times U_1$, we put 
$(u_1, v_1) \in \vec{E}$ if, and only if, $v_1 \in S(u_1)$.  In this notation, 
$S(u_1) = N^+_{\vec{\Gamma}}(u_1)$.     
We now distinguish
two cases.  \\

\noindent {\bf Case 1 ($|B_1| > 2|A_1|$).}  
For the bipartition $A_1 \cup B_1$  
(cf.~(\ref{eqn:4.20.2019.8:05p})), we infer 
$$
\sum_{a_1 \in A_1} \big|N^-_{\vec{\Gamma}}(a_1) \cap B_1\big|
= 
\sum_{b_1 \in B_1} \big|N^+_{\vec{\Gamma}}(b_1) \cap A_1\big|  
= 
\sum_{b_1 \in B_1} |S(b_1) \cap A_1|    
\stackrel{\text{Fct.\ref{fact:u1B1A1}}}{=}  
\sum_{b_1 \in B_1} |S(b_1)|
\stackrel{(\ref{eqn:4.20.2019.6:56p})}{\geq}  
|B_1| > 2 |A_1|.  
$$
By averaging, there exists 
$\bar{a}_1 \in A_1$ which satisfies
$|N_{\vec{\Gamma}}^-(\bar{a}_1) \cap B_1| \geq 3$, so let 
$b_1, b_1', b_1'' \in N_{\vec{\Gamma}}^-(\bar{a}_1)$.  
Then 
$$
\bar{a}_1 \in
N_{\Gamma}^+(b_1) 
\cap 
N_{\Gamma}^+\big(b_1'\big) 
\cap 
N_{\Gamma}^+\big(b_1''\big) 
= 
S(b_1) \cap S\big(b_1'\big) \cap S\big(b_1''\big), 
$$
and so $S(b_1) \cap S(b_1') \cap S(b_1'') \neq \emptyset$.  \hfill $\Box$  \\

\noindent {\bf Case 2 ($|B_1| \leq 2|A_1|$).}  
For the bipartition $A_1 \cup U_1^{\bad}$ (recall $A_1 \subseteq U_1^{\good}$), 
we infer 
\begin{equation}  
\label{eqn:4.21.2019.12:18a}  
\sum_{u_1 \in U_1^{\bad}} 
\big|N^-_{\vec{\Gamma}}(u_1) \cap A_1\big|  
= 
\sum_{a_1 \in A_1} 
\big|N^+_{\vec{\Gamma}}(a_1) \cap U_1^{\bad}\big|  
= 
\sum_{a_1 \in A_1} 
\big|S(a_1) \cap U_1^{\bad}\big|  
\stackrel{\text{def}}{\geq}  
|A_1|, 
\end{equation}  
where we used the definition of $A_1$  
from~(\ref{eqn:4.20.2019.8:05p}).  Moreover, 
from the bipartition $U_1^{\good} = A_1 \cup B_1$, we infer   
\begin{multline}  
\label{eqn:4.21.2019.12:03p}
3|A_1| \geq |A_1| + |B_1| = \big|U_1^{\good}\big|
\stackrel{(\ref{eqn:Clm13.2})}{\geq}
\big(\tfrac{1}{3} - 75\lambda^{1/4}\big) n   \\
\stackrel{(\ref{eqn:5.8.2019.1:39p})}{\geq} 
648 \lambda^{1/4} n  
\stackrel{(\ref{eqn:Clm13.2})}{\geq}
9 \big|U_1^{\bad}\big| 
\qquad 
\implies \qquad |A_1|
\geq 3 \big|U_1^{\bad}\big|.    
\end{multline}  
Combining~(\ref{eqn:4.21.2019.12:18a})   
and~(\ref{eqn:4.21.2019.12:03p}) yields 
$\sum_{u_1 \in U_1^{\bad}} 
|N^-_{\vec{\Gamma}}(u_1) \cap A_1|  
\geq 
3|U_1^{\bad}|$, 
and so an average vertex $\bar{u}_1 \in U_1^{\bad}$
satisfies $|N_{\vec{\Gamma}}^-(\bar{u}_1) \cap A_1| \geq 3$.   
Let $a_1, a_1', a_1'' \in N_{\vec{\Gamma}}^-(\bar{u}_1)$, in which case 
$$
\bar{u}_1 \in
N_{\Gamma}^+(a_1) 
\cap 
N_{\Gamma}^+\big(a_1'\big) 
\cap 
N_{\Gamma}^+\big(a_1''\big) 
= 
S(a_1) \cap S\big(a_1'\big) \cap S\big(a_1''\big), 
$$
and so $S(a_1) \cap S(a_1') \cap S(a_1'') \neq \emptyset$. 
\end{proof}

For the remainder of the proof, we fix
distinct $u_1, v_1, w_1 \in U_1^{\good}$ guaranteed by 
Corollary~\ref{cor:4.20.2019.7:41p}.  We also fix 
an element $x_1 \in S(u_1) \cap S(v_1) \cap S(w_1)$.   
We garner the following useful corollary.

\begin{cor}
\label{cor:4.21.2019.2:26p}  
The coloring $c$ is constant 
on the edges $E_G(x_1, U_1)$.  
\end{cor}

\begin{proof}[Proof of Corollary~\ref{cor:4.21.2019.2:26p}]
We first show that 
\begin{equation}  
\label{eqn:4.21.2019.2:45p}  
c(\{u_1, x_1\}) = c(\{v_1, x_1\}) = c(\{w_1, x_1\}).  
\end{equation}  
For that, 
since $u_1, v_1, w_1 \in U_1^{\good}$ are distinct good vertices, 
Corollary~\ref{cor:Clm13.4} guarantees that $c_{u_1}$, $c_{v_1}$, and $c_{w_1}$
are distinct, so we   
take $c_{u_1}$ to be {\sl red}, $c_{v_1}$ to be {\sl blue}, and 
$c_{w_1}$ to be {\sl yellow}.  
Assume, on the contrary, that $c(\{u_1, x_1\}) \neq c(\{v_1, x_1\})$.  
Then $(u_1, x_1, v_1)$ is a rainbow $U_1$-path 
where $u_1 \in U_1^{\good}$
is a good vertex, so Statement~(2) of Corollary~\ref{cor:abc}
guarantees that $c(\{x_1, v_1\})$ is $c_{u_1} = \text{red}$.  
Applying the same argument to $(v_1, x_1, u_1)$, we infer
that $c(\{u_1, x_1\})$ is $c_{v_1} = \text{blue}$.  
Now, $c_{w_1} = \text{yellow}$ appears on neither 
$(w_1, x_1, u_1)$ 
nor $(w_1, x_1, v_1)$ (since $x_1 \in S(w_1)$ guarantees
that $c(\{w_1, x_1\})$ is not $c_{w_1} = \text{yellow}$).  
Since $w_1 \in U_1^{\good}$ is a good vertex, Statement~(2) of 
Corollary~\ref{cor:abc} guarantees that both 
$(w_1, x_1, u_1)$ and $(w_1, x_1, v_1)$ are monochromatic, and so 
$c(\{w_1, x_1\})$ is both red and blue, a contradiction.  

Corollary~\ref{cor:4.21.2019.2:26p}
now easily follows from~(\ref{eqn:4.21.2019.2:45p}), where we take
that common color to be {\sl green}.    
By the argument above, 
any edge $\{x_1, y_1\} \in E_G(x_1, U_1)$ that isn't colored green
must be colored each of red, blue, and yellow, which isn't possible.  
\end{proof}    

\subsubsection{Finale}  
We return to our goal in~(\ref{eqn:4.21.2019.4:18p}).    
Let $u_1, v_1, w_1 \in U_1^{\good}$ and $x_1 \in S(u_1) \cap S(v_1)
\cap S(w_1)$ be fixed from the previous subsection, where 
all of $E_G(x_1, U_1)$ is colored {\sl green}, which is the only color
from before which 
we now need to reference.  
Then $E_G(x_1, U_1 \cup U_2)$ admits at most $|U_2| + 1$ colors, 
the set of which we call $C = C(x_1, U_1, U_2)$.  
As such, 
the number of non-$C$ colors on $E_G(x_1, U_0)$ is at least 
$$
\deg_G^c(x_1) - \deg_G^c(x_1, U_1) - \deg_G^c(x_1, U_2)
\geq 
\delta^c(G) - 1 - |U_2| 
\geq 
\tfrac{n+5}{3} - 1 - |U_2|
= 
\tfrac{n+2}{3} - |U_2|
\stackrel{(\ref{eqn:2.23.2019.12:33p})}{\geq}    
\tfrac{n+2}{3} - \left\lfloor \tfrac{n}{3} \right\rfloor,    
$$
which is positive.  
Fix $u_0 \in N_G(x_1, U_0)$ where $c(\{u_0, x_1\}) \not\in C$.  
In particular, 
$c(\{u_0, x_1\})$ is not green, and 
we take $c(\{u_0, x_1\})$ to be {\sl purple}.  
(It won't matter if $c(\{u_0, x_1\})$ appeared in the previous subsection, 
so long as $c(\{u_0, x_1\})$ is not green.)  
Define 
\begin{equation}
\label{eqn:4.21.2019.4:30p}  
T_0 = \big\{v_0 \in U_0^{\good}: \, v_0 \neq u_0, \, 
c_{v_0} \neq c(\{u_0, x_1\}) = \text{purple}, \, c_{v_0} \neq \text{green}\big\},   
\end{equation}  
where 
Corollary~\ref{cor:Clm13.4}  
guarantees 
\begin{equation}
\label{eqn:4.27.2019.1:33p}  
|T_0| \geq \big|U_0^{\good}\big| - 3 
\stackrel{(\ref{eqn:Clm13.2})}{\geq}    
\big(\tfrac{1}{3} - 75 \lambda^{1/4}\big)n  - 3  
= \Omega(n).  
\end{equation}  
We make the following critical observation.

\begin{obs}  
\label{obs:4.21.2019.4:31p}  
An edge
$\{v_0, x_1\} \in E_G(x_1, T_0)$ must be colored $c_{v_0}$.
\end{obs}

\begin{proof}[Proof of Observation~\ref{obs:4.21.2019.4:31p}]  
For a fixed $\{v_0, x_1\} \in E_G(x_1, T_0)$, we distinguish two cases.  \\

\noindent {\bf Case 1 ($c(\{v_0, x_1\}) \neq c(\{u_0, x_1\}) = \text{{\rm 
purple}}$).}  
Here, 
$(v_0, x_1, u_0)$ is a rainbow path where $v_0 \in T_0 \subseteq
U_0^{\good}$ is a good vertex.  
Statement~(2) of Corollary~\ref{cor:abc} guarantees that $c_{v_0}$ appears on 
$(v_0, x_1, u_0)$, and since $c_{v_0} \neq c(\{u_0, x_1\}) = \text{purple}$ holds 
by the 
definition of $T_0$, we must have $c(\{v_0, x_1\}) = c_{v_0}$.  
\hfill $\Box$   \\

\noindent {\bf Case 2 ($c(\{v_0, x_1\}) = c(\{u_0, x_1\}) = 
\text{{\rm purple}}$).}  
Among the fixed distinct 
vertices $u_1, v_1, w_1 \in U_1^{\good}$ above, 
Corollary~\ref{cor:Clm13.4} guarantees that 
at most one of the distinct colors $c_{u_1}$, $c_{v_1}$, $c_{w_1}$  
can equal $c(\{u_0, x_1\}) = c(\{v_0, x_1\}) = \text{purple}$, 
and at most one of $c_{u_1}$, $c_{v_1}$, $c_{w_1}$ can equal $c_{v_0}$.  
Assume w.l.o.g.~that
$$
c_{v_0} \neq c_{u_1} \neq c(\{u_0, x_1\}) = c(\{v_0, x_1\}) = \text{purple}.  
$$
Now, the path $(u_1, x_1, v_0)$ is a green-purple rainbow path
where $u_1 \in U_1^{\good}$ and $v_0 \in T_0 \subseteq U_0^{\good}$
are good vertices satisfying $c_{u_1} \neq c_{v_0}$.  
Statement~(3) guarantees that one of $c_{v_0} \neq c_{u_1}$ appears on 
$(u_1, x_1, v_0)$, but neither do.  Indeed, $c_{v_0}$ is neither
green nor purple by the definition of $T_0$ 
(cf.~(\ref{eqn:4.21.2019.4:30p})),   
and $c_{u_1}$ is not green by $x_1 \in S(u_1)$ and it is not purple
by our choice above.  
\end{proof}

We now conclude the proof of Statement~(1) of Lemma~\ref{lem:ext} when $\ell \equiv 2$ (mod 3).  
Fix a vertex $v_0 \in T_0$.  
By combining Statement~(1) of Corollary~\ref{cor:abc} with 
Observation~\ref{obs:4.21.2019.4:31p},   
we conclude that 
all edges $E_G(v_0, U_2 \cup \{x_1\})$ are colored the single primary color 
$c_{v_0}$.  
However distinctly the edges $E_G(v_0, U_1 \setminus \{x_1\})$ are colored, 
the 
edges $E_G(v_0, U_1 \cup U_2)$ are colored with at most 
$1 + (|U_1| - 1) = |U_1| = \lfloor n/3 \rfloor$ 
many colors, one of which is the primary color $c_{v_0}$.  
(Recall that it suffices to consider the case $|U_2| = |U_1| = \lfloor n/3 \rfloor$.)  
All remaining colors incident to $v_0$ are special and 
are applied to $E_G(v_0, U_0)$, the number of which is precisely
given by the parameter $|c^{\spec}(v_0, U_0)|$ from~(\ref{eqn:4.27.2019.1:06p}).
Altogether, we conclude    
\begin{equation}
\label{eqn:4.27.2019.1:50p}  
\big|c^{\spec}(v_0, U_0)\big| 
\geq 
\deg_G^c(v_0, U_0) - 
|c(E_G(v_0, U_1 \cup U_2)|  
\geq
\delta^c(G) - \big\lfloor \tfrac{n}{3} \big\rfloor   
\geq 
\tfrac{n+5}{3} - \big\lfloor \tfrac{n}{3} \big\rfloor  
\geq \tfrac{5}{3},   
\end{equation}  
and therefore 
$|c^{\spec}(v_0, U_0)| 
\geq 2$.
Now, 
(\ref{eqn:4.27.2019.1:33p})   
and~(\ref{eqn:4.27.2019.1:50p}) together contradict Statement~(5) of Corollary~\ref{cor:abc}.   

\subsection{Proof of Lemma~\ref{lem:ext}: Statement~(2) when $\boldsymbol{\ell \equiv 2}$ (mod 3)}  
Let $\ell \equiv 2$ (mod 3).  
The hypothesis of Statement~(2) of Lemma~\ref{lem:ext} 
gives that $\delta^c(G) \geq (n+4)/3$  
(cf.~(\ref{eqn:2.23.2019.1:07p})).     
We again assume
w.l.o.g.~that~(\ref{eqn:2.23.2019.12:33p}) holds, and we   
want to conclude
that $(G, c)$ admits a properly colored $\ell$-cycle $C_{\ell}$.   
\begin{equation}
\label{eqn:2.23.2019.4:47p}  
\text{{\it We assume, on the contrary, 
that $(G, c)$ does not admit a properly colored $\ell$-cycle $C_{\ell}$.}}  
\end{equation}  
Our assumption in~(\ref{eqn:2.23.2019.4:47p}) will guarantee vertices   
$x_1 \in U_1^{\good}$ 
and $y_1, z_1 \in U_1$, 
where 
$(x_1, y_1, z_1)$ is a rainbow $U_1$-path 
satisfying $c(\{x_1, y_1\}) \neq c_{x_1}$.  Then~(\ref{eqn:2.23.2019.4:47p})   
contradicts 
Statement~(2) of Corollary~\ref{cor:abc}.

We begin our work with an observation.  
Fix an auxiliary vertex $u_0 \in U_0^{\good}$, where   
Statement~(1) of Corollary~\ref{cor:abc} guarantees that $E_G(u_0, U_2)$ is colored
only with $c_{u_0}$.  
We observe that 
\begin{equation}
\label{eqn:4.27.2019.3:45p}  
E_G(u_0, U_0) \text{ is also colored only with $c_{u_0}$}.
\end{equation} 
To see~(\ref{eqn:4.27.2019.3:45p}),   
suppose $v_0 \in N_G(u_0, U_0)$ admits $c(\{u_0, v_0\}) \neq c_{u_0}$.
Let $u_1 \in N_G(v_0, U_1^{\good})$  
have color $c(\{v_0, u_1\}) \neq c(\{u_0, v_0\})$, where we used
$\deg^c_G(v_0, U^{\good}_1) \geq ((1/9) - 144 \lambda^{1/4})n$  
implicit in~(\ref{eqn:Clm13.2}).   
Statement~(1) of 
Corollary~\ref{cor:abc} guarantees that $c(\{v_0, u_1\}) = c_{u_1}$.
As such, the number $|c^{\spec}(u_1, U_1)|$ of special colors incident to $u_1$ in $U_1$
satisfies 
\begin{multline}  
\label{eqn:2.23.2019.5:31p}  
\big|c^{\spec}(u_1, U_1)\big| \geq \deg_G^c(u_1) - 
\deg_G^c(u_1, U_0) - \deg_G^c(u_1, U_2)  
\geq 
\delta^c(G) - 
\deg_G^c(u_1, U_0) - \deg_G^c(u_1, U_2)  \\ 
= 
\delta^c(G) - 
1 - \deg_G^c(u_1, U_2)  
\geq 
\tfrac{n+4}{3} - 
1 - |U_2|  
\stackrel{(\ref{eqn:2.23.2019.12:33p})}{\geq}    
\tfrac{n+1}{3} - \big\lfloor \tfrac{n}{3} \big\rfloor \geq \tfrac{1}{3}, 
\end{multline}  
so fix $v_1 \in N_G(u_1, U_1)$ where $c(\{u_1, v_1\}) \neq c_{u_1}$.  
Now, $P = (u_0, v_0, u_1, v_1)$ is a properly colored path where $c(\{u_0, v_0\}) \neq c_{u_0}$.  
Proposition~\ref{prop:Clm13.5}  
guarantees (with $i = 0$, $j = 1$, $k = 4$, and $\ell \equiv 2$ (mod 3))
that $P$ 
can be extended to a strong rainbow $\ell$-cycle $C_{\ell}$, 
contradicting~(\ref{eqn:2.23.2019.4:47p}).    
This proves~(\ref{eqn:4.27.2019.3:45p}).

We choose the first promised vertex
$x_1 \in U_1^{\good}$ arbitrarily, where the auxiliary vertex
$u_0 \in U_0^{\good}$ above is still fixed.  
To choose the second promised vertex $y_1 \in U_1$, define 
\begin{equation}
\label{eqn:2.25.2019.1:14p}  
A_{u_0} = \{u_1 \in N_G(u_0, U_1): \, c(\{u_0, u_1\}) \neq c_{u_0} \}     
\text{ and }  
B_{x_1} = \{v_1 \in N_G(x_1, U_1): \, c(\{x_1, v_1\}) \neq c_{x_1} \}.  
\end{equation}  
(The set $B_{x_1}$ is the same as $S(x_1)$ from the previous subsection.)  
Then $A_{u_0} \cup B_{x_1} \subseteq U_1$, and so 
\begin{equation}
\label{eqn:4.28.2019.1:37p}  
|A_{u_0} \cap B_{x_1}| = 
|A_{u_0}| + |B_{x_1}| - |A_{u_0} \cup B_{x_1}|  
\geq 
|A_{u_0}| + |B_{x_1}| - |U_1|.  
\end{equation}  
From our observation above (cf.~(\ref{eqn:4.27.2019.3:45p})), 
all of $E_G(u_0, U_0 \cup U_2)$ is colored
with $c_{u_0}$, and therefore $|A_{u_0}| \geq \deg_G^c(u_0) - 1$.  
Since $x_1 \in U_1^{\good}$ is a good vertex, Statement~(1) of Corollary~\ref{cor:abc}
guarantees that all of $E_G(x_1, U_0)$ is 
colored with $c_{x_1}$, and therefore 
$$
|B_{x_1}| \geq \deg_G^c(x_1) - \deg_G^c(x_1, U_0) - \deg_G^c(x_1, U_2)
\geq \deg_G^c(x_1) - 1 - |U_2|.  
$$
Returning to~(\ref{eqn:4.28.2019.1:37p}), we conclude   
\begin{multline*}  
|A_{u_0} \cap B_{x_1}| \geq |A_{u_0}| + |B_{x_1}| - |U_1|
\geq \deg_G^c(u_0) + \deg_G^c(x_1) - 2 - |U_1| - |U_2|  \\
= 
\deg_G^c(u_0) + \deg_G^c(x_1) - 2 - (n - |U_0|)  
\geq 2 \delta^c(G) - 2 - n + |U_0| \\
\geq 2 \big(\tfrac{n+4}{3}\big) - 2 - n + |U_0|  
= \tfrac{2n+2}{3}  - n + |U_0|  
\stackrel{(\ref{eqn:2.23.2019.12:33p})}{\geq}  
\tfrac{2n+2}{3}  - n + \big\lceil \tfrac{n}{3} \big\rceil
\geq 
\tfrac{2}{3}.  
\end{multline*}  
Fix $y_1 \in A_{u_0} \cap B_{x_1}$ arbitrarily.

To choose the third promised vertex $z_1 \in U_1$, we make a couple observations.  
First, we observe that the path $(x_1, y_1, u_0)$
must be monochromatic.  Indeed, 
since $y_1 \in A_{u_0} \cap B_{x_1}$, 
we infer from~(\ref{eqn:2.25.2019.1:14p}) that 
$c(\{u_0, y_1\}) \neq c_{u_0}$ 
and 
$c(\{x_1, y_1\}) \neq c_{x_1}$.   
Thus, if $(x_1, y_1, u_0)$ were rainbow, 
then Statement~(3) of Corollary~\ref{cor:abc} would guarantee that $(G, c)$ admits a 
properly colored $\ell$-cycle $C_{\ell}$, 
contradicting~(\ref{eqn:2.23.2019.4:47p}).    
Henceforth, we take $c(\{u_0, y_1\}) = c(\{x_1, y_1\})$ to be {\sl blue}.  
Second, 
we observe that 
\begin{equation}
\label{eqn:4.28.2019.2:48p}  
\text{all of $E_G(y_1, U_0)$ is colored by $c(\{u_0, y_1\}) = c(\{x_1, y_1\}) = \text{blue}$}.  
\end{equation}  
Indeed, suppose $\{v_0, y_1\} \in E_G(y_1, U_0)$ admitted $c(\{v_0, y_1\}) \neq c(\{u_0, y_1\}) 
= \text{blue}$.  
Then 
$(u_0, y_1, v_0)$ is a rainbow path
with $c(\{u_0, y_1\}) \neq c_{u_0}$
(because $y_1 \in A_{u_0}$ 
from~(\ref{eqn:2.25.2019.1:14p})).    
Statement~(2) of Corollary~\ref{cor:abc}
then 
guarantees that $(G, c)$ admits a properly colored $\ell$-cycle $C_{\ell}$,   
again 
contradicting~(\ref{eqn:2.23.2019.4:47p}).    
Now, all of $E_G(y_1, U_0 \cup \{x_1\})$ is colored blue, and so the number of 
non-blue colors of $E_G(y_1, U_1)$ is at least 
\begin{multline*}  
\deg_G^c(y_1) - \deg_G^c(y_1, U_0) - 
\deg_G^c(y_1, U_2) 
= 
\deg_G^c(y_1) - 1 - \deg_G^c(y_1, U_2)  \\
\geq 
\delta^c(G) - 1 - |U_2| 
\geq 
\tfrac{n+4}{3} - 1 - |U_2|  
\stackrel{(\ref{eqn:2.23.2019.12:33p})}{\geq}  
\tfrac{n+1}{3} - \big\lfloor \tfrac{n}{3} \big\rfloor 
\geq \tfrac{1}{3}.    
\end{multline*}  
Fix any $z_1 \in N_G(y_1, U_1)$ for which $c(\{y_1, z_1\})$ is not blue.  
Since $c(\{x_1, y_1\})$ is blue, we infer that 
$(x_1, y_1, z_1)$ is a rainbow path where $c_{x_1} \neq c(\{x_1, y_1\}) = 
\text{blue}$ is guaranteed by $y_1 \in B_{x_1}$ 
from~(\ref{eqn:2.25.2019.1:14p}).  
Thus, 
Statement~(2) of Corollary~\ref{cor:abc} guarantees from 
the rainbow 
$U_1$-path $(x_1, y_1, z_1)$ (where $c_{x_1} \neq c(\{x_1, y_1\}) = \text{blue}$) that 
$(G, c)$ admits a properly colored $\ell$-cycle $C_{\ell}$, 
again contradicting~(\ref{eqn:2.23.2019.4:47p}).

\section{Proof of Lemma~\ref{lem:ext} - Part 3: Strong or Short Cycles and the Case 
$\ell \equiv 1$ (mod 3)}  
\label{sec:ex3}  
Continuing from the previous sections, 
we prove Lemma~\ref{lem:ext} in the case $\ell \equiv 1$ (mod 3).  
For this, we will need a number of supporting details, where 
we begin 
by establishing an analogue of Corollary~\ref{cor:abc} for 
$\ell \equiv 1$ (mod 3)  
(another corollary of Proposition~\ref{prop:Clm13.5}).   
We use the following terminology and notation.  
For a fixed $j \in \mathbb{Z}_3$, recall the set $U_j = U_j^{\good} \cup U_j^{\bad}$
(cf.~(\ref{eqn:Clm13.2})).  We 
shall say that a vertex $u_j \in U_j^{\bad}$ is an {\it internal (bad) vertex} 
if $\deg_G^c(u_j, U_j) \geq 3$, 
and that $u_j \in U_j^{\bad}$ is an {\it external (bad) vertex} otherwise.  
We then define 
\begin{equation}
\label{eqn:IjEj}  
I_j^{\bad} = \left\{u_j \in U_j^{\bad}: \ \deg_G^c(u_j, U_j) \geq 3 \right\}
\qquad \text{and} \qquad 
E_j^{\bad} = \left\{u_j \in U_j^{\bad}: \ \deg_G^c(u_j, U_j) \leq 2 \right\}.       
\end{equation}  

\begin{cor}
\label{cor:13.7}  
Let $\ell \equiv 1$ {\rm (mod 3)}.  Fix an index $j \in \mathbb{Z}_3$,  
a vertex $u_j \in U_j^{\good}$, and an edge $\{u_j, v\} \in E$.
\begin{enumerate}  
\item  
If $v \in U_j$
or $v \in I_{j+1}^{\bad}$, 
then $c(\{u_j, v\}) = c_{u_j}$ is the primary color of $u_j$. 
\item 
The edges 
$E_G(u_j, U_{j-1})$
admit at least 
$\deg_G^c(u_j) - 1 - |U_{j+1} \setminus I_{j+1}^{\bad}|$
non-$c_{u_j}$ colors.   
\item
If $v \in U_{j-1}^{\bad}$ 
and $c(\{u_j, v\}) \neq c_{u_j}$, 
then 
$\deg_G^c(v, U_j) \geq ((1/6) - 37\lambda^{1/4}) n$.   
\item 
If 
$v \in E_{j-1}^{\bad}$ 
and $c(\{u_j, v\}) \neq c_{u_j}$, then 
$\deg_G^c(v, U_j) \geq \deg_G^c(v) - 3$.   
\item 
If $v \in U_{j-1}^{\good}$ 
and $c(\{u_j, v\}) \neq c_{u_j}$, then 
$\deg_G^c(v, U_j \setminus I_j^{\bad}) \geq \deg_G^c(v) 
- 1$.
\end{enumerate}  
\end{cor}

\begin{proof}[Proof of Corollary~\ref{cor:13.7}]  
Let $\ell \equiv 1$ (mod 3). Fix an index $j \in \mathbb{Z}_3$, and w.l.o.g.~let $j = 
0 \in \mathbb{Z}_3$.  Fix 
a good vertex 
$u_0 \in U_0^{\good}$, and fix an edge $\{u_0, v\} \in E$.  

For Statement~(1), 
assume first that $v \in U_0$.  If $c(\{u_0, v\}) \neq c_{u_0}$, 
then $P = (u_0, v)$
is a $c_{u_0}$-free rainbow path which
Proposition~\ref{prop:Clm13.5} extends to a strong rainbow
$\ell$-cycle $C_{\ell}$ (using $i = j = 0$, $k = 2$, and $K = \ell \equiv 1$ (mod 3)), 
contradicting~(\ref{eqn:overarching1}).    
Assume next that $v \in I^{\bad}_1$, in which case 
$v$ sees at least three colors in $U_1$.  If $c(\{u_0, v\}) \neq c_{u_0}$, then $v$
admits a neighbor $w_1 \in U_1$ where $c(\{v, w_1\})$ is neither $c_{u_0}$ nor
$c(\{u_0, v\})$.  Now, $P = (u_0, v, w_1)$ is a $c_{u_0}$-free rainbow path which
Proposition~\ref{prop:Clm13.5} extends to a strong rainbow
$\ell$-cycle $C_{\ell}$ (using $i = 0$, $j = 1$, $k = 3$, and $K = \ell \equiv 1$ (mod 3)), 
again contradicting~(\ref{eqn:overarching1}).    

Statement~(2) is an easy consequence of Statement~(1).  For $u_0 \in U_0^{\good}$ satisfies
\begin{equation}  
\label{eqn:2.27.2019.1:53p}  
\deg_G^c(u_0) = \deg_G^c(u_0, U_0) + \deg_G^c(u_0, U_1) + \deg_G^c(u_0, U_2),   
\end{equation}  
where Statement~(1) guarantees that 
all edges of $E_G(u_0, U_0 \cup I_1^{\bad})$ 
are colored $c_{u_0}$.  
However colors are assigned to 
$E_G(u_0, U_1 \setminus I_1^{\bad})$, 
at least $\deg_G^c(u_0) - 1 - |U_1\setminus I_1^{\bad}|$
many non-$c_{u_0}$ colors remain, and these must occur on the edges
of $E_G(u_0, U_2)$.   

For Statement~(3), 
we prepare an observation used multiple times below.  
For an edge $\{u_0, u_2\} \in E_G(u_0, U_2)$ satisfying $c(\{u_0, u_2\}) \neq c_{u_0}$, 
we observe that 
\begin{equation}
\label{eqn:2.27.2019.2:11p}  
\text{{\it 
every edge $\{u_2, u_1\} \in E_G(u_2, U_1)$ must be colored either 
$c(\{u_0, u_2\})$ or $c_{u_0}$,}} 
\end{equation}  
lest $(u_0, u_2, u_1)$ is a $c_{u_0}$-free rainbow
path which 
Proposition~\ref{prop:Clm13.5} extends to a strong rainbow
$\ell$-cycle $C_{\ell}$ (using $i = 0$, $j = 1$, $k = 3$, and $K = \ell \equiv 1$ (mod 3)), 
contradicting~(\ref{eqn:overarching1}).    
Now, as in Statement~(3), assume $v \in U_2^{\bad}$  
where $c(\{u_0, v\}) \neq c_{u_0}$.  
Then~(\ref{eqn:2.27.2019.2:11p})   
gives $\deg_G^c(v, U_1^{\good}) \leq 2$, where~(\ref{eqn:badvertexsplit})   
and~(\ref{eqn:badvertexsplitprime}) add that $v \in U_2^{\bad}$ satisfies
$$
\deg_G^c\big(v, U_0^{\good}\big) \geq 
\max\Big\{  
\deg_G^c\big(v, U_1^{\good}\big), \, 
\deg_G^c\big(v, U_2^{\good}\big) \Big\}  
\geq 
\deg_G^c\big(v, U_2^{\good}\big).  
$$  
Since $V = U_1^{\good} \cup U_2^{\good} \cup U_3^{\good} \cup V^{\bad}$ is a partition, 
\begin{multline}  
\label{eqn:2.27.2019.1:23p}  
\deg_G^c(v) = 
\deg_G^c\big(v, U_0^{\good}\big) 
+
\deg_G^c\big(v, U_1^{\good}\big) 
+
\deg_G^c\big(v, U_2^{\good}\big) 
+
\deg_G^c\big(v, V^{\bad}\big) \\
\leq 
\deg_G^c\big(v, U_0^{\good}\big) 
+
\deg_G^c\big(v, U_1^{\good}\big) 
+
\deg_G^c\big(v, U_2^{\good}\big) 
+
\big|V^{\bad}\big|  \\
\stackrel{(\ref{eqn:2.27.2019.2:11p})}{\leq}     
\deg_G^c\big(v, U_0^{\good}\big)
+ 2
+ 
\deg_G^c\big(v, U_2^{\good}\big)
+
\big|V^{\bad}\big|  
\stackrel{(\ref{eqn:badvertexsplitprime})}{\leq}
2 \deg_G^c\big(v, U_0^{\good}\big) 
+ 2 
+ 
\big|V^{\bad}\big|  \\
\stackrel{(\ref{eqn:Clm13.2})}{\leq}     
2 \deg_G^c\big(v, U_0^{\good}\big) 
+ 2 
+ 72 \lambda^{1/4} n \leq  
2 \deg_G^c\big(v, U_0^{\good}\big) 
+ 73 \lambda^{1/4} n.  
\end{multline}  
Thus, 
we conclude Statement~(3) from 
\begin{multline*}  
\deg_G^c(v, U_0)  
\geq 
\deg_G^c\big(v, U_0^{\good}\big)
\stackrel{(\ref{eqn:2.27.2019.1:23p})}{\geq}    
\tfrac{1}{2} \big(\deg_G^c(v) - 73 \lambda^{1/4}n\big)  \\
\geq 
\tfrac{1}{2} \big(\delta^c(G) - 73 \lambda^{1/4}n\big)  
\geq 
\tfrac{1}{2} \big(\tfrac{n+4}{3}  - 73 \lambda^{1/4}n\big)  
\geq 
\big(\tfrac{1}{6} - 37 \lambda^{1/4}\big)n.   
\end{multline*}

For Statement~(4), assume that $v \in E^{\bad}_2$ and that $c(\{u_0, v\}) \neq c_{u_0}$.  
As before with~(\ref{eqn:2.27.2019.1:53p}), 
we have 
$$
\deg_G^c(v) = \deg_G^c(v, U_0) + \deg_G^c(v, U_1) + \deg_G^c(v, U_2).   
$$
By 
the definition of $E_2^{\bad}$ in~(\ref{eqn:IjEj}), 
we have 
$\deg_G^c(v, U_2) \leq 2$.  
Moreover, (\ref{eqn:2.27.2019.2:11p}) gives 
$\deg_G(v, U_1) \leq 2$, where these colors can only be $c(\{u_0, v\})$ and $c_{u_0}$.  
Since $c(\{u_0, v\})$ is used on $E_G(v, U_0)$, Statement~(4)
follows.  

For Statement~(5), assume that $v \in U_2^{\good}$ and that $c(\{u_0, v\}) \neq c_{u_0}$.  
As before with~(\ref{eqn:2.27.2019.1:53p}), 
we have 
$$
\deg_G^c(v) = \deg_G^c(v, U_0) + \deg_G^c(v, U_1) + \deg_G^c(v, U_2).   
$$
Statement~(1) ensures that 
all edges of $E_G(v, U_2 \cup I_0^{\bad})$ are 
assigned the primary color $c_v$.  
Moreover, (\ref{eqn:2.27.2019.2:11p}) gives 
all edges of 
$E_G(v, U_1)$
are assigned 
$c(\{u_0, v\})$ and $c_{u_0}$, which must include   
$c_v$.  
Now, all 
non-$\{c_{u_0}, c_v\}$ colors incident to $v$ must be on the edges 
$E_G(v, U_0\setminus I_0^{\bad})$, where $c(\{u_0, v\})$ is one such color used.    
In either case
of $c_v \in \{c_{u_0}, c(\{u_0, v\})\}$,   
only $c_{u_0}$ is possibly not used
on 
$E_G(v, U_0\setminus I_0^{\bad})$.   
\end{proof}

\subsection{On 4-cycles $\boldsymbol{C_4}$ in $\boldsymbol{(G, c)}$ when 
$\boldsymbol{\ell \equiv 1}$ (mod 3)}  
Since $4 \equiv 1$ (mod 3), 
which is the modular case of Lemma~\ref{lem:ext} we seek to prove, 
we study 4-cycles $C_4$ in $(G, c)$
from the point of view of 
Proposition~\ref{prop:Clm13.5}
and Corollary~\ref{cor:13.7}.  We begin with the following notation and terminology.    
For $j \in \mathbb{Z}_3$ and $u_j \in U_j$, 
recall from~(\ref{eqn:3.3.2019.6:33p}) that an edge $\{u_j, u_{j-1}\} \in E_G(u_j, U_{j-1})$
is said to be {\it typical} when 
$c(\{u_j, u_{j-1}\}) = c_{u_j}$ is the primary color of $u_j$, and is said to be {\it special}
otherwise.  
In the reverse 
of~(\ref{eqn:3.3.2019.6:33p}), 
we write 
\begin{multline*}  
N_G^{\typ}\big(u_{j-1}, U_j^{\good}\big) = \left\{ u_j \in N_G\big(u_{j-1}, U_j^{\good}\big): 
c(\{u_{j-1}, u_j\}) = c_{u_j} \right\}, \\
N_G^{\spec}\big(u_{j-1}, U_j^{\good}\big) = \left\{ u_j \in N_G\big(u_{j-1}, U_j^{\good}\big): 
c(\{u_{j-1}, u_j\}) \neq c_{u_j} \right\}, \\
\deg^{\typ}_G\big(u_{j-1}, U_j^{\good} \big) = 
\Big| N_G^{\typ}\big(u_{j-1}, U_j^{\good}\big) \Big|,  
\qquad \text{and} \qquad 
\deg^{\spec}_G\big(u_{j-1}, U_j^{\good} \big) = 
\Big| N_G^{\spec}\big(u_{j-1}, U_j^{\good}\big) \Big|.   
\end{multline*}  
We proceed with an initial observation.

\begin{obs}
\label{obs:3.11.2019.1:47}  
Let $\ell \equiv 1$ {\rm (mod 3)}.  
Fix $j \in \mathbb{Z}_3$, $u_{j-1} \in U_{j-1}$, and 
$u_j \neq v_j \in N_G^{\spec}(u_{j-1}, U_j^{\good})$.   
Then 
$c(\{u_{j-1}, u_j\})$, $c(\{u_{j-1}, v_j\})$, $c_{u_j}$, and $c_{v_j}$ can't all
be distinct.  
\end{obs}  

\begin{proof}[Proof of Observation~\ref{obs:3.11.2019.1:47}]    
Let $\ell \equiv 1$ (mod 3).  Fix $j \in \mathbb{Z}_3$, and w.l.o.g.~let $j = 0$.
Fix $u_2 \in U_2$ and fix 
$u_0 \neq v_0 \in N_G^{\spec}(u_2, U_0^{\good})$.  
We apply 
Proposition~\ref{prop:Clm13.3} 
and Corollary~\ref{cor:Clm13.4}
to each of $u_0 \neq v_0 \in U_0^{\good}$   
to determine 
at least 
$$
\big|U_2^{\good}\big| - 322 \lambda^{1/4} n
\stackrel{(\ref{eqn:Clm13.2})}{\geq}    
\left(\frac{1}{3} - 397 \lambda^{1/4} \right)n 
\stackrel{(\ref{eqn:5.8.2019.1:39p})}{>} 0 
$$
many vertices $w_2 \in U_2^{\good}$
for which 
$\{u_0, w_2\}, \{v_0, w_2\} \in E$ 
and 
$$
c(\{u_0, w_2\}) = c_{u_0} \neq c_{v_0} = c(\{v_0, w_2\}).
$$
If Observation~\ref{obs:3.11.2019.1:47} is false, then 
$(u_0, w_2, v_0, u_2)$ 
is a strong rainbow 4-cycle which 
Statement~(3) of 
Proposition~\ref{prop:Clm13.5} extends to a strong rainbow
$\ell$-cycle $C_{\ell}$,  
contradicting~(\ref{eqn:overarching1}).    
\end{proof}

We continue with a second observation, which is a corollary of the one above.    

\begin{cor}  
\label{cor:13.8}  
Let $\ell \equiv 1$ {\rm (mod 3)}.  Fix $j \in \mathbb{Z}_3$ and $u_{j-1} \in U_{j-1}$.  
If $\deg_G^{\spec}(u_{j-1}, U_j^{\good}) \geq 4$, 
then all but at most one of the edges $\{u_{j-1}, u_j\} \in E$, where $u_j \in N_G^{\spec}(u_{j-1}, 
U_j^{\good})$ are monochromatic.  Thus, 
$$
\deg_G^{\typ}\big(u_{j-1}, U_j^{\good}\big) 
\geq 
\deg_G^c\big(u_{j-1}, U_j^{\good}\big) - 3,    
$$
where in particular 
$$
\deg_G^{\typ}\big(u_{j-1}, U_j^{\good}\big) 
\geq 
\left\{
\begin{array}{cc}
\big(\tfrac{1}{6} - 110 \lambda^{1/4}\big)n 
& \text{if $\deg_G^{\spec}(u_{j-1}, U_j^{\good}) \geq 1$,}  \\
\big(\tfrac{1}{9}  - 144 \lambda^{1/4}\big)n  & \text{otherwise.}  
\end{array}
\right.  
$$
\end{cor}

\begin{proof}[Proof of Corollary~\ref{cor:13.8}]    
Let $\ell \equiv 1$ (mod 3).  Fix $j \in \mathbb{Z}_3$ and w.l.o.g.~assume $j = 0$.  
Fix $u_2 \in U_2$ and assume that $\deg_G^{\spec}(u_2, U_0^{\good}) \geq 4$.  
For sake of argument, 
\begin{multline}  
\label{eqn:5.7.2019.9:55a}  
\text{{\sl we assume the special edges $\{u_2, u_0\} \in E$,}} \\ 
\text{{\sl 
where $u_0 \in N_G^{\spec}(u_2, U_0^{\good})$, are not entirely monochromatic.}}    
\end{multline}  
We consider two cases.  \\

\noindent {\bf Case 1 ($\exists \ u_0 \neq v_0 \in N_G^{\spec}(u_2, U_0^{\good})$: 
$c(\{u_2, u_0\}) = c(\{u_2, v_0\})$).}  
Fix $u_0 \neq v_0 \in N_G^{\spec}(u_2, U_0^{\good})$ with $c(\{u_2, u_0\}) = c(\{u_2, v_0\})$.  
Using~(\ref{eqn:5.7.2019.9:55a}),   
we infer the existence of an edge 
$\{u_2, w_0\} \in E$, where $w_0 \in N_G^{\spec}(u_2, U_0^{\good})$, 
for which $c(\{u_2, w_0\}) \neq c(\{u_2, u_0\}) = c(\{u_2, v_0\})$.  
We claim that 
\begin{equation}
\label{eqn:3.11.2019.3:07p}  
c(\{u_2, u_0\}) = c(\{u_2, v_0\}) = c_{w_0}.    
\end{equation}  
Indeed, 
pivoting $\{u_2, u_0\}$ against $\{u_2, w_0\}$, 
Observation~\ref{obs:3.11.2019.1:47}   
ensures $c(\{u_2, u_0\}) = c_{w_0}$ or $c(\{u_2, w_0\}) = c_{u_0}$.  
In the former case, 
(\ref{eqn:3.11.2019.3:07p}) holds by the   
hypothesis of Case~1.  In the latter case, 
we pivot $\{u_2, v_0\}$ against $\{u_2, w_0\}$, where 
Observation~\ref{obs:3.11.2019.1:47}   
gives 
$c(\{u_2, v_0\}) = c_{w_0}$ or $c(\{u_2, w_0\}) = c_{v_0}$.  If 
$c(\{u_2, w_0\}) = c_{u_0}$, 
then 
Corollary~\ref{cor:Clm13.4} gives $c(\{u_2, w_0\}) \neq c_{v_0}$, and so 
$c(\{u_2, v_0\}) = c_{w_0}$  
and again~(\ref{eqn:3.11.2019.3:07p}) holds.

If the first conclusion 
of Corollary~\ref{cor:13.8} does not hold,    
we ignore the edge $\{u_2, w_0\}$ to 
infer 
the existence of an edge $\{u_2, x_0\} \in E$, where 
$x_0 \in N_G^{\spec}(u_2, U_0^{\good})$, 
for which $c(\{u_2, x_0\}) \neq c(\{u_2, u_0\})$.  
By~(\ref{eqn:3.11.2019.3:07p}), 
$$
c_{x_0} = c(\{u_2, u_0\}) = c(\{u_2, v_0\}) = c_{w_0}, 
$$
and $x_0 \neq w_0 \in U_0^{\good}$ 
contradicts Corollary~\ref{cor:Clm13.4}.    \hfill $\Box$  \\

\noindent {\bf Case 2 ($\forall \ u_0 \neq v_0 \in N_G^{\spec}(u_2, U_0^{\good})$,  
$c(\{u_2, u_0\}) \neq c(\{u_2, v_0\})$).}  
Since $\deg_G^{\spec}(u_2, U_0^{\good}) \geq 4$, 
fix distinct $u_0, v_0, w_0, x_0 \in 
N_G^{\spec}(u_2, U_0^{\good})$.
Using Observation~\ref{obs:3.11.2019.1:47},    
we take w.l.o.g.~$c(\{u_2, u_0\}) = c_{v_0}$.  
Observation~\ref{obs:3.11.2019.1:47}   
then ensures that 
$c(\{u_2, w_0\}) = c_{u_0}$ (since $c_{v_0} \neq c_{w_0}$ 
from Corollary~\ref{cor:Clm13.4})
and $c(\{u_2, x_0\}) = c_{u_0}$, 
which contradicts the hypothesis of Case~2.  \hfill $\Box$  \\

The remaining assertions 
of Corollary~\ref{cor:13.8} are now easy to establish.  
When $\deg_G^{\spec}(u_2, U_0^{\good}) = 0$, all edges of $E_G(u_2, U_0^{\good})$
are typical, 
and so 
$$
\deg_G^{\typ}\big(u_2, U_0^{\good}\big) \geq 
\deg^c_G\big(u_2, U_0^{\good}\big) 
\stackrel{(\ref{eqn:Clm13.2})}{\geq}    
\big(\tfrac{1}{9} - 72 \lambda^{1/4}\big) n - 72 \lambda^{1/4} n
= \big(\tfrac{1}{9} - 144 \lambda^{1/4}\big) n.   
$$
When $\deg_G^{\spec}(u_2, U_0^{\good}) \geq 1$, 
the first assertion of 
Corollary~\ref{cor:13.8} guarantees that edges $\{u_2, u_0\} \in E$ with $u_0 \in 
N_G^{\spec}(u_2, U_0^{\good})$ are colored with at most three colors.  
Thus, 
\begin{multline*}  
\deg_G^{\typ}\big(u_2, U_0^{\good}\big) \geq 
\deg^c_G\big(u_2, U_0^{\good}\big)  - 3
\stackrel{\text{Cor.\ref{cor:13.7}}}{\geq}    
\big(\tfrac{1}{6} - 37 \lambda^{1/4} \big)n 
- 72 \lambda^{1/4} n 
- 3  \\
= 
\big(\tfrac{1}{6} - 109 \lambda^{1/4}\big)n - 3  
\geq 
\big(\tfrac{1}{6} - 110 \lambda^{1/4}\big)n,   
\end{multline*}  
as promised.  
\end{proof}

\begin{definition}[$j$-special]\label{def:jspecial}
\rm 
A 4-cycle $(u_j, u_{j-1}, v_j, v_{j-1})$ is
{\it $j$-special} 
for some $j \in \mathbb{Z}_3$ 
if 
$u_j, v_j \in U_j^{\good}$, $u_{j-1}, v_{j-1} \in U_{j-1}$, 
$\{u_j, u_{j-1}\}, \{v_j, v_{j-1}\} \in E$ are typical, and 
$\{u_j, v_{j-1}\}, \{v_j, u_{j-1}\} \in E$ are special.  
\end{definition}

\begin{obs}  
\label{obs:13.9}  
Let $\ell \equiv 1$ {\rm (mod 3)}, and fix $j \in \mathbb{Z}_3$.  
A $j$-special 4-cycle $(u_j, u_{j-1}, v_j, v_{j-1})$ receives precisely three colors, 
where in particular $c(\{u_j, v_{j-1}\}) = c(\{v_j, u_{j-1}\})$.  
\end{obs}

\begin{proof}[Proof of Observation~\ref{obs:13.9}]
Let $\ell \equiv 1$ (mod 3).  Fix $j \in \mathbb{Z}_3$, and w.l.o.g.~let $j = 0$.  Fix a 0-special
4-cycle $(u_0, u_2, v_0, v_2)$.  By its definition, we infer  
$c(\{u_2, u_0\}) = c_{u_0}$ and $c(\{v_2, v_0\})
= c_{v_0}$ are primary, 
which Corollary~\ref{cor:Clm13.4}   
ensures are distinct.  
By its definition, $c(\{u_0, v_2\}) \neq c_{u_0}$ 
and 
$c(\{v_0, u_2\}) \neq c_{v_0}$ 
are special.  Observe that $c(\{u_0, v_2\}) \neq c_{v_0}$ since otherwise 
Proposition~\ref{prop:Clm13.3} guarantees that 
$G - \{v_0, v_2\}$ 
(here denoting edge-removal) 
contradicts~(\ref{eqn:overarching3}).    
Similarly, $c(\{v_0, u_2\}) \neq c_{u_0}$.  
Thus, 
if $c(\{u_0, v_2\}) \neq c(\{v_0, u_2\})$, then $(u_0, u_2, v_0, v_2)$ is a strong 
rainbow 4-cycle
which Statement~(3) of 
Proposition~\ref{prop:Clm13.5} extends to a strong rainbow
$\ell$-cycle $C_{\ell}$, 
contradicting~(\ref{eqn:overarching1}).    
\end{proof}

We conclude this subsection with a
corollary of the preceding observation.

\begin{cor}
\label{cor:13.10}  
Let $\ell \equiv 1$ {\rm (mod 3)}.  Fix $j \in \mathbb{Z}_3$, $u_{j-1} \neq v_{j-1} \in U_{j-1}$, 
and a color $\alpha$ from $c$.  Set 
\begin{multline*}  
\qquad 
\qquad 
\qquad 
A 
= A_{\alpha}(u_{j-1}) 
= \left\{u_j \in N_G^{\spec}(u_{j-1}, U_j): \, c(\{u_j, u_{j-1}\}) = \alpha \right\}  \\
\quad \text{and} \quad 
B 
= B_{\alpha}(v_{j-1}) 
= \left\{v_j \in N_G^{\spec}(v_{j-1}, U_j): \, c(\{v_j, v_{j-1}\}) \neq \alpha \right\}.    
\qquad 
\qquad 
\qquad 
\end{multline*}  
Then $|A \cup B| < ((1/6) + 186 \lambda^{1/4})n$.  
\end{cor}

\begin{proof}[Proof of Corollary~\ref{cor:13.10}]
Let $\ell \equiv 1$ (mod 3).  Fix $j \in \mathbb{Z}_3$, and w.l.o.g.~let $j = 0$.  Fix
$u_2 \neq v_2 \in U_2$ and fix a color $\alpha$ of $c$.  Let $A = A(u_2)$ 
and $B = B(v_2)$ be defined as above, but 
assume for contradiction that 
\begin{equation}
\label{eqn:3.12.2019.5:28p}  
|A \cup B| \geq \big(\tfrac{1}{6} + 186 \lambda^{1/4} \big) n.  
\end{equation}  
We will use~(\ref{eqn:3.12.2019.5:28p}) to guarantee distinct vertices 
\begin{equation}
\label{eqn:3.12.2019.5:35p}  
N_G^{\typ}\big(u_2, 
U_0^{\good}\big)
\cap (A \cup B) \ni u_0 \neq  
v_0 \in 
N_G^{\typ}\big(v_2, 
U_0^{\good}\big)
\cap (A \cup B).    
\end{equation}  
If~(\ref{eqn:3.12.2019.5:35p}) holds, then it will conclude our proof, because 
$(u_2, u_0, v_2, v_0)$ 
would be a rainbow 0-special 4-cycle,  
contradicting 
Observation~\ref{obs:13.9}.    
To see this, we first note 
from~(\ref{eqn:3.12.2019.5:35p})
that $u_0 \neq v_0 \in U_0^{\good}$ 
are good vertices, 
where $u_0 \in N_G^{\typ}(u_2, U_0^{\good})$ guarantees
$c(\{u_0, u_2\}) = c_{u_0}$ 
and $v_0 \in N_G^{\typ}(v_2, U_0^{\good})$ guarantees
$c(\{v_0, v_2\}) = c_{v_0}$, and where
$c_{u_0} \neq c_{v_0}$ is guaranteed by Corollary~\ref{cor:Clm13.4}.   
Since $u_0 \in A \cup B$ happens only from $u_0 \in B\setminus A$
(because $c(\{u_0, u_2\}) = c_{u_0}$ is not special for $u_0$), we infer 
$c(\{u_0, v_2\}) \neq \alpha$
is some non-$\alpha$ special color for $u_0$.  
Since $v_0 \in A \cup B$ happens only from $v_0 \in A \setminus B$
(because $c(\{v_0, v_2\}) = c_{v_0}$ is not special for $v_0$), we infer $c(\{v_0, u_2\}) = \alpha$
is special for $v_0$.   
(Note:  the existence of vertices $u_0$ and $v_0$ in~(\ref{eqn:3.12.2019.5:35p}) implies
they are necessarily distinct.)      
Thus, $(u_2, u_0, v_2, v_0)$ is a rainbow 0-special 4-cycle, as claimed.

To prove~(\ref{eqn:3.12.2019.5:35p})   
(from~(\ref{eqn:3.12.2019.5:28p})), define 
$$
A^{\good} = A \cap U_0^{\good} \qquad \text{and} \qquad B^{\good} = B \cap U_0^{\good}.  
$$ 
Then
\begin{equation}
\label{eqn:3.13.2019.5:08p}  
\big|A^{\good} \cup B^{\good}\big| 
\stackrel{(\ref{eqn:Clm13.2})}{\geq}    
|A \cup B| - 72 \lambda^{1/4} n 
\stackrel{(\ref{eqn:3.12.2019.5:28p})}{\geq}  
\big(\tfrac{1}{6} + 114 \lambda^{1/4}\big) n, 
\end{equation}  
and so one of $|A^{\good}|$ or $|B^{\good}|$ is large.  
Our proof will ultimately show that both are large, and so we begin by
assuming the former is non-empty.  
We therefore infer that 
$\deg_G^{\spec}(u_2, U_0^{\good}) \geq |A^{\good}| > 0$, 
and so Corollary~\ref{cor:13.8}  
gives 
\begin{equation}
\label{eqn:3.13.2019.4:36p}  
\deg_G^{\typ}\big(u_2, U_0^{\good}\big) \geq \big(\tfrac{1}{6} - 110 \lambda^{1/4}\big)n.  
\end{equation}  
If $N_G^{\typ}(u_2, U_0^{\good})$ and $A \cup B$ were disjoint, we would have 
\begin{multline}  
\label{eqn:3.13.2019.4:40p}  
|U_0|  
\geq 
\deg_G^{\typ}\big(u_2, U_0^{\good}\big)
+ 
|A \cup B|
\stackrel{(\ref{eqn:3.13.2019.4:36p})}{\geq}    
\big(\tfrac{1}{6} - 110\lambda^{1/4}\big)n    
+ |A\cup B|  \\
\stackrel{(\ref{eqn:3.12.2019.5:28p})}{\geq}  
\big(\tfrac{1}{6} - 110\lambda^{1/4}\big)n   
+ 
\big(\tfrac{1}{6} + 186 \lambda^{1/4} \big)n  
= 
\big(\tfrac{1}{3} + 76\lambda^{1/4}\big) n  
\stackrel{(\ref{eqn:Clm13.2})}{>}    
|U_0|, 
\end{multline}  
a contradiction.  
This guarantees the existence of the vertex $u_0$ 
in~(\ref{eqn:3.12.2019.5:35p}).

To guarantee the 
existence of the vertex $v_0$ 
in~(\ref{eqn:3.12.2019.5:35p}), we argue similarly.  For that, 
$N_G^{\typ}(u_2, U_0^{\good})$ and $A^{\good}$ are disjoint subsets of $U_0^{\good}$, 
and so 
\begin{multline}  
\label{eqn:3.13.2019.5:06p}  
\big|A^{\good}\big| + 
\big(\tfrac{1}{6} - 110 \lambda^{1/4}\big)n    
\stackrel{(\ref{eqn:3.13.2019.4:36p})}{\leq}    
\big|A^{\good}\big| + \deg_G^{\typ}\big(u_2, U_0^{\good}\big) \leq \big|U_0^{\good}\big| \\
\stackrel{(\ref{eqn:Clm13.2})}{\leq}    
\big(\tfrac{1}{3} + 75 \lambda^{1/4}\big)n   
\qquad 
\implies \qquad 
\big|A^{\good}\big| \leq \big(\tfrac{1}{6} + 185 \lambda^{1/4}\big) n.  
\end{multline}  
Thus, 
\begin{multline*}  
\deg_G^{\spec}\big(v_2, U_0^{\good}\big)
\geq 
\big|B^{\good} \setminus A^{\good}\big| = \big|A^{\good} \cup B^{\good}\big| - \big|A^{\good}\big|  \\
\stackrel{(\ref{eqn:3.13.2019.5:08p})}{\geq}     
\big(\tfrac{1}{6} + 186\lambda^{1/4}\big)n - \big|A^{\good}\big|  
\stackrel{(\ref{eqn:3.13.2019.5:06p})}{\geq}  
\lambda^{1/4} n  > 0,   
\end{multline*}  
and so Corollary~\ref{cor:13.8}  
gives 
$$
\deg_G^{\typ}\big(v_2, U_0^{\good}\big) \geq 
\big(\tfrac{1}{6} - 110\lambda^{1/4}\big)n.  
$$
We now proceed identically to before with~(\ref{eqn:3.13.2019.4:36p})--(\ref{eqn:3.13.2019.4:40p}).    
\end{proof}

\subsection{Amenable elements of $\mathbb{Z}_3$}  
Recall the partition $V(G) = U_0 \cup U_1 \cup U_2$ 
of $(G, c)$ 
from~(\ref{eqn:Clm13.2}).  For each $j \in \mathbb{Z}_3$, recall the (so-called 
{\it internal} and {\it external}) 
sets of bad vertices 
$$
I_j^{\bad} = \{u_j \in U_j^{\bad}: \, \deg^c_G(u_j, U_j) \geq 3\}
\qquad \text{and} \qquad 
E_j^{\bad} = \{u_j \in U_j^{\bad}: \, \deg^c_G(u_j, U_j) \leq 2\}
$$
from~(\ref{eqn:IjEj}).    
In particular, $U_j^{\bad} = I_j^{\bad} \cup E_j^{\bad}$ is a partition, 
and so 
\begin{equation}
\label{eqn:3.14.2019.2:15p}  
U_j = U_j^{\good} \cup U_j^{\bad} = U_j^{\good} \cup I_j^{\bad} \cup E_j^{\bad}
\end{equation}  
are partitions.  
We set 
\begin{equation}
\label{eqn:3.14.2019.2:16p}  
\hat{U}_j = U_j^{\good} \cup E_j^{\bad} \qquad \text{and} \qquad 
\Delta_j = m - \big|\hat{U}_j\big|, 
\end{equation}  
where $m = \lfloor n/3 \rfloor$ 
from~(\ref{eqn:theo23}).  
The following observation follows by elementary means
(independent from $\ell$ (mod 3)), 
and plays an important role
in our proof of Lemma~\ref{lem:ext}.  

\begin{obs}
\label{obs:13.11}  
There exists $j \in \mathbb{Z}_3$ so that  
\begin{enumerate}
\item 
$\Delta_j \geq 0$; 
\item
$|I_{j+1}^{\bad}| \leq 2 \Delta_j$; 
\item  
$|U_{j+2}| \leq m + 2\Delta_j + 2$.  
\end{enumerate}  
\end{obs}  
\noindent  We say that an element $j \in \mathbb{Z}_3$ satisfying Conclusions~(1)--(3)
of Observation~\ref{obs:13.11} is {\it amenable}.

\begin{proof}[Proof of Observation~\ref{obs:13.11}]
As they are defined in~(\ref{eqn:3.14.2019.2:16p}),   
let $j \in \mathbb{Z}_3$ satisfy 
\begin{equation}
\label{eqn:3.14.2019.2:26p}  
\Delta_j = \max\{ \Delta_0, \Delta_1, \Delta_2\},  
\qquad \text{and w.l.o.g.~let $j = 0$}.  
\end{equation}  
Conclusion~(1) holds with $j = 0$ 
from~(\ref{eqn:3.14.2019.2:26p}),     
lest 
$\Delta_1, \Delta_2 \leq \Delta_0 \leq -1$ 
and $m = \lfloor n/3 \rfloor \geq (n-2)/3$ 
give
\begin{multline*}  
3m - 
\big|\hat{U}_0\big|
-
\big|\hat{U}_1\big|
-
\big|\hat{U}_2\big|
\stackrel{(\ref{eqn:3.14.2019.2:16p})}{=}    
\Delta_0 + \Delta_1 + \Delta_2 \leq -3  \\
\implies \qquad 
3m + 3 \leq 
\big|\hat{U}_0\big|
+
\big|\hat{U}_1\big|
+
\big|\hat{U}_2\big|  
\stackrel{(\ref{eqn:3.14.2019.2:15p}), 
(\ref{eqn:3.14.2019.2:16p})}{\leq}   
|U_0| + |U_1| + |U_2| = n \leq 3m + 2, 
\end{multline*}  
a contradiction.  Conclusion~(3) also holds with $j = 0$ 
from~(\ref{eqn:3.14.2019.2:26p}),     
since 
\begin{multline*}  
|U_2| = n - |U_0| - |U_1|
\stackrel{(\ref{eqn:3.14.2019.2:15p}), 
(\ref{eqn:3.14.2019.2:16p})}{=}    
n 
-
\big|I_0^{\bad}\big|
- 
\big|\hat{U}_0\big|
- 
\big|I_1^{\bad}\big|
- 
\big|\hat{U}_1\big|
\leq 
n 
-
\big|\hat{U}_0\big|
- 
\big|\hat{U}_1\big|  \\
\leq m + 2 
+ 
\Big(m 
-
\big|\hat{U}_0\big|
\Big)  
+ 
\Big(m 
-
\big|\hat{U}_1\big|
\Big)  
\stackrel{(\ref{eqn:3.14.2019.2:16p})}{=}    
m + 2 + \Delta_0 + \Delta_1 
\stackrel{(\ref{eqn:3.14.2019.2:26p})}{\leq}    
m + 2\Delta_0 + 2.  
\end{multline*}  
For sake of argument, we assume that 
Conclusion~(2) fails with $j = 0$
from~(\ref{eqn:3.14.2019.2:26p}):       
\begin{equation}
\label{eqn:3.14.2019.5:39p}
\big|I_1^{\bad}\big| \geq 2 \Delta_0 + 1.  
\end{equation}

Observe that 
\begin{multline}  
\label{eqn:3.14.2019.6:56.5p}  
\big|I_0^{\bad}\big| + 
\big|I_1^{\bad}\big|  
\leq 
\big|I_0^{\bad}\big| + 
\big|I_1^{\bad}\big| + 
\big|I_2^{\bad}\big| 
\stackrel{(\ref{eqn:3.14.2019.2:15p}),   
(\ref{eqn:3.14.2019.2:16p})}{=}    
n 
- \big|\hat{U}_0\big|
- \big|\hat{U}_1\big|
- \big|\hat{U}_2\big|  \\
\leq 
2 + 
\Big(m - \big|\hat{U}_0\big|\Big)
+
\Big(m - \big|\hat{U}_1\big|\Big)
+
\Big(m - \big|\hat{U}_2\big|\Big)
\stackrel{(\ref{eqn:3.14.2019.2:16p})}{=}    
2 + \Delta_0 + \Delta_1 + \Delta_2, \\  
\implies \qquad 
0 \leq \big|I_0^{\bad}\big| \leq 2 + \Delta_0 + \Delta_1 + \Delta_2 - \big|I_1^{\bad}\big|
\stackrel{(\ref{eqn:3.14.2019.5:39p})}{\leq}  
1 + \Delta_1 + \Delta_2 - \Delta_0.  
\end{multline}  
If $\Delta_1 \leq - 1$, then 
Observation~\ref{obs:13.11} would hold with $j = 2$.   
Indeed, $\Delta_1 \leq -1$ and 
$\Delta_2 \leq \Delta_0$  
from~(\ref{eqn:3.14.2019.2:26p}) would render 
$I_0^{\bad} = \emptyset$, and so 
$\Delta_1 = -1$, and $\Delta_2 = \Delta_0$ 
would be the maximum
from~(\ref{eqn:3.14.2019.2:26p}).   
If 
$\Delta_2 = \Delta_0$ is the maximum
from~(\ref{eqn:3.14.2019.2:26p}), 
we already observed that 
Conclusion~(1) would hold with $j = 2$, i.e., $\Delta_2 \geq 0$, 
and that Conclusion~(3) would hold with $j = 2$, i.e., $|U_1| \leq 
m + 2\Delta_2 + 2$.  
If $I_0^{\bad} = \emptyset$ and $\Delta_2 \geq 0$, 
Conclusion~(2) would hold with $j = 2$, i.e., 
$|I_0^{\bad}| = 0 \leq 2 \Delta_2$.  
For sake of argument, we assume that 
\begin{equation}
\label{eqn:3.14.2019.6:09p}  
\Delta_1 \geq 0.  
\end{equation}

Note that~(\ref{eqn:3.14.2019.6:09p}) says Conclusion~(1) holds
with $j = 1$.  
Conclusion~(3) also holds with $j = 1$, since 
\begin{multline*}  
|U_0| 
\stackrel{(\ref{eqn:3.14.2019.2:15p}),   
(\ref{eqn:3.14.2019.2:16p})}{=}  
\big|\hat{U}_0\big| + \big|I_0^{\bad}\big|  
\stackrel{(\ref{eqn:3.14.2019.6:56.5p})}{\leq}  
\big|\hat{U}_0\big| + 1 + \Delta_1 + \Delta_2 - \Delta_0 \\
\stackrel{(\ref{eqn:3.14.2019.2:26p})}{\leq}    
\big|\hat{U}_0\big| + 1 + \Delta_1  
\stackrel{(\ref{eqn:3.14.2019.2:16p})}{=}  
m - \Delta_0 + 1 + \Delta_1 
\stackrel{(\ref{eqn:3.14.2019.2:26p})}{\leq}    
m + 1 + \Delta_1 
\stackrel{(\ref{eqn:3.14.2019.6:09p})}{\leq}    
m + 2\Delta_1 + 2,   
\end{multline*}  
where we used $\Delta_0 \geq 0$.  
For sake of argument, we assume Conclusion~(2) fails with $j = 1$:  
\begin{equation}
\label{eqn:3.14.2019.8:02p}  
\big|I_2^{\bad}\big| \geq 2 \Delta_1 + 1.  
\end{equation}

We conclude that 
Observation~\ref{obs:13.11} holds with $j = 2$.   
For that, observe that 
\begin{multline*}  
\big|I_0^{\bad}\big|
+ 2 \Delta_0 + 2 \Delta_1 + 2
\stackrel{(\ref{eqn:3.14.2019.5:39p}), 
(\ref{eqn:3.14.2019.8:02p})}{\leq}      
\big|I_0^{\bad}\big| 
+
\big|I_1^{\bad}\big| 
+
\big|I_2^{\bad}\big| 
\stackrel{(\ref{eqn:3.14.2019.6:56.5p})}{\leq}    
2 + \Delta_0 + \Delta_1 + \Delta_2  \\
\implies \qquad 
0 \leq \big|I_0^{\bad}\big| 
\leq 
\Delta_2 - \Delta_0 - \Delta_1  
\stackrel{(\ref{eqn:3.14.2019.2:26p})}{\leq}    
- \Delta_1
\stackrel{(\ref{eqn:3.14.2019.6:09p})}{\leq}
0.    
\end{multline*}  
Thus, 
$I_0^{\bad} = \emptyset$, 
$\Delta_1 = 0$, and $\Delta_2 = \Delta_0$  
is the maximum 
from~(\ref{eqn:3.14.2019.2:26p}).    
As 
$\Delta_2 = \Delta_0$ is the maximum from~(\ref{eqn:3.14.2019.2:26p}),     
we already observed that 
Conclusion~(1) holds, i.e., $\Delta_2 \geq 0$, and 
that 
Conclusion~(3) holds, i.e., $|U_1| \leq m + 2\Delta_2 + 2$.      
Since $\Delta_1 = 0$ and $I_0^{\bad} = \emptyset$, Conclusion~(2) also holds, i.e., 
$|I_0^{\bad}| = 0 = 2\Delta_1$, which completes the proof  
of Observation~\ref{obs:13.11}.  
\end{proof}

Observation~\ref{obs:13.11}  
guarantees at least one {\it amenable} element 
$j \in \mathbb{Z}_3$, i.e., one where 
Conclusions~(1)--(3) of Observation~\ref{obs:13.11} hold.   
We next consider properties of an amenable $j \in \mathbb{Z}_3$  
when $\ell \equiv 1$ (mod 3).

\begin{fact}
\label{fact:13.12}  
Let $\ell \equiv 1$ {\rm (mod 3)}.  
Let $j \in \mathbb{Z}_3$ be amenable, and fix $u_{j+1} \in U_{j+1}$.   
Then
$$
\deg_G^{\spec}(u_{j+1}, U_{j+2}^{\good}) \leq 
\left\{
\begin{array}{cc}
2 \Delta_j + 5 & \text{if 
$u_{j+1} \in \hat{U}_{j+1}$
{\rm (cf.~(\ref{eqn:3.14.2019.2:16p}))},}  \\   
n/(10)  & \text{if 
$u_{j+1} \in I_{j+1}^{\bad}$
{\rm (cf.~(\ref{eqn:IjEj})}    
and~{\rm (\ref{eqn:3.14.2019.2:16p}))}.}  
\end{array}
\right.  
$$
\end{fact}

\begin{proof}[Proof of Fact~\ref{fact:13.12}]
Let $\ell \equiv 1$ (mod 3).  
Fix an amenable $j \in \mathbb{Z}_3$, and w.l.o.g.~let $j = 0$.  
Note first that 
\begin{equation}
  \label{eqn:3.14.2019.4:37p}  
  \deg_G^c(v, U_2) \geq \deg_G^c(v) - 3 \text{ 
for every $v \in \hat{U}_1$ with $\deg^{\spec}_G(v, U^{\good}_2) > 0$}, 
\end{equation}  
since 
$u_2 \in N^{\spec}_G(v, U^{\good}_2)$
gives $u_2 \in U^{\good}_2$ and $c(\{u_2, v\}) \neq c_{u_2}$, 
and 
$v \in \hat{U}_1 = U^{\good}_1 \cup E^{\bad}_1$
allows us 
to apply Statement~(4) or~(5) of Corollary~\ref{cor:13.7}.


Fix a vertex $u_1 \in U_1$, and   
first let $u_1 \in \hat{U}_1$.  Assume, on the contrary, 
that 
\begin{equation}
\label{eqn:3.15.2019.5:57p}  
\deg_G^{\spec}(u_1, U_2^{\good}) \geq 2 \Delta_0 + 6 
\stackrel{\text{Obs.\ref{obs:13.11}}}{\geq}    
6.   
\end{equation}  
Now, the edges $E_G(u_1, U_2)$ consist of $\deg_G^c(u_1, U_2)$ many distinctly colored edges 
together with some number of edges of repeated colors.  
By Corollary~\ref{cor:13.8}, the special edges 
$E_G^{\spec}(u_1, U_2^{\good})$, i.e., those of the form 
$\{u_1, u_2\} \in E$ for 
$u_2 \in N_G^{\spec}(u_1, U_2^{\good})$, come in at most two colors, so         
we have 
\begin{multline*}  
|U_2| \geq \deg_G(u_1, U_2) \geq \deg_G^c(u_1, U_2) + \deg_G^{\spec}\big(u_1, U_2^{\good}\big)
- 2  \\
\stackrel{(\ref{eqn:3.14.2019.4:37p})}{\geq}    
\deg_G^c(u_1) + \deg_G^{\spec}\big(u_1, U_2^{\good}\big) - 5
\stackrel{(\ref{eqn:3.15.2019.5:57p})}{\geq}    
\deg_G^c(u_1) + 2\Delta_0 + 1  
\stackrel{(\ref{eqn:theo23})}{\geq}    
m + 2 \Delta_0 + 3, 
\end{multline*}  
which contradicts Conclusion~(3) of 
Observation~\ref{obs:13.11}.

Now let $u_1 \in I_1^{\bad}$.  Assume, on the contrary, that 
\begin{equation}
\label{eqn:3.17.2019.4:19p}  
\deg_G^{\spec}\big(u_1, U_2^{\good}\big) > \tfrac{n}{10}.  
\end{equation}  
By Corollary~\ref{cor:13.8}, all but one of the edges
$E_G^{\spec}(u_1, U_2^{\good})$ are monochromatic, and in some color $\alpha$ of $c$.    
To prepare an upcoming application of Corollary~\ref{cor:13.10}, we define
\begin{multline}  
\label{eqn:3.18.2019.12:44p}  
A_2 =    
A_2(u_1) = 
\left\{v_2 \in N_G^{\spec}\big(u_1, U_2^{\good}\big): \, 
c(\{u_1, v_2\}) = \alpha
\right\}, \\ 
\text{where} \qquad 
|A_2| 
\stackrel{\text{Cor.\ref{cor:13.8}}}{\geq}
\deg_G^{\spec}\big(u_1, U_2^{\good}\big) - 1
\stackrel{(\ref{eqn:3.17.2019.4:19p})}{>}
\tfrac{n}{10} - 1 \geq \tfrac{n}{11}.    
\end{multline}  
For Corollary~\ref{cor:13.10}, we will 
identify a set $B_2 \subseteq
U_2^{\good}$ corresponding to $A_2$ above\footnote{Strictly speaking, the set $B_2$ we will define below 
will be a subset
of that in the hypothesis of Corollary~\ref{cor:13.10}.}.    
For that, we first consider the following
superset
$\overline{B}_2 \supseteq B_2$, from which we will later extract $B_2$:  
\begin{equation}
\label{eqn:3.18.2019.1:28p}  
\overline{B}_2 = N_G^{\typ} \big(u_1, U_2^{\good}\big), 
\qquad \text{where} \qquad 
\big|\overline{B}_2\big| 
\stackrel{\text{Cor.\ref{cor:13.8}}}{\geq}    
\big(\tfrac{1}{6} - 110 \lambda^{1/4}\big)n.  
\end{equation}

For fixed $u_2 \in \overline{B}_2$,
\begin{equation}
\label{eqn:3.18.2019.1:14p}
\deg_G^{\typ}(v_1, A_2) > 0 \text{ for every  $v_1 \in N^{\spec}_{G}(u_2, \hat{U_1}$)},
\end{equation}  
since then $\deg^{\spec}_G(v_1, U^{\good}_2) > 0$ (with $u_2 \in N^{\spec}_G(v_1, U^{\good}_2)$) and 
inclusion-exclusion gives 
\begin{multline*}  
U_2^{\good} \supseteq A_2 \cup N_G^{\typ}\big(v_1, U_2^{\good}\big)
\qquad \implies \qquad 
\big|U_2^{\good}\big| \geq |A_2| + 
\deg_G^{\typ}\big(v_1, U_2^{\good}\big)
- 
\deg_G^{\typ}(v_1, A_2)  \\
\implies \qquad 
\deg_G^{\typ}(v_1, A_2)  
\geq |A_2| + 
\deg_G^{\typ}\big(v_1, U_2^{\good}\big)
- 
\big|U_2^{\good}\big| 
\stackrel{(\ref{eqn:3.18.2019.12:44p})}{>}
\tfrac{n}{11}   
+ 
\deg_G^{\typ}\big(v_1, U_2^{\good}\big)
- 
\big|U_2^{\good}\big|  \\
\stackrel{(\ref{eqn:Clm13.2})}{\geq}    
\tfrac{n}{11} + 
\deg_G^{\typ}\big(v_1, U_2^{\good}\big)
- 
\big(\tfrac{1}{3} + 75 \lambda^{1/4}\big)n  
\stackrel{\text{Cor.\ref{cor:13.8}}}{\geq}    
\tfrac{n}{11} + 
\deg_G^c\big(v_1, U_2^{\good}\big)
- 
3  
- 
\big(\tfrac{1}{3} + 75 \lambda^{1/4}\big)n \\  
\stackrel{(\ref{eqn:Clm13.2})}{\geq}    
\tfrac{n}{11} + \deg_G^c(v_1, U_2) - 3 - \big(\tfrac{1}{3} + 147 \lambda^{1/4}\big)n  
\stackrel{\text{\eqref{eqn:3.14.2019.4:37p}}}{\geq}    
\tfrac{n}{11} + \deg_G^c(v_1) - 6 - \big(\tfrac{1}{3} + 147 \lambda^{1/4}\big)n  \\
\geq 
\tfrac{n}{11} + \delta_c(G)  - 6 - \big(\tfrac{1}{3} + 147 \lambda^{1/4}\big)n  
\stackrel{(\ref{eqn:overarching2})}{\geq}    
\tfrac{n}{11} - 6 - 147 \lambda^{1/4}n  
\geq 
\tfrac{n}{11} - 148 \lambda^{1/4}n  
\stackrel{(\ref{eqn:5.8.2019.1:39p})}{>} 0. 
\end{multline*}  

We can now show that 
\begin{equation}\label{fixedalpha}
  c(\{u_2, v_1\}) = \alpha  \text{ for every } v_1 \in N_G^{\spec}(u_2, \hat{U}_1). 
\end{equation}
Indeed, by \eqref{eqn:3.18.2019.1:14p},
there exists $v_2 \in N_G^{\typ}(v_1, A_2)$. 
Note that $v_2 \in A_2$ implies that $\{u_1, v_2\}$ is a special
edge with color $\alpha$.
Furthemore, 
both $u_2 \in \overline{B}_2$ and $v_2 \in A_2$ are good vertices, 
both $\{u_2, u_1\}$ and $\{v_2, v_1\}$ are typical edges, and
both $\{u_2, v_1\}$ and $\{v_2, u_1\}$ are special edges.  Thus, 
the 4-cycle $(u_1, u_2, v_1, v_2)$
is $2$-special (cf. Definition~\ref{def:jspecial}), 
so 
Observation~\ref{obs:13.9} guarantees that 
$c(\{u_2, v_1\}) = c(\{v_2, u_1\}) = \alpha$.

To extract the desired subset $B_2 \subseteq \overline{B}_2$
from~(\ref{eqn:3.18.2019.1:28p}), we double-count
$$
Z = \left\{\{u_2, v_1\} \in E: \, u_2 \in \overline{B}_2, \, 
v_1 \in N_G^{\spec}(u_2, I_1^{\bad}), \,  
\text{ and }  c(\{v_1, u_2\}) \neq \alpha
\right\}.
$$
For each $u_2 \in \overline{B}_2$, 
Statement (2) of Corollary~\ref{cor:13.7} ensures that $E_G(u_2, U_1)$ admits at least 
$$
\deg_G^c(u_2) - 1 - \big|U_0 \setminus I_0^{\bad}\big|
\stackrel{(\ref{eqn:3.14.2019.2:16p})}{=}    
\deg_G^c(u_2) - 1 - \big|\hat{U}_0\big|
$$  
many special colors for $u_2$.  By \eqref{fixedalpha},
every special edge $\{u_2, v_1\} \in E_G(u_2, \hat{U}_1)$ is colored $c(\{u_2, v_1\}) = \alpha$, 
which is forbidden in $Z$.  Thus, 
$$
|Z| \geq \sum_{u_2 \in \overline{B}_2} \Big(\deg_G^c(u_2) - 2 - \big|\hat{U}_0\big|\Big)    
\stackrel{(\ref{eqn:theo23})}{\geq}    
\sum_{u_2 \in \overline{B}_2} \Big(m - \big|\hat{U}_0\big| \Big)  
\stackrel{(\ref{eqn:3.14.2019.2:16p})}{=}    
\sum_{u_2 \in \overline{B}_2} \Delta_0 
= \Delta_0 \big|\overline{B}_2\big|  
\stackrel{\text{Obs.\ref{obs:13.11}}}{\geq}
\tfrac{1}{2} \big|I_1^{\bad}\big| \big|\overline{B}_2\big|.    
$$
Averaging $|Z|$ over $I_1^{\bad}$, 
we infer the existence of a vertex $v_1 \in I_1^{\bad}$ where 
\begin{equation}
\label{eqn:3.25.2019.7:06p}  
B_2 = 
B_2(v_1) = \left\{u_2 \in N_G^{\spec}\big(v_1, \overline{B}_2\big): \, c(\{u_2, v_1\})
\neq \alpha\right\} 
\end{equation}  
satisfies  
\begin{equation}
\label{eqn:3.25.2019.7:10p} 
|B_2|  
\geq \tfrac{1}{2} \big|\overline{B}_2\big|  
\stackrel{(\ref{eqn:3.18.2019.1:28p})}{\geq}    
\tfrac{1}{2} \big(\tfrac{1}{6} - 110 \lambda^{1/4}\big) n.  
\end{equation}

Consider the sets $A_2 = A_2(u_1)$ and $B_2 = B_2(v_1)$
from~(\ref{eqn:3.18.2019.12:44p})  
and~(\ref{eqn:3.25.2019.7:06p}). 
Since $B_2 \subseteq \overline{B}_2$  
where $A_2 \cap \overline{B}_2 = \emptyset$
from~(\ref{eqn:3.18.2019.1:28p}),   
we infer that $A_2 \cup B_2$
is a disjoint union of size 
$$
|A_2 \cup B_2| = |A_2| + |B_2|  
\stackrel{(\ref{eqn:3.18.2019.12:44p})}{\geq}    
\tfrac{n}{11} + |B_2| 
\stackrel{(\ref{eqn:3.25.2019.7:10p})}{\geq}   
\tfrac{n}{11} + \tfrac{1}{2} \big(\tfrac{1}{6} - 110 \lambda^{1/4}\big)n
= 
\big(\tfrac{23}{132} - 55 \lambda^{1/4}\big)n 
\stackrel{(\ref{eqn:5.8.2019.1:39p})}{>}  
\big(\tfrac{1}{6} + 186 \lambda^{1/4}\big) n, 
$$
which contradicts Corollary~\ref{cor:13.10},   
and concludes the proof of Fact~\ref{fact:13.12}.  
\end{proof}

Fact~\ref{fact:13.12} admits the following easy but useful corollary.

\begin{cor}
\label{cor:13.14}  
Let $\ell \equiv 1$ {\rm (mod 3)}, and fix an amenable element $j \in \mathbb{Z}_3$.  
Fix an integer $\Delta \geq \max\{1, \Delta_j\}$, and fix $W_{j+1} \subseteq U_{j+1}$ of size 
$|W_{j+1}| < n / (100)$.  Then 
$$
\Big|E_G^{\spec}\big(W_{j+1}, U_{j+2}^{\good}\big)\Big|
\leq \tfrac{3}{10} \Delta n, 
$$
where 
$E_G^{\spec}(W_{j+1}, U_{j+2}^{\good})$
includes all $\{w_{j+1}, u_{j+2}\} \in E$ with $w_{j+1} \in W_{j+1}$ and $u_{j+2} \in N_G^{\spec}(w_{j+1}, U_{j+2}^{\good})$.  
\end{cor}

\begin{proof}[Proof of Corollary~\ref{cor:13.14}]  
Let $\ell \equiv 1$ (mod 3).  Fix an amenable element $j \in \mathbb{Z}_3$, and w.l.o.g.~let $j = 0$.  
Fix an integer $\Delta \geq \max \{1, \Delta_0\}$, and fix $W_1 \subseteq U_1$ of size $|W_1| < n/(100)$.  
Then 
\begin{multline*}  
\Big|E_G^{\spec}\big(W_1, U_2^{\good}\big)\Big|
= 
\sum_{w_1 \in W_1} \deg_G^{\spec}\big(w_1, U_2^{\good}\big)  
= 
\sum_{w_1 \in W_1 \cap \hat{U}_1} \deg_G^{\spec}\big(w_1, U_2^{\good}\big) \\
+
\sum_{w_1 \in W_1 \cap I_1^{\bad}} \deg_G^{\spec}\big(w_1, U_2^{\good}\big) 
\stackrel{\text{Fct.\ref{fact:13.12}}}{\leq}
\big|W_1 \cap \hat{U}_1\big| (2 \Delta_0 + 5)
+ 
\big|W_1 \cap I^{\bad}_1\big| 
\tfrac{n}{10}  \\
\leq 
|W_1| (2 \Delta_0 + 5)
+ 
\big|I^{\bad}_1\big| 
\tfrac{n}{10}  
\stackrel{\text{Obs.\ref{obs:13.11}}}{\leq}    
|W_1| (2 \Delta_0 + 5) +  \tfrac{n}{5} \Delta_0 
\leq 
\tfrac{n}{100} (2 \Delta + 5) + \tfrac{n}{5} \Delta,  
\end{multline*}  
and the quantity above is at most $27 \Delta n / 100$.  
\end{proof}

\subsection{Proof of Lemma~\ref{lem:ext} in the case $\boldsymbol{\ell \equiv 1}$ (mod 3)}  
We now prove Lemma~\ref{lem:ext}  
in the case $\ell \equiv 1$ (mod 3).  
For this case, 
recall from~(\ref{eqn:overarching2})  
that 
the hypotheses of Lemma~\ref{lem:ext} 
assume
\begin{equation}  
\label{eqn:2.23.2019.1:07p}  
\delta^c(G) \geq 
\left\{
\begin{array}{cc}
(n+5)/3  &   \text{in Statement~(1) with $\ell \equiv 1$ (mod 3),}  \\
(n+4)/3  &   \text{in Statement~(2) with $\ell \equiv 1$ (mod 3).}
\end{array}
\right.
\end{equation}  
We begin our work with Statement~(2).

\subsubsection{Statement~(2) of Lemma~\ref{lem:ext}}  
Statement~(2)
of Lemma~\ref{lem:ext} seeks 
to conclude that 
$(G, c)$ admits a properly colored $\ell$-cycle
$C_{\ell}$.   
To prove this, 
we proceed by fixing an amenable 
element $j \in \mathbb{Z}_3$
from Observation~\ref{obs:13.11}.
We first claim that 
\begin{equation}
\label{eqn:3.30.2019.4:42p}  
E_G^{\spec}(U_{j+1}^{\good}, U_{j+2}^{\good})
\neq \emptyset.   
\end{equation}  
To prove~(\ref{eqn:3.30.2019.4:42p}),   
we 
consider the identity
\begin{equation}
\label{eqn:3.310.2019.5:46p}  
\big| E_G^{\spec}\big(U_{j+1}, U_{j+2}^{\good}\big) \big| = 
\big| E_G^{\spec}\big(U_{j+1}^{\bad}, U_{j+2}^{\good}\big) \big| 
+ 
\big| E_G^{\spec}\big(U_{j+1}^{\good}, U_{j+2}^{\good}\big) \big|.      
\end{equation}  
We will use 
Corollary~\ref{cor:13.14}
to bound $| E_G^{\spec}(U_{j+1}^{\bad}, U_{j+2}^{\good})|$, 
and 
we will use 
Corollary~\ref{cor:13.7}  
to bound 
$|E_G^{\spec}(U_{j+1}, U_{j+2}^{\good})|$.  
First, in 
the context of Corollary~\ref{cor:13.14}, 
we set 
$\Delta = 1 + \Delta_j \geq \max\{1, \Delta_j\}$, where we used 
$\Delta_j \geq 0$ 
from the amenability of $j \in \mathbb{Z}_3$.  We also 
set $W_{j+1} = U_{j+1}^{\bad}$, where 
$$
|W_{j+1}| = \big|U_{j+1}^{\bad}\big| 
\stackrel{(\ref{eqn:Clm13.2})}{\leq} 72 \lambda^{1/4} n \stackrel{(\ref{eqn:5.8.2019.1:39p})}{<} \tfrac{n}{100}.   
$$  
Consequently, 
Corollary~\ref{cor:13.14} guarantees 
\begin{equation}
\label{eqn:3.30.2019.4:05p}  
\big|E_G^{\spec}\big(U_{j+1}^{\bad}, U_{j+2}^{\good}\big) \big| \leq \tfrac{3}{10} \Delta n.  
\end{equation}  
Second, Statement~(2) of 
Corollary~\ref{cor:13.7}  
guarantees that every $u_{j+2} \in U_{j+2}^{\good}$ satisfies 
\begin{multline}  
\label{eqn:4.1.2019.11:25a}  
\big|E_G^{\spec}(u_{j+2}, U_{j+1})\big| \geq \deg_G^c(u_{j+2}) - 1 - \big|U_j \setminus I_j^{\bad}\big|
\stackrel{(\ref{eqn:3.14.2019.2:16p})}{=}   
\deg_G^c(u_{j+2}) - 1 - \big|\hat{U}_j\big|  \\
\stackrel{(\ref{eqn:theo23})}{\geq}    
m + 1 - \big|\hat{U}_j\big|  
\stackrel{(\ref{eqn:3.14.2019.2:16p})}{=}  
1 + \Delta_j
= \Delta,
\end{multline}  
and so 
\begin{equation}  
\label{eqn:3.30.2019.4:16p}  
\big| E_G^{\spec}\big(U_{j+1}, U_{j+2}^{\good}\big) \big| = 
\sum_{u_{j+2} \in U_{j+2}^{\good}} \deg_G^{\spec}(u_{j+2}, U_{j+1}) \geq \Delta \big|U_{j+2}^{\good}\big|.  
\end{equation}  
Applying~(\ref{eqn:3.30.2019.4:05p})
and~(\ref{eqn:3.30.2019.4:16p})
to~(\ref{eqn:3.310.2019.5:46p}) yields 
$$
\Delta \big|U_{j+2}^{\good}\big| 
\leq 
\big| E_G^{\spec}\big(U_{j+1}, U_{j+2}^{\good}\big) \big| 
\leq 
\tfrac{3}{10} \Delta n + 
\big| E_G^{\spec}\big(U_{j+1}^{\good}, U_{j+2}^{\good}\big) \big|,  
$$
and so 
$$
\big| E_G^{\spec}\big(U_{j+1}^{\good}, U_{j+2}^{\good}\big) \big| 
\geq 
\Delta \big(\big|U_{j+2}^{\good}\big| - \tfrac{3}{10}n \big)  \\
\stackrel{(\ref{eqn:Clm13.2})}{\geq}  
\Delta n \big(\tfrac{1}{3} - 75 \lambda^{1/4} - \tfrac{3}{10}\big)  
> \Delta n \big(\tfrac{3}{100} - 75 \lambda^{1/4} \big) 
\stackrel{(\ref{eqn:5.8.2019.1:39p})}{>}  0, 
$$
where we used $\Delta = 1 + \Delta_j \geq 1$ 
from the amenability of $j \in \mathbb{Z}_3$.  
This proves~(\ref{eqn:3.30.2019.4:42p}).

To prove Statement~(2) of Lemma~\ref{lem:ext},   
fix an edge $\{u_{j+1}, u_{j+2}\} \in E_G^{\spec}(U_{j+1}^{\good}, U_{j+2}^{\good})$ 
from~(\ref{eqn:3.30.2019.4:42p}), 
where $u_{j+1} \in U_{j+1}^{\good}$ and $u_{j+2} \in U_{j+2}^{\good}$.  
We claim that 
\begin{equation}
\label{eqn:3.30.2019.6:34p}  
E_G^{\spec}\Big(N_G^{\typ}(u_{j+2}, U_{j+1}), 
N_G^{\typ}\big(u_{j+1}, U_{j+2}^{\good}\big)\Big) \neq \emptyset.     
\end{equation}  
If~(\ref{eqn:3.30.2019.6:34p}) holds, then   
it concludes our proof, as follows.  
Fix 
$\{v_{j+1}, v_{j+2}\} \in E$ 
of~(\ref{eqn:3.30.2019.6:34p}),  
where 
$v_{j+1} \in N_G^{\typ}(u_{j+2}, U_{j+1})$ 
and $v_{j+2} \in N_G^{\typ}(u_{j+1}, U_{j+2}^{\good})$.   
We first observe
that $(u_{j+2}, v_{j+1}, v_{j+2}, u_{j+1})$ is a 
$(j+2)$-special 4-cycle (cf.~Definition~\ref{def:jspecial}).  
Indeed, 
$u_{j+2} \in U_{j+2}^{\good}$ is good from~(\ref{eqn:3.30.2019.4:42p}) and   
$v_{j+2} \in U_{j+2}^{\good}$ is good from~(\ref{eqn:3.30.2019.6:34p}).  
The edge $\{u_{j+2}, v_{j+1}\} \in E$ is typical because $v_{j+1} \in N_G^{\typ}(u_{j+2}, U_{j+1})$
from~(\ref{eqn:3.30.2019.6:34p}), and the edge $\{v_{j+2}, u_{j+1}\} \in E$ is typical 
because $v_{j+2} \in N^{\typ}_G(u_{j+1}, U_{j+2}^{\good})$   
from~(\ref{eqn:3.30.2019.6:34p}).
The edge $\{u_{j+2}, u_{j+1}\} \in E$ is special from
from~(\ref{eqn:3.30.2019.4:42p}), 
and   
the edge $\{v_{j+2}, v_{j+1}\} \in E$ is special from~(\ref{eqn:3.30.2019.6:34p}).    
Since the 4-cycle $(u_{j+2}, v_{j+1}, v_{j+2}, u_{j+1})$ is $(j+2)$-special, Observation~\ref{obs:13.9}  
guarantees that its edges 
receive precisely 3-colors,   
where the special edges $(\{u_{j+2}, u_{j+1}\})$
and   
$\{v_{j+2}, v_{j+1}\}$ 
match in color and the typical edges $\{u_{j+2}, v_{j+1}\}$ and $\{v_{j+2}, u_{j+1}\}$ do not.
Thus, $(u_{j+2}, v_{j+1}, v_{j+2}, u_{j+1})$ is a properly colored 4-cycle.  Equivalently, 
$(u_{j+2}, u_{j+1}, v_{j+2}, v_{j+1})$ is a strong properly colored 4-cycle which 
Proposition~\ref{prop:Clm13.5}  
extends to a strong properly colored $\ell$-cycle $C_{\ell}$, as promised by 
Statement~(2)   
of Lemma~\ref{lem:ext}.

To prove~(\ref{eqn:3.30.2019.6:34p}), we proceed similarly to~(\ref{eqn:3.30.2019.4:42p}), 
and   
begin 
by considering the identity
\begin{multline}  
\label{eqn:3.31.2019.6:07p}  
\big|E_G^{\spec}\big(U_{j+1}, N_G^{\typ}\big(u_{j+1}, U_{j+2}^{\good}\big) \big) \big|
= 
\big|E_G^{\spec}\big(U_{j+1} \setminus N_G^{\typ}(u_{j+2}, U_{j+1}),  
N_G^{\typ}\big(u_{j+1}, U_{j+2}^{\good}\big) \big) \big| \\
+ 
\big|E_G^{\spec}\big(N_G^{\typ}(u_{j+2}, U_{j+1}),  
N_G^{\typ}\big(u_{j+1}, U_{j+2}^{\good}\big) \big) \big|.  
\end{multline}  
As before, 
Corollary~\ref{cor:13.14} 
will bound the first summand
of~(\ref{eqn:3.31.2019.6:07p}), 
and Corollary~\ref{cor:13.7} will bound the left hand side 
of~(\ref{eqn:3.31.2019.6:07p}).   
First, we again set $\Delta = 1 + \Delta_j \geq 0$, 
but we now set   
$W_{j+1} = U_{j+1} \setminus N_G^{\typ}(u_{j+2}, U_{j+1})$.    
Since $u_{j+2} \in U_{j+2}^{\good}$ is a good vertex,  
Proposition~\ref{prop:Clm13.3} guarantees
$$
|W_{j+1}| = 
\big|U_{j+1} \setminus N_G^{\typ}(u_{j+2}, U_{j+1})\big|    
\leq 
233 \lambda^{1/4} n \stackrel{(\ref{eqn:5.8.2019.1:39p})}{<} \tfrac{n}{100}.  
$$ 
Consequently, 
Corollary~\ref{cor:13.14} 
guarantees
\begin{multline}  
\label{eqn:3.31.2019.6:18p}  
\big|E_G^{\spec}\big(U_{j+1} \setminus N_G^{\typ}(u_{j+2}, U_{j+1}),  
N_G^{\typ}\big(u_{j+1}, U_{j+2}^{\good}\big) \big) \big|   \\
\leq 
\big|E_G^{\spec}\big(U_{j+1} \setminus N_G^{\typ}(u_{j+2}, U_{j+1}),  
U_{j+2}^{\good}\big) \big) \big| 
\leq
\tfrac{3}{10} \Delta n.
\end{multline}  
Second, and 
identically to~(\ref{eqn:4.1.2019.11:25a})  
and~(\ref{eqn:3.30.2019.4:16p}),   
\begin{multline}  
\label{eqn:3.31.2019.6:57p}  
\big|E_G^{\spec}\big(U_{j+1}, N_G^{\typ}\big(u_{j+1}, U_{j+2}^{\good}\big) \big) \big|  \\
= \sum_{v_{j+2} \in 
N_G^{\typ}(u_{j+1}, U_{j+2}^{\good})}
\deg_G^{\spec}(v_{j+2}, U_{j+1})  
\geq \Delta 
\big|N_G^{\typ}\big(u_{j+1}, U_{j+2}^{\good}\big)\big|. 
\end{multline}  
Applying~(\ref{eqn:3.31.2019.6:18p})   
and~(\ref{eqn:3.31.2019.6:57p})  
to~(\ref{eqn:3.31.2019.6:07p}) yields 
\begin{multline}  
\label{eqn:4.1.2019.12:05p}  
\big|E_G^{\spec}\big(N_G^{\typ}(u_{j+2}, U_{j+1}),  
N_G^{\typ}\big(u_{j+1}, U_{j+2}^{\good}\big) \big) \big|  
\geq 
\Delta
\big(
\big|N_G^{\typ}\big(u_{j+1}, U_{j+2}^{\good}\big)\big| 
- \tfrac{3}{10} n\big)  \\
\stackrel{\text{Cor.\ref{cor:13.8}}}{\geq}    
\Delta
\big(
\deg_G^c\big(u_{j+1}, U_{j+2}^{\good}\big) - 3 - \tfrac{3}{10} n \big)
\stackrel{(\ref{eqn:Clm13.2})}{\geq}    
\Delta
\big(
\big(\tfrac{1}{3} - 76 \lambda^{1/4}\big) n - 3 - \tfrac{3}{10} n \big)  \\
> 
\Delta n \big(\tfrac{3}{100} - 76 \lambda^{1/4} \big) 
\stackrel{(\ref{eqn:5.8.2019.1:39p})}{>} 0, 
\end{multline}  
where we used that $\Delta = 1 + \Delta_j \geq 1$ and that $n$ is sufficiently large.  
This proves~(\ref{eqn:3.30.2019.6:34p}), and completes the proof of Statement~(2) of 
Lemma~\ref{lem:ext}.

\subsubsection{Statement~(1) of Lemma~\ref{lem:ext}}  Statement~(1) of Lemma~\ref{lem:ext}
assumes that $\delta^c(G) \geq (n+5)/3$ and 
seeks to conclude
that $(G, c)$ admits\footnote{Throughout our proof, we have assumed in~(\ref{eqn:overarching1})   
that $(G, c)$ avoids rainbow $\ell$-cycles $C_{\ell}$.  Finding one now shows 
that our assumption~(\ref{eqn:overarching1}) is flawed.}  
a rainbow $\ell$-cycle $C_{\ell}$.   
The argument here is similar to that of the previous subsection, where in fact we build upon
that same argument.  For that, note 
that $\delta^c(G) \geq (n+5)/3 \geq (n+4)/3$
allows 
all conclusions of the previous subsection to
hold
for the amenable element $j \in \mathbb{Z}_3$.    
As before, let 
$\{u_{j+1}, u_{j+2}\} \in E_G^{\spec}(U_{j+1}^{\good}, U_{j+2}^{\good})$ be fixed.  
We first observe that 
\begin{equation}
\label{eqn:3.30.2019.3:36p}  
\Delta_{j+2} \leq - 1,  
\end{equation}  
since 
\begin{multline*}  
\big|\hat{U}_{j+2}\big| \geq \deg_G^c\big(u_{j+1}, \hat{U}_{j+2}\big)
= 
\deg_G^c\big(u_{j+1}, U_{j+2} \setminus I_{j+2}^{\bad}\big) 
\stackrel{\text{Cor.\ref{cor:13.7}}}{\geq}    
\deg_G^c(u_{j+1}) - 1
\stackrel{(\ref{eqn:theo23})}{\geq}    
m + 1  \\
\implies \qquad 
-1 \geq m - 
\big|\hat{U}_{j+2}\big| 
\stackrel{(\ref{eqn:3.14.2019.2:16p})}{=}  \Delta_{j+2}.     
\end{multline*}  
Second, we 
observe that 
\begin{equation}
\label{eqn:4.1.2019.12:39p}  
n \equiv 2 \text{ {\rm (mod 3)}} \qquad \text{or}  \qquad 
\Delta_j \geq 1.
\end{equation}  
To argue~(\ref{eqn:4.1.2019.12:39p}),   
\begin{equation}
\label{eqn:4.1.2019.12:50p}  
\text{{\it we assume, on the contrary, that 
$n \not\equiv 2$ {\rm (mod 3)} and $\Delta_j = 0$.}} 
\end{equation}
From~(\ref{eqn:4.1.2019.12:50p}),   
we will conclude that 
$j+1 \in \mathbb{Z}_3$ is also amenable, whence~(\ref{eqn:3.30.2019.3:36p}) 
also holds for $j+1 \in \mathbb{Z}_3$, 
in which case $\Delta_{j+1+2} = \Delta_j \leq - 1$ contradicts
$\Delta_j = 0$ 
of~(\ref{eqn:4.1.2019.12:50p}).    
To see that $j + 1 \in \mathbb{Z}_3$ is amenable, 
we note 
from~(\ref{eqn:3.14.2019.2:16p})   
that  
\begin{equation}
\label{eqn:4.1.2019.1:41p}  
\big|\hat{U}_j\big|
\stackrel{(\ref{eqn:4.1.2019.12:50p})}{=} m, \quad \text{and} \qquad    
\big|\hat{U}_{j+2}\big| 
\stackrel{(\ref{eqn:3.30.2019.3:36p})}{\geq}    
m + 1
\qquad \implies \qquad 
\big|\hat{U}_{j+1}\big| \leq m, 
\end{equation}  
lest $3m + 2 \leq |\hat{U_0}| + |\hat{U}_1| + |\hat{U_2}| \leq n$
contradicts~(\ref{eqn:4.1.2019.12:50p})  
(recall $m = \lfloor n/3 \rfloor$ from~(\ref{eqn:theo23})).     
Thus, 
\begin{equation}
\label{eqn:4.1.2019.1:34p}  
\Delta_{j+1} 
\stackrel{(\ref{eqn:3.14.2019.2:16p})}{=} 
m - \big|\hat{U}_{j+1}\big| \geq 0    
\end{equation}  
satisfies the first condition of amenability
in Observation~\ref{obs:13.11}.    
Moreover, 
\begin{multline}  
\label{eqn:4.1.2019.1:26p}  
\big|I_j^{\bad}\big|, 
\big|I_{j+2}^{\bad}\big| 
\leq 
\sum_{k \in \mathbb{Z}_3}  
\big|I_k^{\bad}\big|  
\stackrel{(\ref{eqn:3.14.2019.2:15p}), (\ref{eqn:3.14.2019.2:16p})}{=}  
n - 
\sum_{k \in \mathbb{Z}_3}  
\big|\hat{U}_k\big|
\stackrel{(\ref{eqn:4.1.2019.12:50p})}{\leq}    
1 + 
\sum_{k \in \mathbb{Z}_3}  
\Big(m - 
\big|\hat{U}_k\big|\Big)  \\
\stackrel{(\ref{eqn:3.14.2019.2:16p})}{=}  
1 + 
\sum_{k \in \mathbb{Z}_3} 
\Delta_k 
= 
1 + \Delta_j + \Delta_{j+1} + \Delta_{j+2}  
\stackrel{(\ref{eqn:3.30.2019.3:36p})}{\leq}    
\Delta_j + \Delta_{j+1}  
\stackrel{(\ref{eqn:4.1.2019.12:50p})}{=}  
\Delta_{j+1} 
\stackrel{(\ref{eqn:4.1.2019.1:34p})}{\leq}  
2 \Delta_{j+1},      
\end{multline}  
and so $|I_{j+2}^{\bad}| \leq 2 \Delta_{j+1}$ satisfies the second condition of amenability
in Observation~\ref{obs:13.11}. 
Finally, 
$$
|U_{j+3}| = |U_j| 
\stackrel{(\ref{eqn:3.14.2019.2:15p}), (\ref{eqn:3.14.2019.2:16p})}{=}  
\big|\hat{U}_j\big| + \big|I_j^{\bad}\big|
\stackrel{(\ref{eqn:4.1.2019.12:50p})}{=}    
m + 
\big|I_j^{\bad}\big|
\stackrel{(\ref{eqn:4.1.2019.1:26p})}{\leq}    
m + 2 \Delta_{j+1}
\leq 
m + 2 \Delta_{j+1} + 2, 
$$
and so $|U_{j+3}| \leq m + 2 \Delta_{j+1} + 2$ satisfies the third condition of 
amenability in Observation~\ref{obs:13.11}.

The remainder of our proof for 
Statement~(1) of Lemma~\ref{lem:ext} splits into the two cases
of~(\ref{eqn:4.1.2019.12:39p}).    
For these, recall that 
$\{u_{j+1}, u_{j+2}\} \in E_G^{\spec}(U_{j+1}^{\good}, U_{j+2}^{\good})$ was fixed
at the start of this proof, where we now set   
$c(\{u_{j+1}, u_{j+2}\}) = \alpha$ for $\alpha \neq c_{u_{j+2}}$
on account that 
$\{u_{j+1}, u_{j+2}\}$ is a special edge.  \\

\noindent {\bf Case 1 ($n \equiv 2$ {\rm (mod 3)}).}  
We revisit~(\ref{eqn:3.30.2019.6:34p}) by confirming that   
\begin{multline}  
\label{eqn:4.1.2019.11:49a}  
\emptyset \neq 
E_{G, \neg \alpha}^{\spec}\big(N_G^{\typ}(u_{j+2}, U_{j+1}), 
N_G^{\typ}\big(u_{j+1}, U_{j+2}^{\good}\big)\big) \\
= 
\big\{
\{v_{j+1}, v_{j+2}\} \in 
E_G^{\spec}\big(N_G^{\typ}(u_{j+2}, U_{j+1}), 
N_G^{\typ}\big(u_{j+1}, U_{j+2}^{\good}\big)\big) 
: 
c(\{v_{j+1}, v_{j+2}\}) \neq \alpha
\big\}, 
\end{multline}  
where the set above consists of those edges 
of~(\ref{eqn:3.30.2019.6:34p}) which are not colored $\alpha$.  
If true, then 
any $\{v_{j+1}, v_{j+2}\} \in E$  
of~(\ref{eqn:4.1.2019.11:49a})   
gives a strong rainbow 4-cycle 
$(u_{j+2}, u_{j+1}, v_{j+2}, v_{j+1})$ 
(recall Observation~\ref{obs:13.9}) 
which    
Proposition~\ref{prop:Clm13.5}  
extends to a strong rainbow $\ell$-cycle $C_{\ell}$, as promised by
Statement~(1) of Lemma~\ref{lem:ext}.  
To see~(\ref{eqn:4.1.2019.11:49a}),   
we replay the details 
of~(\ref{eqn:3.31.2019.6:07p})--(\ref{eqn:4.1.2019.12:05p}) with the added hypothesis 
$n = 3m + 2$.    
We again have 
\begin{multline}  
\label{eqn:4.1.2019.2:02p}  
\big|E_{G, \neg \alpha}^{\spec}\big(U_{j+1}, N_G^{\typ}\big(u_{j+1}, U_{j+2}^{\good}\big) \big) \big|
= 
\big|E_{G, \neg \alpha}^{\spec}\big(U_{j+1} \setminus N_G^{\typ}(u_{j+2}, U_{j+1}),  
N_G^{\typ}\big(u_{j+1}, U_{j+2}^{\good}\big) \big) \big| \\
+ 
\big|E_{G, \neg \alpha}^{\spec}\big(N_G^{\typ}(u_{j+2}, U_{j+1}),  
N_G^{\typ}\big(u_{j+1}, U_{j+2}^{\good}\big) \big) \big|,   
\end{multline}  
where the left hand side 
and first summand 
of~(\ref{eqn:4.1.2019.2:02p}) are defined analogously   
to~(\ref{eqn:4.1.2019.11:49a}).    
Clearly, 
\begin{multline}  
\label{eqn:4.1.2019.2:05p}  
\big|E_{G, \neg \alpha}^{\spec}\big(U_{j+1} \setminus N_G^{\typ}(u_{j+2}, U_{j+1}),  
N_G^{\typ}\big(u_{j+1}, U_{j+2}^{\good}\big) \big) \big|  \\
\leq 
\big|E_G^{\spec}\big(U_{j+1} \setminus N_G^{\typ}(u_{j+2}, U_{j+1}),  
N_G^{\typ}\big(u_{j+1}, U_{j+2}^{\good}\big) \big) \big| 
\stackrel{(\ref{eqn:3.31.2019.6:18p})}{\leq}
\tfrac{3}{10} \Delta n.  
\end{multline}  
Moreover, 
$\delta^c(G) \geq (n+5)/3$ 
for $n \equiv 2$ (mod 3) ensures 
$\delta^c(G) \geq m+3$, and so 
Statement~(2) of 
Corollary~\ref{cor:13.7} guarantees that every $v_{j+2} \in U_{j+2}^{\good}$ satisfies
\begin{multline}  
\label{eqn:4.1.2019.11:42a}  
\big|E_G^{\spec}(v_{j+2}, U_{j+1})\big| \geq \deg_G^c(v_{j+2}) - 1 - \big|U_j \setminus I_j^{\bad}\big|
\stackrel{(\ref{eqn:3.14.2019.2:16p})}{=}   
\deg_G^c(v_{j+2}) - 1 - \big|\hat{U}_j\big|  \\
\geq 
m + 2 - \big|\hat{U}_j\big|  
\stackrel{(\ref{eqn:3.14.2019.2:16p})}{=}  
2 + \Delta_j
= 1 + \Delta.  
\end{multline}  
Consequently, 
those edges 
$E_{G, \neg \alpha}^{\spec}(v_{j+2}, U_{j+1})$
above not colored $\alpha$ satisfy 
$|E_{G, \neg \alpha}^{\spec}(v_{j+2}, U_{j+1})| \geq \Delta$, 
and so 
similarly to~(\ref{eqn:3.31.2019.6:57p}), we have   
\begin{equation}
\label{eqn:4.1.2019.12:11p}  
\big|
E_{G, \neg \alpha}^{\spec}\big(N_G^{\typ}(u_{j+2}, U_{j+1}), 
N_G^{\typ}\big(u_{j+1}, U_{j+2}^{\good}\big)\big) \big|
\geq \Delta
\big|N_G^{\typ}\big(u_{j+1}, U_{j+2}^{\good}\big)\big|. 
\end{equation}  
Similarly to~(\ref{eqn:4.1.2019.12:05p}),   
we apply~(\ref{eqn:4.1.2019.2:05p})   
and~(\ref{eqn:4.1.2019.12:11p})  
to~(\ref{eqn:4.1.2019.2:02p})    
to infer 
\begin{equation}
\label{eqn:4.1.2019.2:28p}  
\big|
E_{G, \neg \alpha}^{\spec}\big(N_G^{\typ}(u_{j+2}, U_{j+1}), 
N_G^{\typ}\big(u_{j+1}, U_{j+2}^{\good}\big)\big) 
\big|
\geq 
\Delta
\big(
\big|N_G^{\typ}\big(u_{j+1}, U_{j+2}^{\good}\big)\big| 
- \tfrac{3}{10} n\big)  
\stackrel{(\ref{eqn:4.1.2019.12:05p})}{>} 0,   
\end{equation} 
which proves~(\ref{eqn:4.1.2019.11:49a}).  \\

\noindent {\bf Case 2 ($\Delta_j \geq 1$).}  
We again confirm~(\ref{eqn:4.1.2019.11:49a}), in which case any $\{v_{j+1}, v_{j+2}\} \in E(G)$  
of~(\ref{eqn:4.1.2019.11:49a})
gives a strong rainbow 4-cycle $(u_{j+2}, u_{j+1}, v_{j+2}, v_{j+1})$ which 
Proposition~\ref{prop:Clm13.5}  
extends to a strong rainbow $\ell$-cycle $C_{\ell}$, as promised by Statement~(1)
of Lemma~\ref{lem:ext}.  
We again replay the details of~(\ref{eqn:4.1.2019.2:02p})--(\ref{eqn:4.1.2019.2:28p}), 
only this time we set $\Delta = \Delta_j \geq 1$ (as opposed to before, when $\Delta = 1 + \Delta_j$).    
Then~(\ref{eqn:4.1.2019.2:05p}) is updated to say that   
\begin{multline}  
\label{eqn:4.1.2019.3:06p}  
\big|E_{G, \neg \alpha}^{\spec}\big(U_{j+1} \setminus N_G^{\typ}(u_{j+2}, U_{j+1}),  
N_G^{\typ}\big(u_{j+1}, U_{j+2}^{\good}\big) \big) \big|  \\
\leq 
\big|E_G^{\spec}\big(U_{j+1} \setminus N_G^{\typ}(u_{j+2}, U_{j+1}),  
N_G^{\typ}\big(u_{j+1}, U_{j+2}^{\good}\big) \big) \big| 
\stackrel{(\ref{eqn:3.31.2019.6:18p})}{\leq}
\tfrac{3}{10} \Delta_j n,      
\end{multline}  
while~(\ref{eqn:4.1.2019.11:42a})   
is updated to say that each $v_{j+2} \in U_{j+2}^{\good}$ satisfies 
$$
\big|E_G^{\spec}(v_{j+2}, U_{j+1})\big| 
\geq \deg_G^c(v_{j+2}) - 1 - \big|\hat{U}_j\big|
\stackrel{(\ref{eqn:theo23})}{\geq}    
m + 1 - \big|\hat{U}_j\big|  
= 1 + \Delta_j.
$$
Consequently, (\ref{eqn:4.1.2019.12:11p}) is updated to say   
\begin{equation}
\label{eqn:4.1.2019.3:11p}  
\big|
E_{G, \neg \alpha}^{\spec}\big(N_G^{\typ}(u_{j+2}, U_{j+1}), 
N_G^{\typ}\big(u_{j+1}, U_{j+2}^{\good}\big)\big) \big|
\geq \Delta_j   
\big|N_G^{\typ}\big(u_{j+1}, U_{j+2}^{\good}\big)\big|. 
\end{equation}  
We apply~(\ref{eqn:4.1.2019.3:06p})  
and~(\ref{eqn:4.1.2019.3:11p})  
to~(\ref{eqn:4.1.2019.2:02p})    
to infer 
$$
\big|
E_{G, \neg \alpha}^{\spec}\big(N_G^{\typ}(u_{j+2}, U_{j+1}), 
N_G^{\typ}\big(u_{j+1}, U_{j+2}^{\good}\big)\big) 
\big|  
\geq 
\Delta_j
\big(
\big|N_G^{\typ}\big(u_{j+1}, U_{j+2}^{\good}\big)\big| 
- \tfrac{3}{10} n\big)  
\stackrel{(\ref{eqn:4.1.2019.12:05p})}{>} 0,   
$$
where we used $\Delta_j \geq 1$ from Case~2.  
This confirms~(\ref{eqn:4.1.2019.11:49a}), 
and concludes our proof of Lemma~\ref{lem:ext}.

\section*{Appendix:  Proof-sketch for Case~3 from the Introduction}  
Recall the partition $V = U_0 \cup U_1 \cup U_2 \cup \{x\} \cup \{y\}$ from
Case~3 of Section~\ref{sharpness}.  
Let $\vec{G}_1 = (V, \vec{E}_1)$ be the oriented graph whereby 
$(u, v) \in \vec{E}_1$ if, and only if, $(u, v)$ is an element of one of the following sets:  
$$
U_0 \times U_1, \quad 
U_1 \times U_2, \quad 
U_2 \times U_0, \quad 
(U_0 \cup U_2) \times \{x\}, \quad 
\{x\} \times (\{y\} \cup U_1), \quad 
\{y\} \times U_2.  
$$
Let $\vec{G}_2 = (V, \vec{E}_2)$ satisfy $(u, v) \in \vec{E}_2$ if, and only if, 
$(u, v)$ is an element of one of the following sets:
$$
U_0 \times U_1, \quad 
U_1 \times U_2, \quad 
U_2 \times U_0, \quad 
\{x\} \times U_1, \quad 
U_1 \times \{y\}, \quad 
\{y\} \times (\{x\} \cup U_2).  
$$
One can see that neither $\vec{G}_1 = (V, \vec{E}_1)$ nor $\vec{G}_2 = (V, \vec{E}_2)$
admit directed $\ell$-cycles $\vec{C}_{\ell}$ when $\ell \equiv 2$ (mod 3).  
Recalling $(\hat{G}, \hat{c})$ 
from~\eqref{eqn:hardcol}, 
construct 
$\hat{G}_1 \subseteq \hat{G}$ by removing 
all edges between $U_1$ and $y$, and set $\hat{c}_1 = \hat{c}|_{E(\hat{G}_1)}$.  
Construct $\hat{G}_2 \subseteq \hat{G}$ by removing all edges between $U_0 \cup U_2$ and $x$,
and set $\hat{c}_2 = \hat{c}|_{E(\hat{G}_2)}$.
Then $E(\hat{G}) = E(\hat{G}_1) \cup E(\hat{G}_2)$,  
and $(\hat{G}_i, \hat{c}_i)$, $i = 1, 2$, is isomorphic to the edge-colored graph determined by $\vec{G}_i$.   
Therefore, a rainbow $\ell$-cycle of $(\hat{G}, \hat{c})$ can coincide 
entirely 
with 
neither 
$\hat{G}_1$ nor $\hat{G}_2$,
and 
must admit an edge from $E(\hat{G}_1) \setminus E(\hat{G}_2)$ and an edge from $E(\hat{G}_2) \setminus E(\hat{G}_1)$.
But this is impossible, because~\eqref{eqn:hardcol} ensures
that every edge in the symmetric difference $E(\hat{G}_1) \triangle E(\hat{G}_2)$ is assigned
the color $\star$.

\begin{figure}
  \centering
  \begin{subfigure}{.4 \textwidth}
    \centering
      \begin{minipage}[t][42mm]{\textwidth}
        \noindent
        \begin{tikzpicture}
          \node [draw,ellipse,minimum height=20mm, minimum width=10mm, label=above left:{$U_0$}] (X1) {};
          \node [draw,ellipse,minimum height=20mm, minimum width=10mm, label=above left:{$U_1$},right=10mm of X1] (X2) {};
          \node [draw,ellipse,minimum height=20mm, minimum width=10mm, label=above left:{$U_2$},right=10mm of X2] (X3) {};
          \node [vertex, below=5mm of X1,label=left:{$x$}] (nm1) {};
          \node [vertex, below=5mm of X2,label=right:{$y$}]  (n) {};
          \draw [ultra thick,->] (X1) to (X2);
          \draw [ultra thick,->] (X2) to (X3);
          \draw [ultra thick,->,bend right] (X3.north) to (X1.north);
          \draw [ultra thick,->] (X1) to (nm1);
          \draw [ultra thick,->,bend left=40] (X3.south) to (nm1);
          \draw [ultra thick,->] (nm1) to (X2);
          \draw [ultra thick,->] (n) to (X3);
          \draw [thick, ->] (nm1) to (n);
        \end{tikzpicture} 
      \end{minipage}
      \caption{The oriented graph $\vec{G}_1 = (V, \vec{E}_1)$.}
    \label{fig:figG1}
  \end{subfigure}
  \begin{subfigure}{.4 \textwidth}
    \centering
      \begin{minipage}[t][42mm]{\textwidth}
        \noindent
        \begin{tikzpicture}
          \node [draw,ellipse,minimum height=20mm, minimum width=10mm, label=above left:{$U_0$}] (X1) {};
          \node [draw,ellipse,minimum height=20mm, minimum width=10mm, label=above left:{$U_1$},right=10mm of X1] (X2) {};
          \node [draw,ellipse,minimum height=20mm, minimum width=10mm, label=above left:{$U_2$},right=10mm of X2] (X3) {};
          \node [vertex, below=5mm of X2,label=right:{$y$}] (nm1) {};
          \node [vertex, below=5mm of X1,label=left:{$x$}]  (n) {};
          \draw [ultra thick,->] (X1) to (X2);
          \draw [ultra thick,->] (X2) to (X3);
          \draw [ultra thick,->,bend right] (X3.north) to (X1.north);
          \draw [ultra thick,->] (X2) to (nm1);
          \draw [ultra thick,->] (nm1) to (X3);
          \draw [ultra thick,->] (n) to (X2);
          \draw [thick, ->] (nm1) to (n);
        \end{tikzpicture} 
      \end{minipage}
      \caption{The oriented graph $\vec{G}_2 = (V, \vec{E}_2)$.}
    \label{fig:figG2}
  \end{subfigure}
  \caption{Neither $\vec{G}_1$ nor $\vec{G}_2$ admit directed $\ell$-cycles when 
  $\ell \equiv 2 \pmod 3$.}
\end{figure}
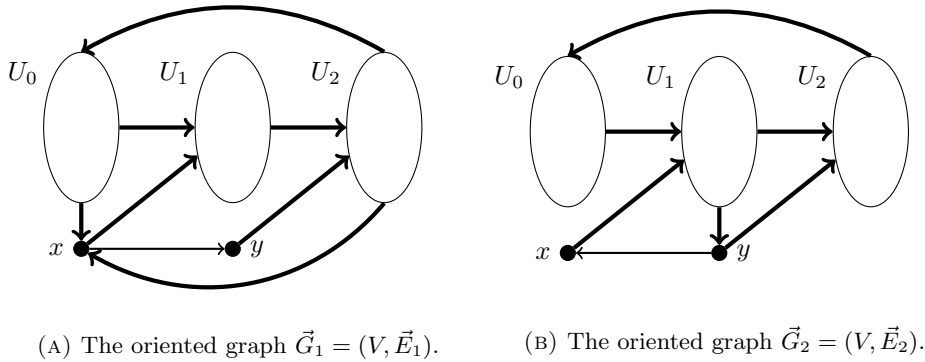


\begin{thebibliography}{10}

\bibitem{alon2017random}
Alon, N., Pokrovskiy, A., and Sudakov, B.~(2017), ``Random 
subgraphs of properly
  edge-coloured complete graphs and long rainbow cycles", {\it 
Israel J.~Math.}~\textbf{222}, no.~1, 317--331.

\bibitem{alon2003testing}
Alon, N., and Shapira, A.~(2003), ``Testing subgraphs in directed graphs",
  {\it Proceedings of the thirty-fifth annual 
ACM symposium on Theory of Computing},
  ACM, 700--709.

\bibitem{caccetta1978minimal}
Caccetta, L., and H\"aggkvist, R.~(1978), ``On 
minimal digraphs with given girth",
{\it Congress.~Numer.}~{\bf 21}, 181--187.  

\bibitem{vcada2016rainbow}
{\v{C}}ada, R., Kaneko, A., Ryj{\'a}{\v{c}}ek, Z., and Yoshimoto, K.~(2016), 
``Rainbow
  cycles in edge-colored graphs", {\it Discrete Math.}~\textbf{339},
  no.~4, 1387--1392.

\bibitem{oddcyc}
  Czygrinow, A., Molla, T., Nagle, B., Oursler, R., 
  ``On Odd Rainbow Cycles in Edge-Colored Graphs", submitted

\bibitem{ErdosStone}
Erd\H{o}s, P., Stone, A.H.~(1946), ``On the structure of linear graphs", 
{\it Bulletin of the American Mathematical Society}~{\bf 52} (12), 1087--1091.

\bibitem{glock2018rainbow}
Glock, S., and Joos, F.~(2018), ``A rainbow blow-up lemma", arXiv preprint
  arXiv:1802.07700 

\bibitem{hladky2017counting}
Hladk{\`y}, J., Kr{\'a}l', D., and Norin, S.~(2017), ``Counting flags in
  triangle-free digraphs", {\it Combinatorica} \textbf{37}, no.~1, 49--76.

\bibitem{ji2018short}
Ji, Y., Wu., S., Song, H.~(2018), ``On short cycles in triangle-free 
oriented
  graphs", {\it Czechoslovak Math.~J.}~\textbf{68}, no.~1, 67--75.

\bibitem{keevash2007rainbow}
Keevash, P., Mubayi, D., Sudakov, B., and Verstra{\"e}te, J.~(2007), 
``Rainbow
  Tur{\'a}n problems", {\it Combin.~Probab.~Comput.}~\textbf{16}, 
no.~1, 109--126.

\bibitem{kelly2008dirac}
Kelly, L, K{\"u}hn, D., and Osthus, D.~(2008), ``A {D}irac-type 
result on Hamilton
  cycles in oriented graphs", {\it Combin.~Probab.~Comput.}~\textbf{17}, 
  no.~5, 689--709.

\bibitem{kelly2010cycles}
Kelly, L., K{\"u}hn, D., and Osthus, D.~(2010), ``Cycles 
of given length in oriented
  graphs, {\it J.~Combin.~Theory Ser.~B} \textbf{100}, no.~3, 251--264.

\bibitem{kim2018rainbow}
Kim, J., K{\"u}hn, D., Kupavskii, A., and Osthus, D.~(2018), 
``Rainbow structures in
  locally bounded colourings of graphs", arXiv preprint arXiv:1805.08424

\bibitem{kuhn2013embedding}
K{\"u}hn, D., Osthus, D., and Piguet, D.~(2013), ``Embedding 
cycles of given length
  in oriented graphs, {\it European J.~Combin.}~\textbf{34}, no.~2, 495--501.

\bibitem{li2014rainbow}
Li, B., Ning, B., Xu, C., and Zhang, S.~(2014), ``Rainbow 
triangles in edge-colored
  graphs, {\it European J.~Combin.}~\textbf{36}, 453--459.

\bibitem{li2013rainbow}
Li, H.~(2013), ``Rainbow $C_3$'s and $C_4$'s in edge-colored graphs",
 {\it Discrete Math.}~\textbf{313}, no.~19, 1893--1896.
\end{thebibliography}
\end{document}